\documentclass[reqno]{amsart}
\usepackage[left=1in,right=1in,top=1in,bottom=1in]{geometry}
\usepackage{tikz}
\usetikzlibrary{shapes,decorations,calc,arrows}
\usepackage[foot]{amsaddr}
\usepackage{amsthm}
\usepackage{amscd}
\usepackage{amsfonts}
\usepackage{amsmath}
\usepackage{amssymb}
\usepackage{mathrsfs}
\usepackage{multirow}
\usepackage{verbatim}
\usepackage{url}
\usepackage[hidelinks]{hyperref}
\usepackage{graphicx}
\usepackage{cite}
\usepackage{fancyhdr}
\usepackage{yfonts}
\usepackage{setspace}
\usepackage{titlesec}
\usepackage{enumitem}

\usepackage{ dsfont }

\titleformat{\section}[hang]
{\normalfont\Large\bfseries}
{\thesection.}{0.5em}{}

\titlespacing*{\section}{0pc}{2pc}{0.25pc}

\titleformat{\subsection}[runin]
{\normalfont\large\bfseries}
{\thesubsection}{0.5em}{}

\titlespacing{\subsection}{0pc}{1.5pc}{0.5pc}

\newcommand{\Aut}{\text{Aut}}

\newcommand{\N}{\mathbb{N}}
\newcommand{\Z}{\mathbb{Z}}
\newcommand{\Q}{\mathbb{Q}}
\newcommand{\R}{\mathbb{R}}
\newcommand{\C}{\mathbb{C}}
\renewcommand{\H}{\mathcal{H}}

\newcommand{\K}{\mathcal{K}}

\newcommand{\vphi}{\varphi}

\newcommand{\eig}{\text{eig}}

\newcommand{\Tr}{\text{Tr}}

\newcommand{\<}{\left\langle}
\renewcommand{\>}{\right\rangle}

\newcommand{\dom}{\text{dom}}

\newcommand{\mc}[1]{\mathcal{#1}}

\newcommand{\Ad}[1]{\text{Ad}\left(#1\right)}

\newcommand{\E}{\mathcal{E}}
\newcommand{\Ind}[1]{\operatorname{Ind}#1}
\renewcommand{\S}{\operatorname{S}}
\newcommand{\Sd}{\operatorname{Sd}}
\newcommand{\Sp}{\operatorname{Sp}}

\newtheorem{thm}{Theorem}[section]
\newtheorem{thmalpha}{Theorem}

\newtheorem{prop}[thm]{Proposition}
\newtheorem{lem}[thm]{Lemma}
\newtheorem*{lem*}{Lemma}
\newtheorem{cor}[thm]{Corollary}

\theoremstyle{definition}
\newtheorem{defi}[thm]{Definition}
\newtheorem{ex}[thm]{Example}

\newtheorem{rem}[thm]{Remark}

\title{Murray--von Neumann Dimension for Strictly Semifinite Weights}
\author{Aldo Garcia Guinto$^\circ$}
\address{$^\circ$Department of Mathematics, Michigan State University\hfill \url{garci575@msu.edu}}
\author{Matthew Lorentz$^\Delta$}
\address{$^\Delta$Department of Mathematics, Michigan State University\hfill \url{lorentzm@msu.edu}}
\author{Brent Nelson$^\square$}
\address{$^\square$Department of Mathematics, Michigan State University \hfill \url{brent@math.msu.edu}}
\date{}

\begin{document}

\maketitle

\begin{abstract}
Given a von Neumann algebra $M$ equipped with a faithful normal strictly semifinite weight $\varphi$, we develop a notion of Murray--von Neumann dimension over $(M,\varphi)$ that is defined for modules over the basic construction associated to the inclusion $M^\vphi \subset M$. For $\varphi=\tau$ a faithful normal tracial state, this recovers the usual Murray--von Neumann dimension for finite von Neumann algebras. If $M$ is either a type $\mathrm{III}_\lambda$ factor with $0<\lambda <1$ or a full type $\mathrm{III}_1$ factor with $\Sd(M)\neq \R$, then amongst extremal almost periodic weights the dimension function depends on $\varphi$ only up to scaling. As an application, we show that if an inclusion of diffuse factors with separable preduals $N\subset M$ is with expectation $\E$ and admits a compatible extremal almost periodic state $\vphi$, then this dimension quantity bounds the index $\Ind{\E}$, and in fact equals it when the modular operators $\Delta_\vphi$ and $\Delta_{\vphi|_N}$ have the same point spectrum. In the pursuit of this result, we also show such inclusions always admit Pimsner--Popa orthogonal bases.
\end{abstract}


\section*{Introduction}

The notion of Murray--von Neumann dimension is as old as the theory of von Neumann algebras itself, having been introduced by the progenitors of the field in \cite{MvN36}. In modern terms, given a von Neumann algebra $M$ equipped with a faithful normal tracial state $\tau$, any Hilbert space $M$-module $\H$ can be identified with an $(M\otimes \C)$-submodule of $L^2(M,\tau)\otimes \K$ (for some auxiliary Hilbert $\K$) in such a way that the projection $p$ onto it lies in $(M\otimes \C)' = (J_\tau M J_\tau) \bar\otimes B(\K)$. The Murray--von Neumann dimension of $\H$ is then defined as
    \[
        \dim_{(M,\tau)}\H:=(\tau\otimes \Tr)[ (J_\tau\otimes 1) p (J_\tau\otimes 1)].
    \]
Notably, this dimension being an invariant of the isomorphism class of the $M$-module $\H$ is a consequence of $\tau\otimes \Tr$ being tracial. Murray--von Neumann dimension is ubiquitous in the von Neumann algebra literature and has been a vital tool for studying the structure of finite von Neumann algebras by way of defining invariants. Highlights of this include (but are certainly not limited to):  $L^2$-Betti numbers for groups due to Atiyah \cite{Ati76} and Singer \cite{Sin77} (see also \cite{CG86});  indices for inclusions of type $\mathrm{II}_1$ factors due to Jones' \cite{Jon83}; and $L^2$-Betti numbers for finite von Neumann algebras due to Connes and Shlyakhtenko \cite{CS05}. Additionally, the dimension function has been extended in a number of ways: to abstract $M$-modules (i.e. those without any Hilbert space structure) by L\"{u}ck \cite{Luc98}; and to semifinite von Neumann algebras $M$ equipped with faithful normal semifinite tracial weights by Petersen \cite{Pet13}. These extensions were used in \cite{KPV15} to provide a notion of $L^2$-Betti numbers for unimodular locally compact groups.

In this article, we offer a further extension of Murray--von Neumann dimension to von Neumann algebras admitting faithful normal \emph{strictly semifinite} weights, where one says a weight $\varphi$ on $M$ is strictly semifinite if its restriction to the centralizer subalgebra $M^\vphi$ is semifinite (see Section~\ref{sec:strictly_semifinite_weights}). Thus this class includes all $\sigma$-finite von Neumann algebras and all semifinite von Neumann algebras. The key idea is to consider the basic construction associated to the inclusion $M^\vphi \subset M$, which necessarily admits a faithful normal conditional expectation (see Lemma~\ref{lem:strictly_semifinite_characterization}). Since this basic construction, denoted $\<M,e_\vphi\>$, is the commutant of $J_\vphi M^\vphi J_\vphi$ in $B(L^2(M,\vphi))$, it follows that any Hilbert space $\<M,e_\vphi\>$-module can be identified with $p\cdot L^2(M,\vphi)\otimes \K$ for some $p\in (J_\vphi M^\vphi J_\vphi)\bar\otimes B(\K)$, and moreover one can show that the quantity
    \[
        (\vphi\otimes \Tr)\left[(J_\vphi\otimes 1) p (J_\vphi\otimes 1)\right]
    \]
is an invariant of the isomorphism class of the $\<M,e_\vphi\>$-module since $\vphi$ is tracial on $M^\vphi$. We thus say a Hilbert space $\H$ is an $(M,\vphi)$-module if it is an $\<M,e_\vphi\>$-module and define its Murray--von Neumann dimension $\dim_{(M,\vphi)}\H$ as the above quantity (see Section~\ref{sec:definitions_and_properties}).

In addition to ensuring the dimension is constant over isomorphism classes of modules, the basic construction $\<M,e_\vphi\>$ is also natural to consider in many cases. For example, it simply equals $M$ when $\vphi$ is tracial, while for \emph{extremal almost periodic} weights it is isomorphic to the associated discrete core (see Section~\ref{sec:discrete_cores} and Lemma~\ref{lem:factor_levels}). Here a faithful normal semifinite weight $\vphi$ is said to be extremal if the centralizer $M^\vphi$ is a factor, and $\vphi$ is said to be almost periodic if the modular operator $\Delta_\vphi$ is diagonalizable (see Section~\ref{sec:almost_periodic} and note that the latter condition implies strict semifiniteness by \cite[Proposition 1.1]{Con74}). In particular, for a type $\mathrm{III}_\lambda$ factor $M$ with $0<\lambda<1$, an extremal almost periodic weight $\vphi$ is known as a \emph{generalized trace} (see \cite[Definition 4.3.1]{Con73}) and $\<M,e_\vphi\>$ is isomorphic to the usual discrete core $M\rtimes_{\sigma^\vphi} (\R/t_0\Z)$ where $t_0=\frac{2\pi}{\log{\lambda}}$. In general, $\<M,e_\vphi\>$ is a semifinite von Neumann algebra containing $M$ with two crucial properties: it can be faithfully represented on $L^2(M,\vphi)$ as an algebra containing $M$; and it admits a unique faithful normal semifinite tracial weight $\tau$ satisfying $\tau(e_\vphi x e_\vphi) = \vphi(x)$ for all $x\in M_+$ (see Proposition~\ref{prop:basic_construction}).

$(M,\vphi)$-modules and the associated Murray--von Neumann dimension function exhibit the expected properties under equivariant maps and direct sums, and additionally behave well under compressions, amplifications, and actions by locally compact groups (see Section~\ref{sec:definitions_and_properties}). They also extend Petersen's theory in the case of Hilbert space modules over semifinite von Neumann algebras (see Proposition~\ref{prop:semifinite_dimension_achived_by_projection}). Conversely, the Murray--von Neumann dimension of an $(M,\vphi)$-module can also be computed by applying Petersen's theory to the unique tracial weight $\tau$ discussed above (see Theorem~\ref{thm:dimension_amplification_compression_formula}).

One significant departure from the case of tracial weights is the following: $(M,\vphi)$-modules need not be $(M,\psi)$-modules for distinct faithful normal strictly semifinite weights $\vphi$ and $\psi$, unless there is some compatibility between the basic constructions, such as, for example, a normal unital $*$-homomorphism $\theta\colon \<M,e_\psi\>\to \<M,e_\vphi\>$. Even amongst weights admitting such compatibility, Murray--von Neumann dimension as a function on $(M,\vphi)$-modules can still vary wildly from one weight to another. On the other hand, this variation still occurs amongst tracial weights---unless one considers a semifinite factor where the variation is limited to a scaling constant. Of course this is due to the uniqueness up to scaling of tracial weights on semifinite factors, and Connes established the analogue of this for type $\mathrm{III}_{\lambda}$ factors with $0<\lambda <1$ and full factors $M$ with $\Sd(M)\neq \R_+$ (see \cite[Theorem 4.3.2]{Con73} and \cite[Theorem 4.7]{Con74}, respectively). He showed that (infinite) extremal almost periodic weights are unique up to scaling \emph{and} inner automorphisms. Using the uniqueness of such weights, we can show the following:

\begin{thmalpha}[{Theorem~\ref{thm:dim_scaling_constant}}]\label{thmalpha:theorem_A}
Let $M$ be a factor with separable predual, and let $\vphi$ and $\psi$ be extremal almost periodic weights. Suppose that either:
    \begin{enumerate}[label=(\alph*)]
    \item $M$ is type $\mathrm{III}_\lambda$ with $0<\lambda <1$; or
    \item $M$ is full.
    \end{enumerate}
Then there exists a $*$-isomorphism $\theta\colon \<M,e_\psi\> \to \<M,e_\vphi\>$ and a constant $c_{\psi,\vphi}>0$ so that for all left $(M,\vphi)$-modules $(\H,\pi)$ one has
    \[
        \dim_{(M,\psi)}(\H, \pi\circ \theta) = c_{\psi,\vphi} \dim_{(M,\vphi)}(\H, \pi).
    \]
\end{thmalpha}

Connes' uniqueness theorems imply $M^\vphi$ and $M^\psi$ are stably isomorphic (and just isomorphic if both weights are infinite), and so the existence of the $*$-isomorphism $\theta$ in the above theorem follows from identifying $\<M,e_\vphi\>$ and $\<M,e_\psi\>$ with amplifications of $M^\vphi$ and $M^\psi$, respectively (see Lemma~\ref{lem:factor_levels}). Then the constant $c_{\psi,\vphi}$ is obtained by comparing the unique tracial weights on $\<M,e_\vphi\>$ and $\<M,e_\psi\>$ that recover $\vphi$ and $\psi$, respectively (see Proposition~\ref{prop:basic_construction}.\ref{part:tracial_weight}).

As an application, we explore the connection between the index in subfactor theory and our extension of Murray--von Neumann dimension. For inclusions of type $\mathrm{II}_1$ factors $N\subset M$, the connection is explicit and has been known since the index was first defined in \cite{Jon83}: $[M:N]=\dim_N L^2(M)$. For more general inclusions of factors, the index depends on a choice of faithful normal conditional expectation $\E\colon M\to N$ and can be defined in several equivalent ways: see \cite{PP86} for a probabilistic definition; see \cite{Kos86} for a definition via spatial derivatives; and see \cite{Lon89} for a definition via the continuous cores. For inclusions $N\subset M$ admitting an extremal almost periodic state, we show that the index can be bounded by the Murray--von Neumann dimension of $L^2(M,\vphi\circ\E)$ over $(N,\vphi)$, and in certain cases the quantities are equal. Note that, unlike in the tracial case, $L^2(M,\vphi\circ \E)$ is not canonically an $(N,\vphi)$-module, so it is necessary to first exhibit a normal unital $*$-homomorphism from $\<N,e_\vphi\>$ to $\<M,e_{\vphi\circ\E}\>$.

\begin{thmalpha}[{Theorem~\ref{thm:dimension_vs_index}}]\label{thmalpha:theorem_B}
Let $N\overset{\E}{\subset} M$ be an inclusion of diffuse factors with separable preduals, and let $\vphi$ be an extremal almost periodic state on $N$ such that $\vphi\circ \E$ is also an extremal almost periodic state. Let $C\subset \text{Sd}(\vphi\circ \E)$ be a transversal of coset representatives for $\text{Sd}(\vphi)$ and let $\pi_C \colon \<N,e_\vphi\> \to \<M,e_{\vphi\circ \E}\>$ be the unique representation satisfying $\pi_C|_N=\text{id}$ and
    \[
        \pi_C(e_\vphi) = \sum_{\mu\in C} 1_{\{\mu\}}(\Delta_{\vphi\circ \E})
    \]
(see Theorem~\ref{thm:subfactors_give_modules}). Then one has
    \begin{align}\label{ineq:intro_dimension_index_relation}
        \inf(C) \Ind{\E} \leq \dim_{(N,\vphi)}(L^2(M,\vphi\circ \E),\pi_C) \leq \sup(C) \Ind{\E}.
    \end{align}
Moreover, the following are equivalent:
    \begin{enumerate}[label=(\roman*)]
    \item $\Ind{\E}<\infty$;
    \item $[\text{Sd}(\vphi\circ \E)\colon \text{Sd}(\vphi)]<\infty$ and $\dim_{(N,\vphi)}(L^2(M,\vphi\circ \E), \pi_C)<\infty$;
    \item $[\text{Sd}(\vphi\circ \E)\colon \text{Sd}(\vphi)]<\infty$ and $[M^{\vphi\circ \E}\colon N^\vphi]<\infty$.
    \end{enumerate}
\end{thmalpha}

Observe that when $\Sd(\vphi)=\Sd(\vphi\circ \E)$, one can choose $C=\{1\}$ so that (\ref{ineq:intro_dimension_index_relation}) reduces to an equality. Alternatively, if $\Sd(\vphi)$ is dense in $\R_+$, then the bounds in (\ref{ineq:intro_dimension_index_relation}) can be made arbitrarily tight by choosing an appropriate transversal $C$ (see Remark~\ref{rem:tight_bounds}). On the way to proving the above theorem, we establish that such inclusions always admit Pimsner--Popa orthogonal bases, regardless of whether or not the index is finite (see Proposition~\ref{prop:existence_PP_basis}).

The structure of the paper is as follows. In Section~\ref{sec:preliminaries} we provide preliminaries on weights, discrete cores, Connes invariants, and subfactor theory. In Section~\ref{sec:modules_and_dimension} we analyze the basic construction for $M^\vphi\subset M$, define $(M,\vphi)$-modules and their Murray von Neumann dimension, and establish many properties thereof. In Section~\ref{sec:appliction_index_subfactors} we give our application to the index for subfactors. Finally, in Section~\ref{sec:constructing_extremal_almost_periodic_inclusions} we consider ways to construct inclusions of the type appearing in Theorem~\ref{thmalpha:theorem_B}.

\subsection*{Acknowledgments}

We wish to thank the following people for helpful discussions related to this work: Michael Hartglass, Ben Hayes, Cyril Houdayer, David Jekel, David Penneys, Dima Shlyakhtenko, and Makoto Yamashita. We also thank Srivatsav Kunnawalkam Elayavalli for directing us to the proof of the existence of Pimsner--Popa bases for infinite index inclusions of type $\mathrm{II}_1$ factors in \cite[Theorem 5.21]{Bur17}. The third author was supported by NSF grants DMS-1856683 and DMS-2247047.



\section{Preliminaries}\label{sec:preliminaries}

Throughout $M$ denotes a von Neumann algebra with predual $M_*$. By a representation of $M$ we mean a normal unital $*$-homomorphism $\pi\colon M\to B(\H)$ for some Hilbert space $\H$. Given a subset $S\subset \H$, we write $[S]$ for the projection $p\in B(\H)$ onto $\overline{\text{span}}(S)$. We denote the trace on $B(\H)$ by $\Tr$, which is assumed to be normalized so that $\Tr([\K])=\dim(\K)$ for all closed subspaces $\K\leq \H$.

\subsection{Weights}

A \emph{weight} on a von Neumann algebra $M$ is a map $\vphi\colon M_+\to [0,\infty]$ satisfying for $x,y\in M_+$ and $c\geq 0$
    \[
        \vphi(cx+y) = c\vphi(x) + \vphi(y).
    \]
A weight is said to be: \emph{faithful} if $\vphi(x)=0$ for $x\in M_+$ implies $x=0$; \emph{normal} if $\vphi(\sup_i x_i) =\sup_i \vphi(x_i)$ for any bounded increasing net $(x_i)_{i\in I}\subset M_+$; \emph{semifinite} if $\{x\in M_+\colon \vphi(x)<\infty\}''=M$; and \emph{tracial} if $\vphi(x^*x)=\vphi(xx^*)$ for all $x\in M$. Every von Neumann algebra admits a faithful normal semifinite weight (see \cite[Theorem VII.2.7]{Tak03}), and semifinite von Neumann algebras are precisely those that admit a faithful normal semifinite tracial weight.

We denote\footnote{These sets are traditionally denoted $\mathfrak{n}_\vphi$ and $\mathfrak{m}_\vphi$, respectively, but we feel the above notation more clearly indicates their meaning.}
    \begin{align*}
        \sqrt{\dom}(\vphi)&:=\{x\in M\colon \vphi(x^*x)<\infty\}\\
        \dom(\vphi)&:=\text{span}\{x^*y\colon x,y\in \sqrt{\dom}(\vphi)\}.
    \end{align*}
Then $\vphi$ can be extended to a linear functional on $\dom(\vphi)$, and $\vphi$ is semifinite if and only if $\dom(\vphi)''=M$.  

Associated to a faithful normal semifinite weight $\varphi$ is the GNS semi-cyclic representation $(L^2(M,\varphi),\pi_\varphi)$ where $L^2(M,\varphi)$ is the completion of $\sqrt{\dom}(\varphi)$ with respect to the inner product
	\[
		\<x,y\>_\varphi:= \varphi(y^*x) \qquad x,y\in \sqrt{\dom}(\vphi),
	\]
and $\pi_\varphi$ is a faithful representation of $M$ determined by $\pi_\varphi(x)y=xy$ for $x\in M$ and $y\in \sqrt{\dom}(\varphi)$ (note that $\sqrt{\dom}(\vphi)$ is a left ideal in $M$). When $\vphi$ is a state, this is just the usual GNS representation. Since $\pi_\vphi$ is a faithful representation we will identify $M\cong \pi_\vphi(M)$. The map $\sqrt{\dom}(\varphi)\ni x\mapsto x^*$ is a densely defined closable operator on $L^2(M,\varphi)$, whose closure we denote by $S_\varphi$. If $S_\vphi=J_\vphi\Delta_\varphi^{1/2}$ is its polar decomposition, then $J_\vphi$ is called the \emph{modular conjugation} for $\vphi$ and $\Delta_\vphi$ is called the \emph{modular operator} for $\vphi$. The modular conjugation is an isometric involution, the modular operator is a densely defined, nonsingular, positive operator, and they enjoy the following properties:
	\[
		J_\varphi M J_\vphi = M'\cap B(L^2(M,\vphi)) \qquad \text{ and }\qquad  \Delta_\vphi^{it} M \Delta_\vphi^{-it} =M \qquad t\in \R.
	\]
The maps $\sigma_t^\varphi:= \text{Ad}(\Delta_\varphi^{it})\in \Aut(M)$, $t\in \R$, define a strongly continuous action $\R\overset{\sigma^\vphi}{\curvearrowright} M$ called the \emph{modular automorphism group} of $\vphi$. The set of fixed points
	\[
		M^\vphi:= \bigcap_{t\in \R} M^{\sigma_t^\vphi}
	\]
is a von Neumann subalgebra of $M$ called the \emph{centralizer} with respect to $\sigma^\vphi$, and $\vphi$ is tracial on $\dom(\vphi)\cap M^\vphi$. In the case that $\vphi$ itself is tracial, $\Delta_\vphi=1$ and so $\sigma^\vphi$ is trivial and $M^\vphi=M$. Note that $\vphi|_{M^\vphi_+}$ need not be semifinite, so $M^\vphi$ can still fail to be a semifinite von Neumann algebra; in fact, it can even be type $\mathrm{III}$ (see \cite{Haa77}). We discuss the case when $\vphi|_{M^\vphi_+}$ is semifinite---and hence $M^\vphi$ is a semifinite von Neumann algebra---in Section~\ref{sec:strictly_semifinite_weights} below.  After \cite[Definition 4.7]{Dyk97}, we say a weight is \emph{extremal} if $M^\vphi$ is a factor.

\begin{rem}\label{rem:extremal_ap_on_semifinite_is_tracial}
We point out that if a semifinite von Neumann algebra $M$ admits an extremal faithful normal semifinite weight, then $\vphi$ is necessarily tracial. Indeed, in this case $\sigma_t^\vphi(x)= h^{it} x h^{-it}$ for a positive non-singular self-adjoint operator $h$ affiliated with $M^\vphi$ (see the proof of \cite[Theorem VIII.3.14]{Tak03}), and so
    \[
        M^\vphi = \{ h^{it}\colon t\in \R\}'\cap M.
    \]
By Stone's theorem, every spectral projection $p$ of $h$ belongs to $\{h^{it}\colon t\in \R\}''$ and hence commutes with $M^\vphi$ by the above equality. On the other hand, $p\in M^\vphi$ since $h$ is affiliated with $M^\vphi$. So if $M^\vphi$ is a factor, then we must have $p\in \C$ for all spectral projections of $h$, implying $h\in\C$ and that $\vphi$ is tracial.$\hfill\blacksquare$
\end{rem}

A non-zero element $x\in M$ is an \emph{eigenoperator} of $\sigma^\vphi$ if there exists $\lambda>0$ so that $\sigma_t^\vphi(x) = \lambda^{it} x$ for all $t\in \R$. (For $x\in \sqrt{\dom}(\vphi)$, this is equivalent to $x\in L^2(M,\vphi)$ being an eigenvector of $\Delta_\vphi$ with eigenvalue $\lambda$.) For example, every non-zero element of $M^\vphi$ is an eigenoperator with eigenvalue $1$. We denote the set of eigenoperators with eigenvalue $\lambda$ by $M^{(\vphi,\lambda)}$, and denote
	\[
		M^{(\vphi,\eig)}:=\text{span}\left(\bigcup_{\lambda>0} M^{(\vphi,\lambda)} \right).
	\]
Note that $M^{(\vphi,\eig)}$ is a unital $*$-subalgebra since $1\in M^{(\vphi,1)}$, $[M^{(\vphi,\lambda)}]^*= M^{(\vphi,\frac{1}{\lambda})}$, and $xy\in M^{(\vphi,\lambda\mu)}$ for $x\in M^{(\vphi,\lambda)}$ and $y\in M^{(\vphi,\mu)}$. In the case that $\vphi$ is tracial, we have $M^{(\vphi,\eig)} = M^{(\vphi,1)}\cup\{0\} = M^\vphi=M$.

\subsection{Strictly semifinite weights}\label{sec:strictly_semifinite_weights}

A normal semifinite weight $\vphi$ on $M$ is called \emph{strictly semifinite} if there exists a family $\{\omega_i\in (M_*)_+\colon i\in I\}$ with pairwise orthogonal support projections $s(\omega_i)$ satisfying
    \[
        \vphi(x) = \sum_{i\in I}\omega_i(x)\qquad x\in M_+.
    \]
(This is a modern refinement of terminology originally due to Combes: compare \cite[Definition 4.1]{Com71} to \cite[Definition XII.4.6]{Tak03}.) A version of following result appears as \cite[Proposition 4.3]{Com71} (see also \cite[Exercise VIII.2.1]{Tak03}), but we prove it here under slightly different hypotheses.

\begin{lem}\label{lem:strictly_semifinite_characterization}
For a faithful normal semifinite weight $\vphi$ on $M$, the following are equivalent:
    \begin{enumerate}[label=(\roman*)]
    \item $\vphi$ is strictly semifinite;
    
    \item there exists a family $\{p_i\in M^\vphi\cap \dom(\vphi)\colon i\in I\}$ of pairwise orthogonal projections with
        \[
            \sum_{i\in I} p_i =1;
        \]  
        
    \item $\vphi|_{M^\vphi_+}$ is semifinite;
    \item there exists a faithful normal $\vphi$-preserving conditional expectation $\mc{E}_\vphi\colon M\to M^\vphi$.
    \end{enumerate}
\end{lem}
\begin{proof}$(i)\Rightarrow(ii):$ Let $\{\omega_i\in (M_*)_+\colon i\in I\}$ be as in the definition of strictly semifinite and set $p_i$ to be the support projection of $\omega_i$ for each $i\in I$. These projections are pairwise orthogonal by assumption and the faithfulness of $\vphi$ demands they sum to $1$. Note that $p_i\in \dom(\vphi)$ since $\vphi(p_i)=\omega_i(p_i)=\omega_i(1)<\infty$. Also, for $x\in \dom(\vphi)$ we have
    \[
        \vphi(xp_i) = \omega_i(xp_i) = \omega_i(p_i x) = \varphi(p_ix).
    \]
Hence $p_i\in M^\vphi$ by \cite[Theorem VIII.2.6]{Tak03}.\\

\noindent$(ii)\Rightarrow(iii):$ Let $q\in M^\vphi$ be an arbitrary projection. For a finite subset $F\subset I$, we have
    \[
        \vphi( q\left( \sum_{i\in F} p_i \right) q ) = \sum_{i\in F} \vphi(p_i q p_i) \leq \sum_{i\in F} \vphi(p_i) <\infty.
    \]
Thus
    \[
        q\left( \sum_{i\in F} p_i \right) q \in M^\vphi\cap \dom(\vphi),
    \]
and this net converges strongly to $q$.\\

\noindent$(iii)\Leftrightarrow (iv):$ This follows from \cite[Theorem IX.4.2]{Tak03}.\\

\noindent$(iii)\Rightarrow(i):$ Let $\{p_i\in M^\vphi\cap \dom(\vphi)\colon i\in I\}$ be a maximal family of pairwise orthogonal projections and set $\omega_i:= \vphi(p_i\cdot p_i)\in (M_*)_+$. Suppose, towards a contradiction, that
    \[
        p:=\sum_{i\in I} p_i < 1.
    \]
Let $(x_j)_{j\in J} \subset M^\vphi\cap \dom(\vphi)$ be a net converging strongly to $1-p$. Then $x_j(1-p)\to (1-p)$ strongly, and in particular there exists $j\in J$ large enough so that $x_j(1-p)\neq 0$. For such $j\in J$ we have
    \[
        0< \vphi((1-p) x_j^*x_j(1-p)) = \vphi( x_j(1-p)x_j^*) \leq \vphi(x_jx_j^*)  <\infty.
    \]
Since $(1-p)x_j^*x_j(1-p)$ is non-zero and positive, we can find $t>0$ so that $q:=1_{[t,\infty)}((1-p)x_j^*x_j(1-p))$ is non-zero. But then
    \[
        \vphi(q) \leq t\vphi((1-p)x_j^*x_j(1-p)) <\infty,
    \]
and so $q$ contradicts the maximality of $\{p_i\colon i\in I\}$. We must therefore have $\sum_{i\in I} p_i = 1$.

Now, consider
    \[
        N_0:= \text{span}\left(\bigcup_{i\in I} p_i M p_i\right),
    \]
and $N:=N_0''$. Since $N_0\subset \dom(\vphi)$, it follows that $\vphi|_N$ is semifinite, and since $p_i\in M^\vphi$ for all $i\in I$ we also have that $\sigma_t^\vphi(N)=N$ for all $t\in \R$. Thus there exists a faithful normal $\vphi$-preserving conditional expectation $\mathcal{E}_N\colon M\to N$. Observe that
    \[
        \mathcal{E}_N(x) = \mathcal{E}_N\left(\sum_{i\in I} p_i x\right) = \sum_{i\in I} p_i \mathcal{E}_N(x) = \sum_{i\in I} p_i \mathcal{E}_N(x)p_i,
    \]
for all $x\in M$. Since normality of $\vphi$ is equivalent to complete additivity (see \cite[Theorem VII.1.11]{Tak03}), for $x\in M_+$ one then has
    \[
        \vphi(x) = \vphi(\mathcal{E}_N(x)) = \vphi\left( \sum_{i\in I} p_i \mathcal{E}_N(x) p_i\right) = \sum_{i\in I} \vphi(p_i\mathcal{E}_N(x) p_i) = \sum_{i\in I} \vphi(p_i xp_i) = \sum_{i\in I} \omega_i(x).
    \]
Hence $\vphi$ is strictly semifinite.
\end{proof}

Within this article, we will always invoke strict semifiniteness to assert either that $\vphi|_{M^\vphi_+}$ is semifinite or that there exists a faithful normal $\vphi$-preserving conditional expectation $\E_\vphi\colon M\to M^\vphi$. Thus the reader is encouraged to adopt either of these as the practical definition of strict semifiniteness.

\begin{rem}\label{rem:strictly_semifinite_corners}
Suppose $M$ is a von Neumann algebra equipped with a faithful normal strictly semifinite weight $\vphi$. For any non-zero projection $p\in M^\vphi$ the restriction of $\vphi$ to $pMp$ is faithful normal and strictly semifinite. Indeed, faithfulness and normality are immediate. To see strict semifiniteness, first observe that $pM^\vphi p = (pMp)^{\vphi|_{pMp}}$, then note that the traciality of $\vphi$ on $M^\vphi$ implies $\dom(\vphi|_{M^\vphi})$ is a two-sided ideal. Hence $p\,\dom(\vphi|_{M^\vphi})p$ is $\sigma$-weakly dense in $(pMp)^{\vphi|_{pMp}}$ and lies in $\dom(\vphi)$, yielding the strict semifiniteness of $\vphi|_{pMp}$. Note that $pM^\vphi p = (pMp)^{\vphi|_{pMp}}$ also implies that $\vphi|_{pMp}$ is extremal whenever $\vphi$ is extremal.$\hfill\blacksquare$
\end{rem}

\subsection{Almost periodic weights}\label{sec:almost_periodic}

A faithful normal semifinite weight $\vphi$ on a von Neumann algebra $M$ is called \emph{almost periodic} if the modular operator $\Delta_\vphi$ is diagonalizable:
    \[
        \Delta_\vphi = \sum_{\lambda>0} \lambda 1_{\{\lambda\}}(\Delta_\vphi).
    \]
The sum is really over the point spectrum of $\Delta_\vphi$, and if $M$ has separable predual (equivalently if $L^2(M,\vphi)$ is separable as a Hilbert space) then the sum is necessarily over countably many $\lambda$. This terminology was introduced by Connes and refers to the equivalent characterization that the function
    \[
        \R\ni t\mapsto \varphi(\sigma_t^{\vphi}(x)y)
    \]
is almost periodic (i.e. is a uniform limit of periodic functions) for all $x,y\in M$ (see \cite[Lemma 7]{Con72}). Note that by \cite[Proposition 1.1]{Con74}, an almost periodic weight is automatically strictly semifinite. Moreover, the relative commutant $(M^\vphi)'\cap M$ is contained in $(M^\vphi)'\cap M^\vphi$, the center of $M^\vphi$ (see \cite[Theorem 10]{Con72}). In particular, if $\vphi$ is extremal almost periodic then $M$ is necessarily also a factor since $M'\cap M \subset (M^\vphi)'\cap M \subset (M^\vphi)'\cap M^\vphi=\C$.

The following result is likely known to experts, but we present a proof here for the convenience of the reader.

\begin{lem}\label{lem:ap_iff_eigenoperators_generate}
Let $M$ be a von Neumann algebra equipped with a faithful normal strictly semifinite weight $\vphi$. Then $\vphi$ is almost periodic if and only if $(M^{(\vphi,\eig)})''=M$.
\end{lem}
\begin{proof}
If $\vphi$ is almost periodic then $M^{(\vphi,\eig)}$ is $\sigma$-weakly dense by \cite[Lemma 1.2.3]{Dyk95}. Conversely, suppose $(M^{(\vphi,\eig)})''=M$ and let $\H$ denote the closure of
    \[
        M^{(\vphi,\eig)}\cdot \sqrt{\dom}(\vphi|_{M^\vphi})
    \]
in $L^2(M,\vphi)$, which we note $\Delta_\vphi$ is diagonalizable over. It suffices to show $\H = L^2(M,\vphi)$, so consider $\xi\in \H^\perp$. For any $x\in M$ and $y\in \sqrt{\dom}(\vphi|_{M^\vphi})$ we have  $\<xy,\xi\>_\vphi=0$ by the strong operator topology density of $M^{(\vphi,\eig)}$ in $M$. Using strict semifiniteness, let $(p_i)_{i\in I} \subset \dom(\vphi|_{M^\vphi})$ be net of projections converging strongly to $1\in M$. Then for all $x\in \sqrt{\dom}(\vphi)$, the previous observation gives
    \[
        \<x,\xi\>_\vphi = \lim_{i\to\infty} \<x, J_\vphi p_i J_\vphi \xi\>_\vphi = \lim_{i\to\infty} \< xp_i, \xi\> = 0.
    \]
Thus $\xi=0$ by the density of $\sqrt{\dom}(\vphi)\subset L^2(M,\vphi)$, and hence $\H=L^2(M,\vphi)$.
\end{proof}

\begin{rem}\label{rem:almost_periodic_corners}
Suppose $M$ is a von Neumann algebra admitting an almost periodic weight $\vphi$. For any non-zero projection $p\in M^\vphi$ one has that the restriction of $\vphi$ to $pMp$ is almost periodic and is extremal whenever $\vphi$ is extremal. Indeed, $\vphi|_{pMp}$ is strictly semifinite by Remark~\ref{rem:strictly_semifinite_corners}, and  $(pMp)^{(\vphi|_{pMp},\eig)}$ is $\sigma$-weakly dense in $pMp$ since it contains $p M^{(\vphi,\eig)}p$. Hence $\vphi$ is almost periodic by Lemma~\ref{lem:ap_iff_eigenoperators_generate}, and furthermore is extremal whenever $\vphi$ is extremal by Remark~\ref{rem:strictly_semifinite_corners}. It follows that any type $\mathrm{III}$ factor $M$ admitting an (extremal) almost periodic weight $\vphi$ also admits an (extremal) almost periodic state: as $M$ is purely infinite we can identify $M$ with $pMp$ for any non-zero projection $p\in \dom(\vphi|_{M^\vphi})$ and $\phi(x):= \frac{1}{\vphi(p)}\vphi(pxp)$ is an (extremal) 
almost periodic state on $pMp$ by the above.$\hfill\blacksquare$
\end{rem}

\subsection{Discrete cores}\label{sec:discrete_cores}

Let $\vphi$ be an almost periodic weight on a von Neumann algebra $M$ and suppose $\Lambda \leq \R_+$ is a subgroup containing the point spectrum of $\Delta_\vphi$. Equipping $\Lambda$ with the discrete topology, let $G:=\widehat{\Lambda}$ be its compact Pontryagin dual group. Note that the \emph{transpose} of the inclusion $\iota\colon \Lambda \hookrightarrow \R_+$ is a continuous group homomorphism $\hat{\iota}\colon \R\to G$ defined by dual pairings as follows:
    \[
        (\hat\iota(t) \mid  \lambda ) = (t \mid \iota(\lambda)) \qquad t\in \R,\ \lambda\in \Lambda.
    \]
Here we are identifying $\R\cong \widehat{\R_+}$ via $( t \mid a) = a^{it}$ for $t\in \R$ and $a\in \R_+$, so the above can also be interpreted as saying $\hat\iota(t)$ is the character on $\Lambda$ defined by $[\hat\iota(t)](\lambda)= \lambda^{it}$. The injectivity of $\iota$ implies $\hat\iota(\R)$ is dense in $G$, and if $\Lambda$ happens to be dense in $\R_+$ then $\hat\iota$ will be injective. For each $s\in G$, define
    \[
        U_s:=\sum_{\lambda \in \Lambda} (s\mid \lambda) 1_{\{\lambda\}}(\Delta_\vphi),
    \]
and for $x\in M$ define
    \[
        \sigma_s^{(\vphi,\Lambda)}(x):= U_s x U_s^*.
    \]
It follows that $G\overset{\sigma^{(\vphi,\Lambda)}}{\curvearrowright} M$ is continuous action which satisfies $\sigma_{\hat{\iota}(t)}^{(\vphi,\Lambda)} = \sigma_t^{\vphi}$ and $M^{\sigma^{(\vphi,\Lambda)}} = M^\vphi$. Additionally, for $x\in M^{(\vphi,\lambda)}$ one has $\sigma_s^{(\vphi,\Lambda)}(x) = (s|\lambda) x$ for all $s\in G$. The crossed product
    \[
        M\rtimes_{\sigma^{(\vphi,\Lambda)}} G
    \]
is called a \emph{discrete core} of $M$. When $\Lambda$ is the group generated by the point spectrum of $\Delta_\vphi$, we will call $\sigma^{(\vphi,\Lambda)}$ the \emph{point modular extension} of $\R\overset{\sigma^\vphi}{\curvearrowright}M$ to $G$. We refer the reader to \cite{Dyk95} for additional details and proofs of the above facts.

\subsection{Connes invariants}

In this section we will recall three invariants associated to factors that are each due to Connes. These will be useful to us for determining the extremality of an almost periodic weight. Throughout this section, let $M$ denote a factor.

For a faithful normal semifinite weight $\vphi$ on $M$, let $\S(\vphi)$ denote the spectrum of its modular operator $\Delta_\vphi$. Then the \emph{modular spectrum} of $M$ is the quantity
    \[
        \S(M):= \bigcap \left\{\S(\vphi)\colon \vphi \text{ is a faithful normal semifinite weight on }M \right\}.
    \]
This invariant can take on the following values:
    \[
        \S(M)=\begin{cases}
            \{1\} & \text{if and only if $M$ is semifinite}\\
            \{0,1\} & \text{if and only if $M$ is type $\mathrm{III}_0$}\\
            \{0\}\cup \lambda^{\Z} & \text{if and only if $M$ is type $\mathrm{III}_\lambda$ for $0<\lambda < 1$}\\
            [0,\infty) & \text{if and only if $M$ is type  $\mathrm{III}_1$}
        \end{cases}
    \]
(see \cite[Lemma 3.1.2]{Con73} and \cite[Theorem XII.1.6]{Tak03}).

For an almost periodic weight $\vphi$ on $M$, let $\Sd(\vphi)$ denote the point spectrum of its modular operator $\Delta_\vphi$. Then the \emph{point modular spectrum} of $M$ is the quantity
    \[
        \Sd(M):=\bigcap \left\{\Sd(\vphi)\colon \vphi \text{ is an almost periodic weight on }M \right\},
    \]
where we set $\Sd(M):=\R_+$ if $M$ does not admit any almost periodic weights (such algebras exist by \cite[Corollary 5.3]{Con74}).
This invariant always satisfies $\Sd(M)\subset \S(M)\setminus \{0\}$ (see \cite[Corollary 1.7]{Con74}), and one has $\overline{\Sd(M)}=\S(M)$ whenever $M$ is full with separable predual (see \cite[Corollary 4.11]{Con74}). The point modular spectrum is always a subgroup of $\R_+$ (see \cite[Theorem 1.3.(b)]{Con74}), and any countable subgroup of $\R_+$ can occur as $\Sd(M)$ for $M$ a full factor with separable predual (see \cite[Corollary 4.4]{Con74}). In contrast to the full case, we note that $\Sd(M)=\{1\}$ whenever $M$ is a hyperfinite factor with separable predual (see \cite[Corollary 1.8]{Con74}).

Lastly, let $G\overset{\alpha}{\curvearrowright} M$ be an action of a locally compact abelian group on $M$. Observe that if $p$ is a projection in the fixed point subalgebra $M^\alpha$, then the action of $G$ restricts to the corner $pMp$. We denote this restricted action by $\alpha^{(p)}$. The \emph{Connes spectrum} of $\alpha$ is the quantity
    \[
        \Gamma(\alpha):= \bigcap \left\{\Sp(\alpha^{(p)})\colon p\in M^\alpha\setminus\{0\} \text{ a projection}\right\}.
    \]
Here $\Sp(\alpha^{(p)})$ denotes the \emph{Arveson spectrum} of $\alpha^{(p)}$, and we refer the reader to \cite[Section XI.1]{Tak03} for further details including the general definition. For our purposes it will be sufficient to consider two specific examples. First, for any faithful normal semifinite weight $\vphi$, the action $\R\overset{\sigma^\vphi}{\curvearrowright} M$ has $\Sp(\sigma^\vphi)=\S(\vphi)\setminus \{0\}$ and $\Gamma(\sigma^\vphi) = \S(M)\setminus\{0\}$ (see \cite[Lemma 3.2.2 and Theorem 3.2.1]{Con73}). Second, if $\vphi$ is an almost periodic weight, $\Lambda\leq \R_+$ is a subgroup containing $\Sd(\vphi)$, and $\sigma^{(\vphi,\Lambda)}$ is the extension of $\R \overset{\sigma^\vphi}{\curvearrowright} M$ to the dual of $\Lambda$ (see Section~\ref{sec:discrete_cores}), then $\Sp(\sigma^{(\vphi,\Lambda)})= \Sd(\vphi)$ (see \cite[Equation (1)]{Con74}). One also has the following, which is essentially \cite[Theorem 1.3.(a)]{Con74} but with enough small changes to warrant a proof.

\begin{lem}\label{lem:point_modular_spectrum_is_intersection_Connes_spectrum}
Let $M$ be a factor with separable predual. For each almost periodic weight $\vphi$ denote by $\Lambda_\vphi$ the countable group generated by $\Sd(\vphi)$ in $\R_+$ and let $\sigma^{(\vphi,\Lambda_\vphi)}$ be the point modular extension of $\R \overset{\sigma^\vphi}{\curvearrowright} M$. Then
    \[
        \Sd(M) = \bigcap \left\{ \Gamma(\sigma^{(\vphi, \Lambda_\vphi)})\colon \vphi \text{ is an almost periodic weight on }M \right\}.
    \]
\end{lem}
\begin{proof}
First note that $M$ having separable predual guarantees each $\Lambda_\vphi$ is countable. Now, by definition of the Connes spectrum we have $\Gamma(\sigma^{(\vphi, \Lambda_\vphi)}) \subset \text{Sp}(\sigma^{(\vphi, \Lambda_\vphi)})$, and as noted above the latter equals $\Sd(\vphi)$. This gives the inclusion `$\supset$'. Conversely, \cite[Proposition 2.3.17]{Con73} implies
    \[
        \Gamma(\sigma^{(\vphi,\Lambda_\vphi)}) = \bigcap \left\{\text{Sp}(\alpha)\colon \alpha\text{ is cocycle equivalent to } \sigma^{(\vphi,\Lambda_\vphi)} \right\}.
    \]
We claim that for each action $\alpha$ cocycle equivalent to $\sigma^{(\vphi,\Lambda_\vphi)}$ there exists an almost periodic weight $\psi$ on $M$ with $\Sd(\psi) = \text{Sp}(\alpha)$. This will imply $\Sd(M) \subset \Gamma(\sigma^{(\vphi,\Lambda_\vphi)})$ and complete the proof. Let $\hat{\iota}$ be the transpose of the inclusion map $\iota\colon \Lambda_\vphi \hookrightarrow \R_+$. Then $\alpha\circ \hat{\iota}$ and $\sigma^{\vphi}$ are cocycle equivalent as actions of $\R$ by \cite[Lemma 3.4.3]{Con73}, and from \cite[Theorem 1.2.4]{Con73} there exists a faithful normal semifinite weight $\psi$ satisfying $\sigma^\psi = \alpha\circ \hat{\iota}$ (see also \cite[Theorem VIII.3.8]{Tak03}). It then follows from \cite[Proposition 1.1.(b) and Equation (1)]{Con74} that $\psi$ is almost periodic with $\Sd(\psi) = \text{Sp}(\alpha)$, as claimed.
\end{proof}

The following provides useful criteria for the extremality of almost periodic weights (see also Lemma~\ref{lem:factor_levels} below).

\begin{lem}\label{lem:extremal_from_Gamma}
Let M be a factor with separable predual equipped with an almost periodic weight $\vphi$, and consider the following statements:
    \begin{enumerate}[label=(\roman*)]
        \item $\Sd(\vphi)= \Sd(M)$;
        \item $\Sd(\vphi)$ is a countable subgroup of $\R_+$ and $\Gamma(\sigma^{(\vphi, \Sd(\vphi))}) = \Sd(\vphi)$;
        \item $M^\vphi$ is a factor.
    \end{enumerate}
Then $(i)\Rightarrow (ii) \Leftrightarrow (iii)$. If $M$ is full, then one also has $(iii)\Rightarrow (i)$.
\end{lem}
\begin{proof}
$(i)\Rightarrow(ii):$ First note that $\Sd(\vphi)=\Sd(M)$ is a subgroup by \cite[Theorem 1.3.(b)]{Con74} and is countable since $L^2(M,\vphi)$ is separable. We also have
    \[
        \Sd(\vphi)=\Sd(M) \subset \Gamma(\sigma^{(\vphi, \Sd(\vphi))}) \subset \Sp(\sigma^{(\vphi, \Sd(\vphi))}) = \Sd(\vphi).
    \]
where the first inclusion follows from Lemma~\ref{lem:point_modular_spectrum_is_intersection_Connes_spectrum}, the second inclusion is by definition of the Connes spectrum, and the last equality follows from \cite[Equation (1)]{Con74}. Hence $M^{\sigma^{(\vphi, \Sd(\vphi))}} = M^\vphi$ is a factor by \cite[Theorem 2.4.1]{Con73}, and consequently $\Gamma(\sigma^{(\vphi, \Sd(\vphi))}) =\Sp(\sigma^{(\vphi, \Sd(\vphi))})= \Sd(\vphi)$.\\

\noindent $(ii)\Leftrightarrow (iii):$ Recall that $M^{\sigma^{(\vphi,\Sd(\vphi))}}=M^\vphi$. So \cite[Theorem 2.4.1]{Con73} implies $M^\vphi$ is a factor if and only if $\Gamma(\sigma^{(\vphi, \Sd(\vphi))}) = \Sp(\sigma^{(\vphi, \Sd(\vphi)})$, but the latter condition is equivalent to (ii) by \cite[Equation (1)]{Con74}.\\

\noindent Finally, if $M$ is full then $(iii)\Rightarrow(i)$ follows from \cite[Lemma 4.8]{Con74}.
\end{proof}

Lastly we note that for $\vphi$ almost periodic one has $\S(\vphi) = \overline{\Sd(\vphi)}$ and hence
    \[
        \Sd(M)\subset \S(M) \subset \overline{\Sd(\vphi)}.
    \]
Thus the condition $\Sd(\vphi)=\Sd(M)$ in the previous lemma also implies $\overline{\Sd(\vphi)}=S(M)$.

\subsection{The basic construction,  Pimsner--Popa orthogonal bases, and index}\label{sec:basic_construction_PP_bases_and_index}

In this section we will quickly recall a few notions from subfactor theory. We refer the reader to \cite[Chapter 3]{Kos98} for additional details in the purely infinite setting. In particular, see \cite[Section 3.4]{Kos98} for details on Pimsner--Popa orthogonal bases and their relation to the index.

If $N\subset M$ is an inclusion of von Neumann algebras such that there exists a faithful normal conditional expectation $\E\colon M\to N$, then we write $N\overset{\E}{\subset} M$ and say the inclusion is \emph{with expectation}. Let $\vphi$ be any faithful normal semifinite weight on $N$ so that $L^2(N,\vphi)$ and $L^2(M,\vphi\circ \E)$ yield standard forms for $N$ and $M$, respectively. We can identify $L^2(N,\vphi)$ as a subspace of $L^2(M,\vphi\circ \E)$ and the projection $e_N$ onto this subspace is called the \emph{Jones projection} of $N\overset{\E}{\subset} M$. The \emph{basic construction} for $N\overset{\E}{\subset} M$ is then the von Neumann algebra generated by $M$ and $e_N$, which we denote $\<M,e_N\>$. One has $e_N xe_N = \E(x)e_N$ for all $x\in M$, which implies $\<M,e_N\>$ is independent of $\vphi$ and generated by $M+Me_N M$. It also implies $e_N \< M, e_N\> e_N = N e_N \cong N$. Another convenient presentation of the basic construction is $\<M,e_N\> = (J_{\vphi\circ \E} N J_{\vphi\circ \E})'$, where $J_{\vphi\circ \E}$ is the modular conjugation for $\vphi\circ \E$.

A \emph{Pimsner--Popa orthogonal system} for $N\overset{\E}{\subset}M$ is a family $\{m_i\in M\colon i\in I\}$ such that for all $i,j\in I$ one has
    \[
        \E(m_i^* m_j) = \delta_{i=j} p_i
    \]
for some projection $p_i\in N$. Note that if $e_N$ is the Jones projection, then this equivalent to saying that $\{m_i e_N\colon i\in I\}$ is a family of partial isometries in $\<M,e_N\>$ with orthogonal range projections. A \emph{Pimsner--Popa orthogonal basis} for $N\overset{\E}{\subset}M$ is a Pimsner--Popa orthogonal system $\{m_i\colon i\in I\}\subset M$ satisfying the further condition that
    \[
        x = \sum_{i\in I} m_i \E(m_i^* x)
    \]
for all $x\in M$, where the sum converges strongly. Note that this is equivalent to the family of partial isometries $\{m_i e_N\colon i\in I\}$ in $\<M,e_N\>$ satisfying
    \begin{align}\label{eqn:PP_basis_second_condition}
        \sum_{i\in I} m_i e_N m_i^* = 1.
    \end{align}
Pimsner--Popa orthogonal bases need not always exist (unless one is willing to allow $m_i$ to be unbounded operators affiliated with $M$). However, they always exist for inclusions of type $\mathrm{II}_1$ factors (see \cite[Proposition 5.21]{Bur17}), and we show in Proposition~\ref{prop:existence_PP_basis} below that this generalizes to inclusions admitting extremal almost periodic states that are compatible with the conditional expectation. 

As noted in the introduction, there are many equivalent ways to define the index of $\E$, but for our purposes the most convenient definition is in terms of Pimsner--Popa orthogonal bases as follows. Associated to the inclusion $N\overset{\E}{\subset} M$ is a unique faithful normal semifinite operator valued weight\footnote{See \cite[Definition IX.4.12]{Tak03}.} 
    \[
        \widehat{\E}\colon \<M,e_N\>_+ \to \widehat{M}_+ 
    \]
satisfying $\widehat{\E}(e_N)=1$, which we say is \emph{dual} to $\E$. Observe that (\ref{eqn:PP_basis_second_condition}) implies $\sum_{i\in I} m_im_i^* = \widehat{\E}(1)\in \widehat{(M'\cap M)}_+$. For a factor $M$ one then defines the \emph{index} of $\E$ to be
    \[
        \Ind{\E} := \sum_{i\in I} m_i m_i^* \in [0, +\infty]
    \]
if a Pimsner--Popa orthogonal basis $\{m_i\in M\colon i\in I\}$ exists, and otherwise $\Ind{\E}:=+\infty$. Note that in the case that $N$ and $M$ are type $\mathrm{II}_1$ factors and $\E$ is the unique trace preserving conditional expectation, one has $\Ind(\E)=[M\colon N]=\dim_N L^2(M)$.


\section{Modules and Dimension}\label{sec:modules_and_dimension}

In this section we will develop notions of modules and Murray--von Neumann dimension for von Neumann algebras equipped with faithful normal strictly semifinite weights. The key idea is to consider the basic construction associated to the centralizer subalgebra, so we require some preliminary results in this setting.

\subsection{Centralizer basic construction}

The following lemma shows how to conjugate spectral projections of the modular operator onto one another, which will be indispensable in the almost periodic case. We remark that the last part of the lemma has appeared as \cite[Lemma 4.9]{Dyk97}, though with a small discrepancy in the inequality.

\begin{lem}\label{lem:spectral_projections_of_Delta}
Let $M$ be a von Neumann algebra equipped with a faithful normal strictly semifinite weight $\vphi$. If $\lambda$ is an eigenvalue of the modular operator $\Delta_\vphi$, then there exists a family of partial isometries $\mathcal{V}_\lambda\subset M^{(\vphi,\lambda)}$ satisfying
    \[
        1_{\{\lambda\}}(\Delta_\vphi) = \sum_{v\in \mathcal{V}_\lambda} v 1_{\{1\}}(\Delta_\vphi)v^*
    \]
and 
    \[
        \sum_{v\in \mathcal{V}_\lambda} vv^* = \bigvee \{ww^*\colon w\in M^{(\vphi,\lambda)} \text{ is a partial isometry}\} \in (M^\vphi)'\cap M^\vphi.
    \]
If $M^\vphi$ is a factor, for any $\mu$ in the point spectrum of $\Delta_\vphi$ one has
    \begin{align}\label{eqn:spectral_compression_equality}
        1_{\{\lambda\mu\}}(\Delta_\vphi) = \sum_{v\in \mathcal{V}_\lambda} v 1_{\mu}(\Delta_\vphi)v^*,
    \end{align}
and when $\lambda\geq 1$ one can choose $\mathcal{V}_\lambda$ to consist of a single co-isometry.
\end{lem}
\begin{proof}
First observe that by strict semifiniteness one can consider $L^2(M^\vphi,\vphi)$ and view it as a non-trivial subspace of $L^2(M,\vphi)$, and in fact $1_{\{1\}}(\Delta_\vphi)$ is precisely the (necessarily non-zero) projection onto this subspace. Now, suppose $\xi\in \ker(\Delta_\vphi - \lambda)\setminus\{0\}$. Note that $\Delta_\vphi J_\vphi = J_\vphi \Delta_\phi^{-1}$ implies $J_\vphi \xi\in \ker(\Delta_\vphi - \lambda^{-1})$. Using Haagerup's theory of $L^p$ spaces (see \cite{Haa79c}), one can identify $\xi$ and $J_\vphi \xi$ with certain $\tau$-measurable operators affiliated with the continuous core $c(M)$ of $M$, and in fact one has $\xi^* = J_\vphi \xi$. It follows that $|\xi|=((J_\vphi \xi)\xi)^{1/2}\in \ker(\Delta_\vphi - 1)$. Moreover, if $\xi=v|\xi|$ is the polar decomposition, then $v\in M^{(\vphi,\lambda)}$. In particular, $M^{(\vphi,\lambda)}$ contains at least one partial isometry and so by Zorn's lemma we may consider a maximal family $\mathcal{V}_\lambda\subset M^{(\vphi,\lambda)}$ of non-zero partial isometries with pairwise orthogonal range projections.

Now, denote $q:=\sum_{v\in \mathcal{V}_\lambda} vv^*$ and note that $q\in M^\vphi$. We first claim that $q$ is a central projection in $M^\vphi$. For now, we adopt the notation of $z(p\colon M^\vphi)$ for the central support of a projection $p\in M^\vphi$. It suffices to show $q= z(q\colon M^\vphi)$, so suppose towards a contradiction that $r:=z(q\colon M^\vphi) - q$ is non-zero. Observe that $r$ is \emph{not} centrally orthogonal (in $M^\vphi$) to each $vv^*$ for $v\in \mathcal{V}_\lambda$, since this would imply
    \[
        z(q\colon M^\vphi) = \bigvee_{v\in \mathcal{V}_\lambda} z(vv^*\colon M^\vphi) \leq z(q\colon M^\vphi) - z(r\colon M^\vphi),
    \]
forcing $z(r\colon M^\vphi)$ and hence $r$ to be zero. Thus there exists some $v\in \mathcal{V}_\lambda$ which is not centrally orthogonal to $r$ and we can therefore find a non-zero partial isometry $w\in M^\vphi$ satisfying $w^*w \leq vv^*$ and $ww^*\leq r$. But then $\mathcal{V}_\lambda \cup\{ wv\}$ contradicts the maximality of $\mathcal{V}_\lambda$.

Next, denote
    \[
        z:=\bigvee \{ww^*\colon w\in M^{(\vphi,\lambda)} \text{ is a partial isometry}\}.
    \]
Note that $z$ is a central projection in $M^\vphi$ since $uwu^*\in M^{(\vphi,\lambda)}$ for all unitaries $u\in M^\vphi$.  We claim that $q=z$. Indeed, we have
    \[
        q = \bigvee_{v\in \mathcal{V}_\lambda} vv^* \leq z,
    \]
so suppose, towards a contradiction, that $z-q\neq 0$. Then by definition of the supremum in the lattice of projections, there exists a partial isometry $w\in M^{(\vphi,\lambda)}$ so that $ww^* \not\leq q$. Equivalently, the projection $(z-q)ww^*$ is non-zero, but then $\mathcal{V}_\lambda \cup\{ (z-q)w\}$ contradicts the maximality of $\mathcal{V}_\lambda$.

Toward establishing the other claimed property of $\mathcal{V}_\lambda$ and (\ref{eqn:spectral_compression_equality}), consider $\eta\in \ker(\Delta_\vphi - \nu)$ for some $\nu>0$. Then $v^*\eta \in \ker(\Delta_\vphi - \frac{\nu}{\lambda})$ for each $v\in \mathcal{V}_\lambda$. Thus for $\mu>0$
    \[
        \sum_{v\in \mathcal{V}_\lambda} v 1_{\{\mu\}}(\Delta_\vphi)v^* \eta = \sum_{v\in \mathcal{V}_\lambda} v \delta_{\nu=\mu\lambda} v^*\eta = \delta_{\nu=\mu\lambda} \left(\sum_{v\in \mathcal{V}_\lambda} vv^*\right) \eta.
    \]
Letting $q=\sum_{v\in \mathcal{V}_\lambda} vv^* \in M^\vphi$ as above, note that $q$ commutes with $1_{\{\mu\lambda\}}(\Delta_\vphi)$ so that the above computation therefore gives
    \begin{align}\label{eqn:spectral_compression_inequality}
        \sum_{v\in \mathcal{V}_\lambda} v 1_{\{\mu\}}(\Delta_\vphi)v^* =q 1_{\{\mu\lambda\}}(\Delta_\vphi)  \leq 1_{\{\mu\lambda\}}(\Delta_\vphi).
    \end{align}
For $\mu=1$, the above inequality will be an equality provided $q$ has trivial kernel on the image of $1_{\{\lambda\}}(\Delta_\vphi)$; that is, provided $\ker(\Delta_\vphi-\lambda )\cap \ker( q)=\{0\}$. Let $\xi\in \ker(\Delta_\vphi-\lambda )\cap \ker( q)$ and let $\xi=w|\xi|$ be its polar decomposition. Then $q\xi=0$ implies $qww^*=0$ since $ww^*$ is the range projection of $\xi$. We have seen above that $w\in M^{(\vphi,\lambda)}$, and so $\mathcal{V}_\lambda \cup \{w\}$ contradicts the maximality of $\mathcal{V}_\lambda$ unless $w$, and hence $\xi=w|\xi|$, is zero. Thus $\mathcal{V}_\lambda$ has the claimed property.

Finally, suppose $M^\vphi$ is a factor. Then $q=1$ since we showed above that it is a central projection in  $M^\vphi$, and hence (\ref{eqn:spectral_compression_equality}) follows from (\ref{eqn:spectral_compression_inequality}). Now assume $\lambda\geq 1$. The first part of the proof implies there exists some $v_0\in M^{(\vphi,\lambda)}$. Let $\{(s_i,r_i) \in M^\vphi\times M^\vphi\colon i\in I\}$ be a maximal family of pairs of partial isometries satisfying $s_is_i^*=v_0^*v_0$, $r_i^*r_i=v_0v_0^*$, $\sum_{i\in I} s_i^*s_i=p$ for some projection $p\in M^\vphi$, and $\sum_{i\in I} r_ir_i^* =q$ for some projection $q\in M^\vphi$. Note that $w:=\sum_{i\in I} r_i v_0 s_i \in M^{(\vphi,\lambda)}$ satisfies
    \[
        w^*w = \sum_{i,j\in I} s_i^* v_0^* r_i^* r_j v_0 s_j = \sum_{i\in I} s_i^* v_0^* r_i^* r_i v_0 s_i = \sum_{i\in I} s_i^* v_0^*v_0 s_i = \sum_{i\in I} s_i^* s_i = p,
    \]
and
    \[
        ww^* = \sum_{i,j\in I} r_i v_0 s_is_j^* v_0^* r_j^* = \sum_{i\in I} r_i v_0 s_is_i^* v_0^* r_i^* = \sum_{i\in I} r_i v_0v_0^* r_i^* = \sum_{i\in I} r_ir_i^* = q.
    \]
We must now separately consider the two cases of $\vphi$ being finite or infinite. 

First suppose $\vphi(1)=\infty$. We claim that $\vphi(q)=\infty$. Indeed, if not then
    \[
        \vphi(p) = \vphi(w^*w) = \frac{1}{\lambda} \vphi(ww^*) = \frac{1}{\lambda} \vphi(q)<\infty.
    \]
Consequently $\vphi(1-p)=\vphi(1-q)=\infty$, and therefore  $v_0^*v_0 \preceq_{M^\vphi} 1-p$ and $v_0v_0^* \preceq_{M^\vphi} 1-q$. But this contradicts the maximality of $\{(s_i,r_i)\colon i\in I\}$. Hence $\vphi(q)=\infty$ so that $q \sim_{M^\vphi} 1$. Letting $u\in M^\vphi$ be such that $uu^*=q$ and $u^*u=1$, we can take $u^*w\in M^{(\vphi,\lambda)}$ to be the desired co-isometry. 

Now suppose $\vphi(1)<\infty$. Then using $\lambda \geq 1$ we have
    \begin{align*}
        \vphi(1-q) - \vphi(v_0v_0^*) &=  \vphi(1) - \vphi(ww^*) - \vphi(v_0v_0^*)\\
            &= \vphi(1) - \lambda \vphi(w^*w) - \lambda \vphi(v_0^*v_0) \leq \lambda[\vphi(1 - p) - \vphi(v_0^*v_0)].
    \end{align*}
Thus if the left-most quantity is strictly positive then so is the right-most quantity, giving $v_0v_0^* \preceq_{M^\vphi} 1-q$ and $v_0^*v_0\preceq_{M^\vphi} 1-p$ and contradicting the maximality of  $\{(s_i,r_i)\colon i\in I\}$. So we must have $\vphi(1-q) \leq \vphi(v_0v_0^*)$, and in particular $1-q\preceq_{M^\vphi} v_0v_0^*$. We can then find a partial isometry $r_0\in M^\vphi$ with $r_0r_0^* = 1-q$ and $r_0^*r_0\leq v_0v_0^*$. Note that $(r_0v_0)(r_0v_0)^* = 1-q$ and
    \[
        \vphi( (r_0 v_0)^*(r_0v_0)) = \frac{1}{\lambda} \vphi( 1-q) \leq \vphi(1-p).
    \]
Thus we can find a partial isometry $s_0\in M^\vphi$ with $s_0s_0^*=(r_0v_0)^*(r_0v_0)$ and $s_0^*s_0\leq 1-p$. Then $v:= w + r_0v_0s_0$ is the desired co-isometry.
\end{proof}

\begin{rem}\label{rem:Sd_is_subgroup}
Recall that $\mu>0$ is in the point spectrum of $\Delta_\vphi$ if and only if $1_{\{\mu\}}(\Delta_\vphi)>0$. Thus (\ref{eqn:spectral_compression_equality}) in the previous lemma implies the point spectrum of $\Delta_\vphi$ is a subgroup of $\R_+$ when $M^\vphi$ is a factor. In the almost periodic case, this can be deduced from work of Connes (see Lemma~\ref{lem:extremal_from_Gamma} and also \cite[Lemma 4.9]{Dyk97}). We also point out that $\sum_{v\in \mathcal{V}_\lambda} vv^* = S_\gamma(1)$ in the notation of \cite[Definition 2.5]{Dyk95}.$\hfill\blacksquare$
\end{rem}

Let $M$ be a von Neumann algebra equipped with a faithful normal strictly semifinite weight $\vphi$, and let $\E_\vphi\colon M\to M^\vphi$ be the unique faithful normal $\vphi$-preserving conditional expectation (see Lemma~\ref{lem:strictly_semifinite_characterization}). Let $e_\vphi$ and $\<M,e_\vphi\>$ denote the Jones projection and basic construction associated to $\E_\vphi$, respectively.
We will identify $M$ and $M^\vphi$ with their representations on $L^2(M,\vphi)$, in which case $e_\vphi$ equals the projection onto the subspace $L^2(M^\vphi,\vphi)\leq L^2(M,\vphi)$. In particular, one has $e_\vphi = 1_{\{1\}}(\Delta_\vphi)$.

\begin{prop}\label{prop:basic_construction}
Let $M$ be a von Neumann algebra equipped with a faithful normal strictly semifinite weight $\vphi$. Then:
    \begin{enumerate}[label=(\alph*)]
    \item in $B(L^2(M,\vphi))$ we have
    \[
        \<M,e_\vphi\> = (J_\vphi M^\vphi J_\vphi)' = (M\cup\{\Delta_\vphi^{it}\colon t\in \R\})'';
    \]
    
    \item the central support of $e_\vphi$ in $\<M,e_\vphi\>$ is $1$ and $(Me_\vphi M)''=\<M,e_\vphi\>$;
    
    \item\label{part:tracial_weight} $\<M,e_\vphi\>$ admits a unique faithful normal semifinite tracial weight $\tau$ satisfying $\tau(e_\vphi xe_\vphi)=\vphi(x)$ for all $x\in M_+$.
    \end{enumerate}
Moreover, if $\vphi$ is almost periodic then
     \[
        \bigcup_{\lambda\in \text{Sd}(\vphi)} \mathcal{V}_\lambda,
    \]
is a Pimsner--Popa orthogonal basis for $M^\vphi\overset{\E_\vphi}{\subset} M$, where $\mathcal{V}_\lambda$ is as in Lemma~\ref{lem:spectral_projections_of_Delta}. If one further assumes that $\vphi$ is extremal, then $\Ind{\E_\vphi} = |\Sd(\vphi)|$.
\end{prop}
\begin{proof}
(a): The first equality is true for any basic construction. For the second, note that
    \[
        M^\vphi = M\cap \{\Delta_\vphi^{it}\colon t\in \R\}',
    \]
and $J_\vphi \Delta_\vphi^{it} = \Delta_\vphi^{it} J_\vphi$ for all $t\in \R$ imply $M\cup \{\Delta_\vphi^{it}\colon t\in \R\}\subset (J_\vphi M^\vphi J_\vphi)'=\<M,e_\vphi\>$. Conversely, $e_\vphi = 1_{\{1\}}(\Delta_\vphi)\in \{\Delta_\vphi^{it}\colon t\in \R\}''$ by Stone's theorem.\\

\noindent(b): First note that since $M+Me_\vphi M$ is weakly dense in $\<M,e_\vphi\>$, the central support of $e_\vphi$ is given by the projection onto
    \[
        \overline{\<M,e_\vphi\>e_\vphi L^2(M,\vphi)} = \overline{Me_\vphi L^2(M,\vphi)}= \overline{M L^2(M^\vphi, \vphi)}.
    \]
By Lemma~\ref{lem:strictly_semifinite_characterization}, the strict semifiniteness of $\vphi$ gives a family $\{p_i\in M^\vphi\cap \dom(\vphi)\colon i\in I\}$ of pairwise orthogonal projections which sum to $1$. For $x\in \sqrt{\dom}(\vphi)$, we have
    \[
        x = \sum_{i\in I} J_\vphi p_i J_\vphi x =  \sum_{i\in I} x p_i \in \overline{ML^2(M^\vphi,\vphi)}.
    \]
Therefore $ML^2(M^\vphi,\vphi)$ is dense in $L^2(M,\vphi)$, and so the central support of $e_\vphi$ is $1$. Since $(M e_\vphi M)''$ is a closed ideal in $\<M,e_\vphi\>$, we have $(Me_\vphi M)''=\<M,e_\vphi\> z$ for some central projection $z$. Since this projection necessarily dominates $e_\vphi \in (M e_\vphi M)''$, we must have $z=1$.\\

\noindent(c): Let $\{v_i\in \<M,e_\vphi\>\colon i\in I\}$ be a maximal family of partial isometries with pairwise orthogonal source projections and $v_iv_i^*\leq e_\vphi$. Since $e_\vphi$ has full central support by the previous part, we must have
    \[
        \sum_{i\in I} v_i^*v_i =1.
    \]
For $x\in \<M,e_\vphi\>_+$ define
    \[
        \tau(x):=\sum_{i\in I} \vphi(v_i x v_i^*),
    \]
where we have identified $M^\vphi\cong M^\vphi e_\vphi = e_\vphi \<M,e_\vphi\> e_\vphi$. We first check that $\tau$ is tracial: for $x\in \<M,e_\vphi\>$ one has
    \[
        \tau(x^*x ) =\sum_{i\in I} \vphi(v_i x^*xv_i^*)= \sum_{i,j\in I} \vphi(v_i x^* v_j^* v_j x v_i^*) = \sum_{i,j\in I} \vphi( v_jxv_i^*v_i x^*v_j^*) = \sum_{j\in I} \vphi(v_jxx^* v_j^*) = \tau(xx^*),
    \]
where we have used that $\vphi$ is tracial on $M^\vphi$. Next, for $x\in M_+$ observe
    \[
        \tau(e_\vphi x e_\vphi) = \sum_{i\in I} \vphi( v_ie_\vphi \mathcal{E}_\vphi(x) e_\vphi v_i^*) = \sum_{i\in I} \vphi( \mathcal{E}_\vphi(x)^{\frac12} e_\vphi v_i^* v_i e_\vphi \mathcal{E}_\vphi(x)^{\frac12} ) = \varphi(\mathcal{E}_\vphi(x)) = \varphi(x).
    \]
This implies $e_\vphi \dom(\vphi) e_\vphi \subset \dom(\tau)$, and since $\tau$ is tracial we further have $\dom(\vphi)e_\vphi \dom(\vphi) \subset \dom(\tau)$. Recalling $(Me_\vphi M)'' = \<M,e_\vphi\>$ by the previous part, we see that the semifiniteness of $\tau$ follows from that of $\vphi$. The full central support of $e_\vphi$ also implies $\tau$ is determined by its restriction to $e_\vphi \<M,e_\vphi\> e_\vphi = M^\vphi e_\vphi$, and so the above equality gives the uniqueness of $\tau$. The normality of $\tau$ follows from that of each $\vphi(v_ixv_i^*)$. Finally, if $\tau(x^*x)=0$ then $v_ix^*xv_i^*=0$ for all $\in I$ by the faithfulness of $\vphi$ and consequently
    \[
        xx^* = \sum_{i\in I} xv_i^* v_i x^* = 0.
    \]
Hence $\tau$ is faithful.\\

\noindent
Finally, assume $\vphi$ is almost periodic, and for each $\lambda\in \Sd(\vphi)$ let $\mathcal{V}_\lambda$ be as in Lemma~\ref{lem:spectral_projections_of_Delta}. For $v\in \mathcal{V}_\lambda$ we have
    \[
        \E_\vphi(v^*v) = v^*v.
    \]
Suppose $w\in  \bigcup_{\lambda}\mathcal{V}_\lambda$ is distinct from $v$, say with $w\in \mathcal{V}_\mu$. If $\mu=\lambda$, then $vv^*$ and $ww^*$ are orthogonal by definition of $\mathcal{V}_\lambda$ and hence $\E_\vphi(v^*w)=0$. Otherwise $\lambda^{-1}\mu\neq 1$ so that $v^*w\in M^{(\vphi,\lambda^{-1}\mu)}$ implies $\E_\vphi(v^*w)=0$. Additionally, the almost periodicity of $\vphi$ and Lemma~\ref{lem:spectral_projections_of_Delta} gives
    \[
        1 = \sum_{\lambda\in \text{Sd}(\vphi)} 1_{\{\lambda\}}(\Delta_\vphi) = \sum_{\lambda\in \text{Sd}(\vphi)} \sum_{v\in \mathcal{V}_\lambda} v 1_{\{1\}}(\Delta_\vphi)v^* = \sum_{\lambda\in \text{Sd}(\vphi)} \sum_{v\in \mathcal{V}_\lambda} v e_\vphi v^*,
    \]
so that $\bigcup_\lambda \mathcal{V}_\lambda$ is a Pimsner--Popa orthogonal basis. If $\vphi$ is extremal, then this also implies
    \[
        \Ind{\E_\vphi} = \sum_{\lambda\in \Sd(\vphi)} \sum_{v\in \mathcal{V}_\lambda} vv^* = \sum_{\lambda\in \Sd(\vphi)} 1= |\Sd(\vphi)|,
    \]
where the second equality follows from Lemma~\ref{lem:spectral_projections_of_Delta}.
\end{proof}

In the next proposition we characterize when a representation of $M$ lifts to a representation of $\<M,e_\vphi\>$. Note that $M\subset B(L^2(M,\vphi))$ is an example of such a representation by the definition of the basic construction, and it turns out the underlying reason is that the set of fixed points under the unitary group $\{\Delta_\vphi^{it}\colon t\in \R\}$, namely $L^2(M^\vphi,\vphi)$, is non-trivial and has dense orbit under $M$.

\begin{prop}\label{prop:characterization_abstract_modules}
Let $M$ be a von Neumann algebra equipped with a faithful normal strictly semifinite weight $\vphi$. Denote by $\<M,e_\vphi\>$ the basic construction for the inclusion $M^\vphi\subset M$ and let $\H$ be a Hilbert space. Then the following are equivalent:
    \begin{enumerate}[label=(\roman*)]
        \item There exists a representation $\pi\colon \<M,e_\vphi\>\to B(\H)$;\label{part:rep}

        \item\label{part:group_rep} There exists a representation $\pi_0\colon M\to B(\H)$ and a strongly continuous unitary representation $U\colon \R\to B(\H)$ such that $\pi_0(M)\H^{U}$ is dense in $\H$ and
            \[
                U_t \pi_0(x) U_t^* = \pi_0(\sigma_t^\vphi(x)) \qquad \forall t\in \R,\ x\in M;
            \]
        
        \item\label{part:isom} There exists a Hilbert space $\K$ and an isometry $v\colon \H\to L^2(M,\vphi)\otimes \K$ satisfying
            \begin{align*}
                (x\otimes1)v\H &\subset v\H \qquad \forall x\in M,\\
                (\Delta_\vphi^{it}\otimes 1)v\H &\subset v\H \qquad \forall t\in \R.
            \end{align*}
    \end{enumerate}
In particular, given any of the above objects one can arrange for the others to satisfy $\pi|_M=\pi_0$, $\pi(\Delta_\vphi^{it})=U_t$ for all $t\in \R$, and $v\pi(x) = (x\otimes 1)v$ for all $x\in \<M,e_\vphi\>$.
\end{prop}
\begin{proof}
\noindent $\ref{part:rep}\Rightarrow \ref{part:group_rep}:$ Set $\pi_0:=\pi|_M$ and $U_t:=\pi(\Delta_\vphi^{it})$ so that $\Ad{U_t}\circ \pi_0 = \pi\circ \Ad{\Delta_\varphi^{it}} = \pi\circ \sigma_t^\varphi$, and
    \[
        \pi(Me_\vphi M)\H \subset \pi_0(M) \H^U.
    \]
Since $1\in (Me_\vphi M)''$ by Proposition~\ref{prop:basic_construction}, we can find a net $(x_i)_{i\in I} \subset Me_\vphi M$ converging strongly to $1$. Thus
    \[
        \H \subset \overline{\pi(Me_\vphi M) \H},
    \]
and so $\pi_0(M)\H^U$ is dense.\\

\noindent$\ref{part:group_rep}\Rightarrow \ref{part:isom}:$ Let $\{\xi_i\in \H^U\colon i\in I\}$ be a maximal family of unit vectors such that $\pi_0(M)\xi_i\perp \pi_0(M)\xi_j$ for all $i\neq j$. Note that
    \[
       \K:= \overline{\text{span}}\left( \bigcup_{i\in I} \pi_0(M)\xi_i \right)
    \]
is invariant for $U$ since each $U_t$ normalizes $\pi_0(M)$ and fixes $\xi_i$, $i\in I$. Consequently, $[\K^\perp]\H^U \subset \H^U$ and so we must have $[\K^\perp]\H^U=\{0\}$ since otherwise we would contradict the maximality of $\{\xi_i\in \H^U\colon i\in I\}$. Thus $\H^U = [\K] \H^U\subset \K$ and therefore $\H= \overline{\pi_0(M)\H^U} =\K$. 

For each $j\in I$, observe that
    \[
        \psi_j(x):= \< \pi_0(x)\xi_j, \xi_j\>
    \]
satisfies
    \[ 
        \psi_j(\sigma_t^\vphi(x)) = \< U_t \pi_0(x) U_t^* \xi_j, \xi_j\> = \< \pi_0(x)\xi_j,\xi_j\> = \psi_j(x).
    \]
It follows that the support projection $s(\psi_j)\in M^\vphi$, and so the Connes cocycle derivatives $u_t^{(j)}:=(D\psi_j \colon D \vphi)_t$ are unitaries  on the subspace $s(\psi_j) L^2(M,\vphi)$ (see \cite[Theorem VIII.3.19]{Tak03}). Moreover, the uniqueness of $u_t^{(j)}$ along with the invariance of $\psi_j$ and $\vphi$ under $\sigma_s^\vphi$ implies $\sigma_s^\vphi( u_t^{(j)}) = u_t^{(j)}$ for all $s,t\in \R$. Thus 
    \[
        u_{s+t}^{{(j)}} = u_s^{(j)}  \sigma_s^\vphi(u_t^{(j)}) = u_s^{(j)} u_t^{(j)},
    \]
and so $t\mapsto u_t^{(j)}$ is a strongly continuous one parameter unitary group. Stone's theorem then yields a positive operator $h_j$ affiliated with $s(\psi_j) M^\vphi s(\psi_j)$ satisfying $h_j^{it} = u_t^{(j)}$ for all $t\in \R$. Since $\vphi(h_j^{\frac12}\cdot h_j^{\frac12})$ satisfies the same relative modular condition as $\psi_j$ with respect to $\vphi$, it follows that $\psi_j = \vphi(h_j^{\frac12}\cdot h_j^{\frac12})$. In particular, $h_j^{\frac12} \in L^2(M^\vphi,\vphi)$ since $\psi_j$ is a state. Therefore for finite subsets $F\subset I$ the map
    \[
        v\left(\sum_{j\in F} \pi_0(x_j)\xi_j\right):= \sum_{j\in F}(x_jh_j^{\frac12})\otimes \delta_j
    \]
extends to an $v\colon \H\to L^2(M,\vphi)\otimes \ell^2(I)$. Observe that $(x\otimes 1)v = v\pi_0(x)$ for all $x\in M$ and $(\Delta_\vphi^{it}\otimes 1)v = v U_t$ for all $t\in \R$, so that $v$ is the desired isometry.
\\

\noindent$\ref{part:isom}\Rightarrow \ref{part:rep}:$ Let $v$ be as in \ref{part:isom} and define $\pi(x):=v^*(x\otimes 1) v$ for $x\in \<M,e_\vphi\>\subset B(L^2(M,\vphi))$. The assumptions on $v$ along with Proposition~\ref{prop:basic_construction} imply that $v\H$ is reducing for $\<M,e_\vphi\>\otimes \C$, and hence 
    \[
        vv^*\in (\<M,e_\vphi\>\otimes \C)' = (J_\vphi M^\vphi J_\vphi)\bar\otimes B(\K).
    \]
Thus for $x,y\in \<M,e_\vphi\>$ we have
    \[
        \pi(x)\pi(y) = v^*(x\otimes 1)vv^* (y\otimes 1)v = v^*( (xy)\otimes 1)vv^* v= \pi(xy).
    \]
It follows that $\pi$ is a normal unital $*$-homomorphism. Finally, observe that
    \[
        v\pi(x) = vv^*(x\otimes 1)v = (x\otimes 1)vv^*v= (x\otimes 1)v
    \]
for all $x\in \<M,e_\vphi\>$.
\end{proof}

\begin{rem}\label{rem:intertwiner_amplification_space}
The Hilbert space $\K$ in Proposition~\ref{prop:characterization_abstract_modules}.\ref{part:isom} constructed in the above proof has dimension bounded by that of $\H$.  One can also replace (and even enlarge) $\K$ as follows. Suppose $v\colon \H\to L^2(M,\vphi)\otimes \K$ is as in Proposition~\ref{prop:characterization_abstract_modules}.\ref{part:isom}. Then for any isometry $w\colon \K\to \K'$ one has 
    \[
        (x\otimes 1) (1\otimes w)v\H = (1\otimes w)(x\otimes 1)v\H \subset (1\otimes w) v\H
    \]
for all $x\in \<M,e_\vphi\>$. Moreover $v^*(1\otimes w^*)(x\otimes 1)(1\otimes w)v = v^*(x\otimes 1)v$, so the isometry $(1\otimes w)v$ induces the same representation $\pi\colon \<M,e_\vphi\>\to B(\H)$ as $v$. Consequently, one can always take $\K=\H$, and if $\H$ is separable then we can take $\K=\ell^2$.$\hfill\blacksquare$
\end{rem}

\begin{rem}
Even without the density assumption in Proposition~\ref{prop:characterization_abstract_modules}.\ref{part:group_rep}, $\<M,e_\vphi\>$ will still admit a representation on the closure of $\pi_0(M) \H^U$. So long as $\H^U\neq \{0\}$, one can in this way always induce representations of $\<M,e_\vphi\>$ from representations of $M$.
$\hfill\blacksquare$
\end{rem}

\subsection{Definitions and general properties}\label{sec:definitions_and_properties}

In light of Proposition~\ref{prop:characterization_abstract_modules}, we make the following definition. 

\begin{defi}
Let $M$ be a von Neumann algebra equipped with a faithful normal strictly semifinite weight $\vphi$. We say a Hilbert space $\H$ is a \textbf{left $(M,\vphi)$-module} (resp. \textbf{right $(M,\vphi)$-module}) if there exists a representation
    \[
        \pi\colon \<M, e_\vphi\>\to B(\H) \qquad (\text{resp. } \pi\colon  \<M,e_\vphi\>^{op}\to B(\H)).
    \]
In this case we write ${_{(M,\vphi)}}(\H,\pi)$ (resp. $(\H,\pi)_{(M,\vphi)}$).$\hfill\blacksquare$
\end{defi}

When working with right $(M,\vphi)$-modules it is useful to note that $\<M,e_\vphi\>^{\text{op}}$ can be identified with $J_\vphi \<M,e_\vphi\> J_\vphi \subset B(L^2(M,\vphi))$ via $x^{\text{op}}\mapsto J_\vphi x^* J_\vphi$, in which case its commutant is nothing but $M^\vphi$. We will also see in Proposition~\ref{prop:left_vs_right_modules} below a one-to-one correspondence between left and right $(M,\vphi)$-modules.

For $\<M,e_\vphi\>\otimes \C$ invariant (resp. $J_\vphi \<M,e_\vphi\> J_\vphi\otimes \C$ invariant) subspaces $\H\leq L^2(M,\vphi)\otimes \K$  or more generally if $\pi$ is clear from context then we suppress the notation for it. We will also adopt the usual terminology for module maps, and in particular we will say two left (resp. right) $(M,\vphi)$-modules $(\H,\pi)$ and $(\K,\rho)$ are \emph{isomorphic} if there exists  a unitary $u\in B(\H,\K)$ satisfying $u\pi(x) = \rho(x) u$ for all $x\in \<M,e_\vphi\>$ (resp. $x\in \<M,e_\vphi\>^{op}$).

\begin{prop}\label{prop:dim_well_defined}
Let $M$ be a von Neumann algebra equipped with a faithful normal strictly semifinite weight $\vphi$, let $(\H,\pi)$ be a left (resp. right) $(M,\vphi)$-module, and let $\K_1$ and $\K_2$ be Hilbert spaces. For $i=1,2$, suppose $v_i\colon \H\to L^2(M,\vphi)\otimes \K_i$ is an isometry satisfying $v_i \pi(x) = (x\otimes 1)v_i$ for all $x\in \<M,e_\vphi\>$ (resp. $\forall x\in \<M,e_\vphi\>^{op}$). Then
    \[
        (\vphi\otimes \Tr)\left[(J_\vphi\otimes 1) v_1v_1^* (J_\vphi\otimes 1)\right] = (\vphi\otimes \Tr)\left[(J_\vphi\otimes 1) v_2v_2^* (J_\vphi\otimes 1)\right]
    \]
(resp. $(\vphi\otimes \Tr)(v_1v_1^*) = (\vphi\otimes \Tr)(v_2v_2^*)$).
\end{prop}
\begin{proof}
Observe that $v_2v_1^*\colon v_1\H \to v_2\H$ is an isometry which we can view as a partial isometry in $B(L^2(M,\vphi)\otimes (\K_1\oplus \K_2))$ with $(v_2v_1^*)^* v_2v_1^*=v_1v_1^*$ and $v_2v_1^*(v_2v_1^*)^* = v_2v_2^*$. Moreover, one has
    \[
        (x\otimes 1)v_2v_1^* = v_2\pi(x)v_1^* = v_2v_1^*(x\otimes 1)
    \]
for all $x\in \<M,e_\vphi\>$ (resp. $x\in \<M,e_\vphi\>^{\text{op}}$ after identifying $\<M,e_\vphi\>^{\text{op}}\cong J_\vphi \<M,e_\vphi\> J_\vphi$). Hence $v_2v_1^*\in (J_\vphi M^\vphi J_\vphi)\bar\otimes B(\K_1\oplus \K_2)$ (resp. $M^\vphi \bar\otimes B(\K_1\oplus \K_2)$). Similarly for $v_1v_1^*$ and $v_2v_2^*$. So using that $\vphi\otimes \Tr$ is tracial on $M^\vphi\bar\otimes B(\K_1\oplus \K_2)$, we have
    \begin{align*}
        (\vphi\otimes \Tr)\left[ (J_\vphi\otimes 1) v_2v_2^*(J_\vphi\otimes 1)\right]&= (\vphi\otimes \Tr)\left[ (J_\vphi\otimes 1) v_2v_1^*v_1v_2^*(J_\vphi\otimes 1)\right]\\
            &= (\vphi\otimes \Tr)\left[ (J_\vphi\otimes 1) v_1v_2^*v_2v_1^*(J_\vphi\otimes 1)\right] = (\vphi\otimes \Tr)\left[ (J_\vphi\otimes 1) v_1v_1^*(J_\vphi\otimes 1)\right].
    \end{align*}
In the case of a right $(M,\vphi)$-module, the same computation without the $J_\vphi\otimes 1$ terms yields $(\vphi\otimes \Tr)(v_2v_2^*) = (\vphi\otimes \Tr)(v_1v_1^*)$.
\end{proof}

In other words, the previous proposition tells us that the quantity $(\vphi\otimes \Tr)\left[(J_\vphi\otimes 1) v_1v_1^* (J_\vphi\otimes 1)\right]$ is an invariant of $_{(M,\vphi)} (\H,\pi)$. When $\vphi$ is a tracial state, this quantity is simply the Murray--von Neumann dimension of $(\H,\pi)$. We therefore adopt the following terminology.

\begin{defi}\label{defi:MvN_dim}
Let $M$ be a von Neumann algebra equipped with a faithful normal strictly semifinite weight $\vphi$, and let $(\H,\pi)$ be a left (resp. right) $(M,\vphi)$-module. We say an isometry $v\colon \H\to L^2(M,\vphi)\otimes \K$ for some Hilbert space $\K$ is a \textbf{standard intertwiner} for $(\H,\pi)$ if
    \[
        v\pi(x) = (x\otimes 1)v \qquad \forall x\in \<M,e_\vphi\>
    \]
(resp. $\forall x\in \<M,e_\vphi\>^{op}$). The \textbf{Murray--von Neumann dimension} of $(\H,\pi)$ over $(M,\vphi)$ is the quantity
    \[
            \dim_{(M,\vphi)}(\H,\pi) :=(\vphi\otimes \Tr)[(J_\vphi\otimes 1)vv^*(J_\vphi\otimes 1)]
    \]
(resp. $\dim (\H,\pi)_{(M,\vphi)}:= (\vphi\otimes \Tr)(vv^*)$). $\hfill\blacksquare$
\end{defi}

Proposition~\ref{prop:characterization_abstract_modules} implies $(M,\vphi)$-modules always admit standard intertwiners and Remark~\ref{rem:intertwiner_amplification_space} implies there are in fact always infinitely many. Nevertheless, Proposition~\ref{prop:dim_well_defined} tells us the Murray--von Neumann dimension is independent of the choice of standard intertwiner. 

As a first example, consider an $\<M,e_\vphi\>\otimes \C$ invariant subspace $\H\leq L^2(M,\vphi)\otimes \K$ for some Hilbert space $\K$. Then $\H$ is a left $(M,\vphi)$-module and the inclusion map $\iota\colon \H\to L^2(M,\vphi)\otimes \K$ gives a standard intertwiner with $\iota\iota^*=[\H]$. If $[\H]=\sum_{i\in I} (J_\vphi p_i J_\vphi)\otimes q_i$ for a family of projections $\{p_i\in M^\vphi \colon i\in I\}$ and an orthogonal family of minimal projections $\{q_i\in B(\K)\colon i\in I\}$ then $\H\cong \bigoplus_{i\in I} L^2(M,\vphi) p_i$ with
    \[
        \dim_{(M,\vphi)} (\H) = \sum_{i\in I} \vphi(p_i).
    \]
Similarly, if $\H$ is a $J_\vphi \<M,e_\vphi\> J_\vphi \otimes \C$ invariant subspace with $[\H] = \sum_{i\in I} p_i\otimes q_i$, then $\H\cong \bigoplus_{i\in I} p_i L^2(M,\vphi)$ with
    \[
        \dim(\H)_{(M,\vphi)} = \sum_{i\in I} \vphi(p_i).
    \]
In particular, if $\vphi$ is a state then the Murray-von Neumann dimension of $L^2(M,\vphi)^{\oplus d}$ as a left or right $(M,\vphi)$-module is $d$.

The following proposition implies it suffices to only consider left $(M, \vphi)$-modules, which we will do for the remainder of the article. Recall that for a Hilbert space $\H$ its conjugate Hilbert space is the set $\overline{\H}:=\{\bar\xi\colon \xi\in \H\}$ equipped with the vector space operations $\alpha \bar\xi + \bar\eta = \overline{ \bar\alpha \xi + \eta}$ and inner product $\<\bar\xi,\bar\eta\>:= \<\eta, \xi\>$. The map $\xi\mapsto \bar\xi$ defines a canonical conjugate linear unitary $j\colon \H \to \overline{\H}$.

\begin{prop}\label{prop:left_vs_right_modules}
Let $M$ be a von Neumann algebra equipped with a faithful normal strictly semifinite weight $\vphi$ and let $(\H,\pi)$ be a left $(M ,\vphi)$-module. Define 
    \begin{align*}
        \overline{\pi}\colon \<M,e_\vphi\>^{\text{op}} &\to B(\overline{\H})\\
             x^{\text{op}} &\mapsto j \pi(x^*) j^*.
    \end{align*}
Then $(\overline{\H}, \overline{\pi})$ is a right $(M,\vphi)$-module with
    \[
        \dim(\overline{\H},\overline{\pi})_{(M,\vphi)} =\dim_{(M,\vphi)}(\H,\pi).
    \]
\end{prop}
\begin{proof}
That $\overline{\pi}$ is representation of $\<M,e_\vphi\>^{\text{op}}$ on $\overline{\H}$ follows readily from the definition. Let $v\colon \H\to L^2(M,\varphi)\otimes \K$ be a standard intertwiner of $(\H,\pi)$. Identify $\<M,e_\vphi\>^{op}\cong J_\vphi \<M,e_\vphi\>J_\vphi$ via $x^{\text{op}}\mapsto J_\vphi x^* J_\vphi$ so that $\overline{\pi}(J_\vphi x^* J_\vphi) = j\pi(x^*) j^*$ for all $x\in \<M,e_\vphi\>$. Observe that the isometry $w:= (J_\vphi \otimes 1)v j^*$ is a standard intertwiner for $(\overline{\H},\overline{\pi})$:
    \[
        w \overline{\pi}(J_\vphi x^* J_\vphi) = [(J_\vphi\otimes 1) v j^*] [j\pi(x^*) j^*] = (J_\vphi \otimes 1) (x^* \otimes 1) vj^* = [(J_\vphi x^* J_\vphi)\otimes 1] w,
    \]
for all $x\in \<M,e_\vphi\>$. Consequently,
    \[
        \dim(\mathcal{\overline{H}}, \overline{\pi})_{(M,\varphi)} = (\varphi \otimes \text{Tr})(ww^*) =(\varphi \otimes \text{Tr})( (J_\varphi \otimes 1) v v^*(J_\varphi \otimes 1)) = \dim_{(M,\varphi)} (\mathcal{H}, \pi),
    \]
as claimed.
\end{proof}

The rest of the subsection is devoted to showing that Murray-von Neumann dimension over $(M,\vphi)$ is well-behaved with respect to natural operations on modules. All of these results are extensions of corresponding results for tracial states (see \cite[Section 2.1.4]{BenThesis}, for example).

\begin{prop}\label{prop:bnded_opts_between_left_Hilbert_modules}
Let $M$ be a von Neumann algebra equipped with a faithful normal strictly semifinite weight $\vphi$, let $(\H, \pi)$ and $(\K, \rho)$ be left $(M,\vphi)$-modules, and let $T \in B(\mathcal{H},\mathcal{K})$ be such that $T\pi(x) = \rho(x) T$ for all $x \in \langle M, e_\vphi\rangle$.
    \begin{enumerate}[label=(\alph*)] 
    \item If $T$ has dense image, then $\dim_{(M,\varphi)}(\mathcal{K}, \rho ) \leq \dim_{(M,\varphi)}(\mathcal{H}, \pi)$. 
    
    \item If $T$ is injective, then we have $\dim_{(M,\varphi)}(\mathcal{H}, \pi) \leq \dim_{(M,\varphi)}(\mathcal{K},\rho).$ 
    
    \item If $T$ is injective with dense range, then $(\mathcal{H},\pi)$ and $(\mathcal{K},\rho)$ are isomorphic as left $(M,\varphi)$-modules with $\dim_{(M,\varphi)}(\mathcal{H}, \pi) = \dim_{(M,\varphi)}(\mathcal{K},\rho)$.
\end{enumerate}
Additionally, if $\vphi$ is extremal then $(\mathcal{H},\pi)$ and $(\mathcal{K},\rho)$ are isomorphic as left $(M,\varphi)$-modules if and only if $\dim_{(M,\varphi)}(\mathcal{H}, \pi) = \dim_{(M,\varphi)}(\mathcal{K},\rho)$.
\end{prop}
\begin{proof}
Denote $\mathcal{L}:=\H\oplus \K$. Then by Remark~\ref{rem:intertwiner_amplification_space} we can find standard intertwiners $v\colon \H\to L^2(M,\vphi)\otimes \mathcal{L}$ and $w\colon \K\to L^2(M,\vphi)\otimes \mathcal{L}$ for $(H, \pi)$ and $(\K,\rho)$, respectively. The map $w  T  v^*: vv^*(L^2(M,\vphi) \otimes \mathcal{L}) \to ww^*(L^2(M,\vphi) \otimes \mathcal{L})$ satisfies
    \[ 
        w  T v^* ( x \otimes 1) = w T \pi(x) v^* = w \rho(x)Tv^*= (x\otimes 1)wTv^*,
    \]
for all $x \in \langle M, e_\varphi\rangle$. Let $S$ denote the extension of this map to $L^2(M,\vphi) \otimes \mathcal{L}$ given by setting it equal to zero on $(1-vv^*)(L^2(M,\vphi) \otimes \mathcal{L})$. It follows that $S \in  (\langle M , e_\varphi\rangle \otimes \C) '= (J_\vphi M^\vphi J_\vphi)\bar\otimes B(\mathcal{L})$. Let $S = u |S|$ be the polar decomposition. Note that $u\in (J_\vphi M^\vphi J_\vphi)\bar\otimes B(\mathcal{L})$ with $u^*u = [\ker(S)^\perp] \leq vv^*$ and $uu^* = [\text{ran}(S)] \leq ww^*$.\\

\noindent(a): If $T$ has dense range, then the range of $S$ is dense in $ww^* (L^2(M,\vphi) \otimes \mathcal{L})$ since $v$ is an isometry. Thus we have that $uu^* = ww^*$, and Proposition~\ref{prop:dim_well_defined} gives
    \[ 
        \dim_{(M,\varphi)}(\mathcal{K}, \rho) =\varphi\otimes \Tr((J_\varphi \otimes 1) uu^*( J_\varphi \otimes 1))=\varphi\otimes \Tr((J_\varphi \otimes 1)u^*u(J_\varphi \otimes 1))\leq  \dim_{(M,\varphi)}(\mathcal{H}, \pi).  
    \]\\

\noindent(b): If $T$ is injective, then $\ker(S)= (1-vv^*)L^2(M,\vphi)\otimes \mathcal{L}$ since $w$ is an isometry. Thus we have $u^*u = vv^*$, and Proposition~\ref{prop:dim_well_defined} gives
    \[ 
        \dim_{(M,\varphi)}(\mathcal{H}, \pi) = \varphi\otimes \Tr((J_\varphi \otimes 1) u^*u( J_\varphi \otimes 1))=\varphi\otimes \Tr((J_\varphi \otimes 1)u^*u(J_\varphi \otimes 1))\leq \dim_{(M,\varphi)}(\mathcal{K}, \rho).  
    \]\\

\noindent(c): If $T$ is injective with dense range, then previous parts imply
    \[
        \dim_{(M,\varphi)}(\mathcal{H}, \pi)=\dim_{(M,\varphi)}(\mathcal{K}\rho).
    \] 
It also follows that $w^*uv \colon \H\to \K$ is a unitary satisfying $\rho(x) w^*uv = w^*uv\pi(x)$ for all $x \in \langle M, e_\vphi\rangle$.\\

\noindent Lastly, assume $\vphi$ is extremal. The ``only if'' part of the final statement follows from (c). Conversely, suppose $\dim_{(M,\varphi)}(\mathcal{H}, \pi) = \dim_{(M,\varphi)}(\mathcal{K},\rho)$. Since $(J_\vphi M^\vphi J_\vphi)\bar\otimes B(\mathcal{L})$ is a semifinite factor in this case, the equality of dimensions implies $vv^*$ and $ww^*$ are equivalent in this factor. Let $u\in (J_\vphi M^\vphi J_\vphi)\bar\otimes B(\mathcal{L})$ be a partial isometry satisfying $u^*u = vv^*$ and $uu^* = ww^*$. Then $w^*uv \colon \H\to \K$ is a unitary satisfying
    \[
        w^* u v \pi(x) = w^* u (x\otimes 1) v = w^*(x\otimes 1) uv = \rho(x) w^* u v
    \]
for all $x\in \<M,e_\vphi\>$, so that $(\H,\pi)$ and $(\K,\rho)$ are isomorphic.
\end{proof}

Suppose $(\H,\pi)$ is a left $(M,\vphi)$-module and $\K\leq \H$ is a closed, $\pi(\<M,e_\vphi\>)$-invariant subspace. Then $(\K,\pi^\K)$ is a left $(M,\vphi)$-submodule of $(\H,\pi)$ where $\pi^\K(x):=\pi(x)|_\K$ for $x\in \<M,e_\vphi\>$. However, in the remainder of this section we will abuse notation and simply write $\pi$ for $\pi^\K$.

\begin{lem}\label{lem:dim_finite_sum}
Let $M$ be a von Neumann algebra equipped with a faithful normal strictly semifinite weight $\vphi$, let $(\H, \pi)$ be a left $(M,\vphi)$-module, and let $(\K,\pi)$ be a left $(M,\vphi)$-submodule. Then $(\H\ominus \K,\pi)$ is a left $(M,\vphi)$-submodule and
    \[
        \dim_{(M,\vphi)}(\H,\pi) = \dim_{(M,\vphi)}(\K,\pi) + \dim_{(M,\vphi)}(\H\ominus \K,\pi).
    \]
\end{lem}
\begin{proof}
First, the $\pi(\<M,e_\vphi\>)$-invariance of $\K$ implies that of $\H\ominus \K$, and hence $(\H\ominus \K,\pi)$ is a left $(M,\vphi)$-submodule. Now, if $v\colon \H\to L^2(M,\vphi)\otimes \H$ is a standard intertwiner, then $v|_K$ and $v|_{\H\ominus\K}$ are standard intertwiners for $(\K, \pi)$ and $(\H\ominus \K,\pi)$, respectively, and they satisfy $v|_K v|_K^* = v[K]v^*$ and $v|_{\H\ominus \K} v|_{\H\ominus \K}^* = v[\H\ominus \K]v^*$. Noting that these projections sum to $vv^*$ gives the claimed equality.
\end{proof}

\begin{prop}\label{prop:net_of_submodules}
Let $M$ be a von Neumann algebra equipped with a faithful normal strictly semifinite weight $\vphi$, let $(\H, \pi)$ be a left $(M,\vphi)$-module, and let $(\mathcal{H}_i,\pi)_{i \in I}$ be a net of left $(M,\vphi)$-submodules.
\begin{enumerate}[label=(\alph*)] 
\item If $(\mathcal{H}_i)_{ i \in I}$ is an increasing net, then 
    \[
        \dim_{(M,\varphi)}\left( \overline{\bigcup_{i\in I} \mathcal{H}_i}, \pi\right)= \sup_{i \in I}\,\dim_{(M,\varphi)}(\mathcal{H}_i, \pi).
    \]
\item If $(\mathcal{H}_i)_{ i \in I}$ is a decreasing net with $\dim_{(M,\varphi)}(\mathcal{H}_{i_0})< \infty$ for some $i_0 \in I$, then 
    \[
        \dim_{(M,\varphi)}\left(\bigcap_{i \in I}\mathcal{H}_i, \pi \right)= \inf_{i\in I}\,\dim_{(M,\varphi)}(\mathcal{H}_i, \pi).
    \]
\item If $\H_i\perp \H_j$ for all distinct $i,j\in I$, then
    \[
        \dim_{(M,\varphi)}\left (\bigoplus_{i\in I} \mathcal{H}_i, \pi \right)= \sum_{ i \in I} \dim_{(M,\varphi)}(\mathcal{H}_i, \pi).
    \]
\end{enumerate}
\end{prop}
\begin{proof}
(a): Set $\mathcal{K} := \overline{ \bigcup_{i \in I}\mathcal{H}_i}$ and fix a standard intertwiner $v\colon\H\to L^2(M,\vphi)\otimes \mathcal{L}$ for $(\H,\pi)$. Note that $v_0:=v|_K$ and $v_i:=v|_{\H_i}$ are standard intertwiners for $(\K,\pi)$ and $(\H_i, \pi)$, respectively, with $v_0v_0^* = \sup_{i\in I} v_i v_i^*$. Hence
    \begin{align*}
        \dim_{(M,\vphi)}(\K,\pi) &= (\vphi\otimes \Tr)[(J_\vphi\otimes 1) v_0v_0^* (J_\vphi\otimes 1)] \\
        &= \sup_{i\in I}\, (\vphi\otimes \Tr)[(J_\vphi\otimes 1) v_iv_i^* (J_\vphi\otimes 1)]  = \sup_{i\in I}\, \dim_{(M,\vphi)}(\H_i,\pi)
    \end{align*}
by the normality of $\vphi\otimes \Tr$.\\
    
\noindent(b): For each $i\in I$ set $\K_i = \H_{i_0} \ominus \H_i$. For $\K:=\bigcup_{i\in I}\K_i$ we then have $\K= \H_{i_0}\ominus \bigcap_{i\in I} \H_i$. Since $(\K_i)_{i\in I}$ is an increasing net, part (a) and Lemma~\ref{lem:dim_finite_sum} imply
    \begin{align*}
        \dim_{(M,\vphi)}(\H_{i_0},\pi) - \dim_{(M,\vphi)}\left(\bigcap_{i\in I}\H_i,\pi \right) &= \dim_{(M,\vphi)}(\K,\pi) \\
        &= \sup_{i\in I}\, \dim_{(M,\vphi)}(\K_i,\pi)\\
        &= \dim_{(M,\vphi)}(\H_{i_0},\pi) - \inf_{i\in I}\, \dim_{(M,\vphi)}(\H_i,\pi)
    \end{align*}
The desired equality follows from the above after subtracting the finite quantity $\dim_{(M,\vphi)}(\H_{i_0},\pi)$.\\

\noindent(c): For each finite subset $F\Subset I$, denote $\K_F:= \bigoplus_{i\in F} \H_i$. Then $(\K_F)_{F\Subset I}$ is an increasing net with $\overline{\bigcup_{F\Subset I}\K_F} = \bigoplus_{i\in I} \H_i$. Part (a) and Lemma~\ref{lem:dim_finite_sum} then give
    \begin{align*}
        \dim_{(M,\vphi)}\left( \bigoplus_{i\in I} \H_i,\pi\right) = \sup_{F\Subset I} \dim_{(M,\vphi)}(\K_F,\pi) = \sup_{F\Subset I} \sum_{i\in F} \dim_{(M,\vphi)}(\H_i,\pi) = \sum_{i\in I} \dim_{(M,\vphi)}(\H_i,\pi),
    \end{align*}
as claimed.
\end{proof}

Note that if $\{(\H_i,\pi_i)\colon i\in I\}$ is a family of left $(M,\vphi)$-modules then 
    \[
        \dim_{(M,\vphi)}\left( \bigoplus_{i\in I} \H_i, \bigoplus_{i\in I} \pi_i \right) = \sum_{i\in I} \dim_{(M,\vphi)}(\H_i,\pi_i)
    \]
by Propositions~\ref{prop:bnded_opts_between_left_Hilbert_modules}.(c) and \ref{prop:net_of_submodules}.(c). In particular, for a fixed left $(M,\vphi)$-module $(\H,\pi)$ one has
    \[
        \dim_{(M,\vphi)}(\H^{\oplus n},\pi^{\oplus n}) = n \dim_{(M,\vphi)}(\H,\pi)
    \]
for all $n\in \N$.

We next show how Murray--von Neumann dimension over $(M,\vphi)$ behaves under compressions by projections $p\in M^\vphi$. Given that such projections might have $\vphi(p)=\infty$, we will not attempt to re-normalize $\vphi|_{pMp}$ in any way. The absence of such a re-normalization factor will cause the resulting formulas to differ from those in the tracial state case, but only superficially. Note also that in this case $(pMp)^{\vphi|_{pMp}} = pM^\vphi p$ and we can identify $\<pMp,e_{\vphi|_{pMp}}\>$ with $\<pMp, p e_\vphi\> \leq \<M,e_\vphi\>$. (Indeed, this identification follows from the proof of \cite[Lemma 2.4]{ILP98}, which as stated assumes requires that $pMp$ is a factor with separable predual but these hypotheses are not needed to prove the existence of an isomorphism $\<pMp,e_{\vphi|_{pMp}}\>\cong \<pMp, p e_\vphi\>$ that is the identity on $pMp$ and carries $e_{\vphi|_{pMp}}$ to $p e_\vphi$.) Thus any representation $\pi$ of $\<M,e_\vphi\>$ restricts to a representation of $\<pMp, e_{\vphi|_{pMp}}\>$, but for the remainder of this section we will abuse notation and simply write $\pi$ for this restriction.

\begin{prop}\label{prop:dimension_of_compressions}
Let $M$ be a von Neumann algebra equipped with a faithful normal strictly semifinite weight $\vphi$ and let $(\H,\pi)$ be a left $(M ,\vphi)$-module. Then for any non-zero projection $p\in M^\vphi$ one has
    \[
        \dim_{(pMp,\vphi|_{pMp})}(\pi(p)\H,\pi) = \dim_{(zMz, \vphi|_{zMz})}(\pi(z)\H, \pi),
    \]
where $z\in M^\vphi$ is the central support of $p$ in $M^\vphi$.
\end{prop}
\begin{proof}
Let $v:\pi(z)\H \to L^2(zMz,\varphi) \otimes \K$ be a standard intertwiner for $(\pi(z)\H,\pi)$ as an $(zMz, \vphi|_{zMz})$-module. Since $vv^* \in (J_\vphi zM^\vphi z J_\vphi) \bar\otimes B(\K)$, if we fix a family of matrix units $\{e_{a,b}: a,b \in A\}$ for $B(\K)$ then there exists $x_{a,b} \in M^\vphi$ such that 
\[ 
    (1 \otimes e_{a,a})vv^*(1 \otimes e_{b,b}) = J_\vphi zx_{a,b}z J_\vphi \otimes e_{a,b}
\]
for all $a,b\in A$. That is, $vv^*$ can be identified with the matrix over $A$ whose $(a,b)$-entry is $J_\vphi zx_{a,b}z J_\vphi$ and consequently
    \begin{align}\label{eqn:dimension_under_central_corner} 
        \dim_{(zMz,\vphi|_{zMz})}(\pi(z)\H, \pi) = \sum_{a \in A} \varphi(zx_{a,a}z)=\sum_{a \in A} \varphi(x_{a,a}z).
    \end{align}

Now, since $z$ is the central support of $p$ in $M^\vphi$, there exists a family of partial isometries $\{ w_i \colon i\in I\} \subset M^\vphi$ satisfying $w_iw_i^* \leq p$ and  $\sum_{i\in I} w_i^*w_i=z$. We can then define an isometry as follows:
\begin{align*}
    w:pL^2(zMz,\varphi) &\to  \bigoplus_{i \in I} J_\vphi w_iw_i^*J_\vphi L^2(pMp,\varphi)\\
      \xi &\mapsto (J_\varphi w_i J_\vphi \xi)_{i \in I}.
\end{align*} 
Since $p \in zM^\vphi z$, we have that $v\pi(p) = (p \otimes 1)v$ and  $v\pi(p):\pi(p)\H \to pL^2(z M z, \vphi) \otimes \K$ is the restriction of an isometry and hence an isometry itself. Noting that $w$ is $\< pMp, pe_\vphi\>$-equivariant, we see that $(w\otimes 1)v \pi(p)=((wp)\otimes 1)v$ is a standard intertwiner for $(\pi(p)\H, \pi)$. Moreover, $((wp)\otimes 1)vv^* ((pw^*)\otimes 1)$ can be identified with the matrix over $I\times A$ whose $((i,a),(j,b))$-entry is
    \[
        (J_\vphi w_i J_\vphi p)J_\vphi zx_{a,b}zJ_\vphi (p J_\vphi w_j^* J_\vphi )=J_\vphi w_i x_{a,b} w_j^* J_\vphi p.
    \]
Identifying the above as an operator on $L^2(pMp,\vphi)$ it equals $J_\vphi w_i x_{a,b} w_j^* J_\vphi$ and thus
\begin{align*}
    \dim_{(pMp, \vphi|_{pMp})}(\pi(p)\H, \pi) &= \sum_{(i,a) \in I\times A} \vphi(w_i x_{a,a} w_i^*)\\
    &=\sum_{(i,a) \in I \times A} \vphi(x_{a,a} w_i^*w_i)\\
    &= \sum_{ a\in A} \vphi(x_{a,a}z)=\dim_{(zMz,\vphi|_{zMz})}(\pi(z)\H, \pi),
\end{align*}
where we have used $w_i\in M^\vphi$ in the second equality.
\end{proof}

\begin{rem}\label{rem:computation_of_dimension_under_central_corner}
Note that (\ref{eqn:dimension_under_central_corner}) in the proof of the previous proposition gives
    \[
        \dim_{(pMp,\vphi|_{pMp})}(\pi(p)\H,\pi) = \dim_{(zMz, \vphi|_{zMz})}(\pi(z)\H, \pi) = (\varphi\otimes \Tr)\left[(J_\vphi\otimes 1)vv^* (J_\vphi\otimes 1)(z\otimes 1)\right],
    \]
for any standard intertwiner $v\colon \H\to L^2(M,\vphi)\otimes \K$. This computation will be useful later in the article.$\hfill\blacksquare$
\end{rem}

The compression formula in Proposition~\ref{prop:dimension_of_compressions} also yields the following amplification formula. 

\begin{prop}\label{prop:amplification_formula}
Let $M$ be a von Neumann algebra equipped with a faithful normal strictly semifinite weight $\vphi$. Given a left $(M,\vphi)$-module $(\H,\pi)$ and a Hilbert space $\K$, one has that $(\H\otimes \K,\pi\otimes1)$ is a left $(M\otimes B(\K), \vphi\otimes \Tr)$-module with
    \[
        \dim_{(M\bar\otimes B(\K), \vphi\otimes \Tr)}(\H\otimes \K,\pi\otimes1) = \dim_{(M,\vphi)}(\H,\pi).
    \]
\end{prop}
\begin{proof}
Noting that
    \[
        (M\bar\otimes B(\K))^{\vphi\otimes \Tr} = M^\vphi \bar\otimes B(\K),
    \]
it follows that $\<M\bar\otimes B(\K),e_{\vphi\otimes \Tr}\> = \<M,e_\vphi\>\bar\otimes B(\K)$. Consequently, $(\H\otimes \K,\pi\otimes 1)$ is indeed a left $(M\bar\otimes B(\K),\vphi\otimes \Tr)$-module. Now for a minimal projection $e\in B(\K)$, the central support of $1\otimes e$ in $M^\vphi \bar\otimes B(\K)$ is $1\otimes 1$ and hence Proposition~\ref{prop:dimension_of_compressions} implies
    \[
        \dim_{(M\bar\otimes B(\K),\vphi\otimes \Tr)}(\H\otimes\K, \pi\otimes 1) = \dim_{( M \otimes \C e, \vphi\otimes \Tr|_{M \otimes \C e})}( \H\otimes (e\K), \pi\otimes 1).
    \]
The claimed formula then holds by identifying $( M \otimes \C e, \vphi\otimes \Tr|_{M \otimes \C e})\cong (M,\vphi)$, under which we have the identification $( \H\otimes (e\K), \pi\otimes 1) \cong (\H,\pi)$.
\end{proof}

The inclusion $\<M,e_\vphi\> \cong \<M,e_\vphi\>\otimes \C\subset \<M\bar\otimes B(\K),e_{\vphi\otimes \Tr}\>$ established in the previous proof implies we can regard $(\H\otimes \K, \pi\otimes 1)$ as an $(M,\vphi)$-module. In particular, for recall from the discussion following Proposition~\ref{prop:net_of_submodules} that for $\K=\C^n$ one has
    \[
        \dim_{(M,\vphi)}(\H\otimes \C^n , \pi\otimes 1) = \dim_{(M,\vphi)}(\H^{\oplus n}, \pi^{\oplus n}) = n \dim_{(M,\vphi)}(\H,\pi).
    \]
Thus Proposition~\ref{prop:amplification_formula} implies
    \[
        \dim_{(M,\vphi)}(\H\otimes \C^n, \pi\otimes 1) = n \dim_{(M\otimes M_n(\C), \vphi \otimes \Tr)}(\H\otimes \C^n, \pi\otimes 1).
    \]
Note that the factor of $n$ above can be identified with the index of the inclusion $M\otimes \C \subset M\otimes M_n(\C)$ (with respect to usual conditional expectation $1\otimes (\frac1n\Tr)$). We will generalize this in Proposition~\ref{prop:finite_index_amplification_formula} below.

Our next result demonstrates how to use actions of locally compact groups to induce modules over the associated crossed product. Note this can be used to give an alternate proof of the above amplification formula by using a trivial action of a discrete group $\Gamma$ with $\ell^2(\Gamma)\cong \K$, iterating with the dual action of $\widehat{\Gamma}$, and appealing to Takesaki duality (see \cite{Tak73}).

\begin{prop}\label{prop:crossed_product_modules}
Let $M$ be a von Neumann algebra equipped with a faithful normal strictly semifinite weight $\vphi$. Suppose $G\overset{\alpha}{\curvearrowright} M$ be an action of a locally compact group and let $\widetilde{\vphi}$ be the weight on $M\rtimes_\alpha G$ dual to $\vphi$. Then every left $(M,\vphi)$-module $(\H,\pi)$ induces a left $(M\rtimes_\alpha G, \widetilde{\vphi})$-module $(L^2(G)\otimes \H, \widetilde{\pi})$ satisfying
    \[
        [\widetilde{\pi}(x) \xi](s) = \pi(\alpha_{s^{-1}}(x))\xi(s) \qquad \text{ and  } \qquad [\widetilde{\pi}(\lambda(t))\xi](s) = \xi(t^{-1} s)
    \]
for all $x\in M$, $s,t\in G$, and $\xi\in L^2(G)\otimes \H$. When $G$ is discrete, one also has
    \[
        \dim_{(M\rtimes_\alpha G,\widetilde{\vphi})}(L^2(G)\otimes \H, \widetilde{\pi}) = \dim_{(M,\vphi)}(\H,\pi).
    \]
\end{prop}

\begin{proof}
Let $v\colon \H\to L^2(M,\vphi)\otimes \K$ be a standard intertwiner for $(\H,\pi)$. We will show that $1\otimes v \colon L^2(G)\otimes \H\to L^2(G)\otimes L^2(M,\vphi)\otimes \K$ gives a standard intertwiner for $(L^2(G)\otimes \H, \widetilde{\pi})$ (recall that one of the properties of the dual weight is that $L^2(M\rtimes_\alpha G, \widetilde{\vphi})\cong L^2(G)\otimes L^2(M,\vphi)$). First, let $G\overset{u}{\curvearrowright} L^2(M,\vphi)$ be the unitary implementation of $G\overset{\alpha}{\curvearrowright} M$: $u_s x u_s^* = \alpha_s(x)$ for $x\in M$, and $u_s J_\vphi u_s^* = J_\vphi$ (see \cite[Theorem IX.1.15]{Tak03}). We can then define a unitary $U\in B(L^2(G)\otimes L^2(M,\vphi))$ by $ [U\xi](s) = u_s \xi(s)$, so that
    \[
        M\rtimes_\alpha G = (L(G)\otimes \C) \vee U^*(\C\otimes M) U \subset B(L^2(G)\otimes L^2(M,\vphi)).
    \]
Also define a unitary $V\in B(L^2(G))$ by $[Vf](s) = \Delta_G(s)^{-\frac12} f(s^{-1})$, where $\Delta_G\colon G\to \R_+$ is the modular function. Then $V^*=V$ and $(V\otimes 1)U=U^*(V\otimes 1)$. With this notation, the modular conjugation for the dual weight $\widetilde{\vphi}$ is given by
    \[
        J_{\widetilde{\vphi}} = U^*(V\otimes J_\vphi) = (V\otimes J_\vphi) U
    \]
(see \cite[Lemma X.1.13]{Tak03}). Consequently, we have
    \begin{align}\label{eqn:lifted_std_intertwiner}
        (J_{\widetilde{\vphi}}\otimes 1)(1\otimes vv^*)(J_{\widetilde{\vphi}}\otimes 1) = (U^*\otimes 1)(1\otimes \left[ (J_\vphi\otimes 1)vv^*(J_\vphi\otimes 1)\right])(U\otimes 1).
    \end{align}
Since $U^*(\C\otimes M)U$ is the copy of $M$ inside $M\rtimes_\alpha G$, since $U^*(\C\otimes M^\vphi)U \subset (M\rtimes_\alpha G)^{\widetilde{\vphi}}$ by \cite[Theorem X.1.17.(ii)]{Tak03}, and since $(J_\vphi\otimes 1) vv^* (J_\vphi\otimes 1) \in M^\vphi\bar\otimes B(\K)$, we see that the above element lies in $(M\rtimes_\alpha G)^{\widetilde{\vphi}}\otimes B(\K)$. Therefore $1\otimes vv^*$ commutes with $\<M\rtimes_\alpha G, e_{\widetilde{\vphi}}\>\otimes \C$ and so we can define a representation
    \begin{align*}
    \widetilde{\pi}: \langle M \rtimes_\alpha G, e_{\tilde{\varphi}} \rangle &\to B(L^2(G) \otimes \H)\\
      x &\mapsto (1 \otimes v^*)(x\otimes 1) (1 \otimes v),
\end{align*}
for which $1\otimes v$ is a standard intertwiner. (Note that $x\otimes 1$ is supported in the first two tensor factors of $L^2(G)\otimes L^2(M,\vphi)\otimes \K$, while $1\otimes v$ is valued in the second two tensor factors.) Observe that for $x\in M$, $s,t\in G$, and  $\xi\in L^2(G)\otimes \H$ one has
    \[
        [\widetilde{\pi}(x)\xi](s) = v^*(\alpha_{s^{-1}}(x)\otimes 1) v \xi(s) = \pi(\alpha_{s^{-1}}(x)) \xi(s),
    \]
and
    \[
        [\widetilde{\pi}(\lambda(t))\xi](s) = v^* v \xi(t^{-1}s) = \xi(t^{-1}s),
    \]
as claimed.

Finally, if $G$ is discrete then we recall that $\widetilde{\vphi}|_M=\vphi$. More precisely, $\widetilde{\vphi}(U^*(1\otimes x)U) = \varphi(x)$ for $x\in M$. So it follows from this and (\ref{eqn:lifted_std_intertwiner}) above that
    \begin{align*}
        \dim_{(M\rtimes_\alpha G, \widetilde{\vphi})}(L^2(G)\otimes \H, \widetilde{\pi}) &= \widetilde{\vphi}\otimes \Tr\left[ (J_{\widetilde{\vphi}}\otimes 1)(1\otimes vv^*)(J_{\widetilde{\vphi}}\otimes 1)\right]\\
            &= \widetilde{\vphi}\otimes \Tr\left[ (U^*\otimes 1)(1\otimes \left[ (J_\vphi\otimes 1)vv^*(J_\vphi\otimes 1)\right])(U\otimes 1)\right]\\
            &= \varphi\otimes \Tr\left[(J_\vphi\otimes 1)vv^*(J_\vphi\otimes 1)\right]\\
            &= \dim_{(M,\vphi)}(\H,\pi),
    \end{align*}
as claimed.
\end{proof}

\begin{rem}
We have not attempted to describe $\widetilde{\pi}(e_{\widetilde{\vphi}})$ in the previous proposition because, while $\Delta_{\widetilde{\vphi}}$ can be computed in terms of the modular operators for $\vphi$ and $\vphi\circ \alpha_s$ (see \cite[Lemma X.1.13]{Tak03}), in general the spectral projection $e_{\widetilde{\vphi}}$ is not so easy to describe in terms of $e_\vphi$. However, if $\vphi\circ \alpha_s = \delta(s) \vphi$ for some continuous homomorphism $\delta\colon G\to \R_+$, then $\Delta_{\widetilde{\vphi}} = (\delta \Delta_G )\otimes \Delta_\vphi$ and
    \[
        \widetilde{\pi}(e_{\widetilde{\vphi}}) = \int_{\R_+} 1_{\{s\in G\colon \delta(s)\Delta_G(s)=\frac1t\}}\ d(1\otimes \pi(P))(t),
    \]
where $P(E)=1_{E}(\Delta_\vphi)$ is a projection-valued measure for $\Delta_\vphi$. In particular, if $\vphi$ is almost periodic then one has
    \[
        \widetilde{\pi}(e_{\widetilde{\vphi}}) =\sum_{\lambda\in \Sd(\vphi)} 1_{\{s\in G\colon \delta(s)\Delta_G(s)=\frac1\lambda\}}\otimes \pi(1_{\{\lambda\}}(\Delta_\vphi)),
    \]
or if $G$ is unimodular and $\delta\equiv 1$ then $\widetilde{\pi}(e_{\widetilde{\vphi}})=1\otimes \pi(e_\vphi)$.$\hfill\blacksquare$
\end{rem}

We are grateful to Dima Shlyakhtenko for bringing our attention to the following example, which shows that the continuous core can be used to compute our Murray--von Neumann dimension by way of the previous proposition.

\begin{ex}
Let $M$ be a von Neumann algebra equipped with a faithful normal strictly semifinite weight $\vphi$. For the action $\R\overset{\sigma^\vphi}{\curvearrowright}M$ of the modular automorphism group, the associated crossed product $c(M):=M\rtimes_{\sigma^\vphi} \R$ is called the \emph{continuous core} of $M$. If we let $C_c(\R,M)$ denote the continuous compactly supported functions from $\R$ to $M$ equipped with the $\sigma$-strong* topology, then $c(M)$ admits a faithful normal semifinite tracial weight $\tau$ satisfying
    \[
        \tau(\lambda(f)^* \lambda(f)) = \int_\R \vphi\left(\hat{f}(s)^* \hat{f}(s)\right) e^{2\pi s}\ ds
    \]
for $f\in C_c(\R,M)$, where
    \[
        \lambda(f)=\int_{\R} \lambda(s) f(s)\ ds \qquad \text{ and } \qquad \hat{f}(s) = \int_{\R} e^{-2\pi i st} f(t)\ dt.
    \]
(see \cite[Section 2.1]{Kos98}; note the different choice of Fourier transform results in a discrepancy in the formulas). In particular, for $f\in C_c(\R)$ and $x\in M_+$ one has
    \[
        \tau(\lambda(f)x \lambda(f)^*)=\tau(x^{\frac12}\lambda(f)^* \lambda(f) x^{\frac12}) = \vphi(x) \int_\R |\hat{f}(s)|^2 e^{2\pi s}\ ds.
    \]
Recall that conjugation by the Fourier transform identifies $L(\R)\cong L^\infty(\R)$. So, if we let $p\in L(\R)$ be the projection identified with $1_{[0,s_0]}\in L^\infty(\R)$ for $s_0:=\frac{1}{2\pi}(\log(2\pi)+1)$, then the above equation implies
    \begin{align}\label{eqn:recover_weight_in_core_corner}
        \tau(p x p) = \vphi(x)
    \end{align}
for all $x\in M_+$.

Now, suppose $(\H,\pi)$ is a left $(M,\vphi)$-module with standard intertwiner $v\colon \H\to L^2(M,\vphi)\otimes \K$ so that $(L^2(\R)\otimes \H, \widetilde{\pi})$ is a left $(c(M), \tilde{\vphi})$-module with standard intertwiner $1\otimes v$ by Proposition~\ref{prop:characterization_abstract_modules} (and its proof), where $\widetilde{\vphi}$ is the weight dual to $\vphi$. Define a strongly continuous unitary representation $\rho\colon \R\to \mathcal{U}(L^2(\R)\otimes\H)$ by
    \[
        [\rho_t \xi](s) = \pi(\Delta_\vphi^{-it})\xi(s-t).
    \]
A few straightforward computations reveal that $\rho$ commutes with $\tilde{\pi}$ and that
    \[
        (1\otimes v) \rho_t (1\otimes v^*) = (J_{\widetilde{\vphi}}\otimes 1) (\lambda(t)^*\otimes  1)(J_{\widetilde{\vphi}}\otimes 1)
    \]
for all $t\in \R$. Thus if $p\in L(\R)\subset c(M)$ is as above, then $q:=(1\otimes v^*) (J_{\widetilde{\vphi}}\otimes 1) (p\otimes  1)(J_{\widetilde{\vphi}}\otimes 1)(1\otimes v) $ lies in the commutant of $\widetilde{\pi}(\<c(M),e_{\widetilde{\vphi}}\>)$. We can therefore view $(q(L^2(\R)\otimes\H), \widetilde{\pi})$ as a left $(c(M),\widetilde{\vphi})$-submodule, and, in particular, it is a left $(c(M),\tau)$-module. Since $L^2(c(M),\widetilde{\vphi})\cong L^2(c(M),\tau)$ as left $c(M)$-modules, it follows that
    \begin{align*}
        \dim_{(c(M),\tau)}(q(L^2(\R)\otimes\H), \widetilde{\pi}) &=(\tau\otimes \Tr)\left[(J_{\widetilde{\vphi}}\otimes 1)(1\otimes v)q(1\otimes v^*) (J_{\widetilde{\vphi}}\otimes 1)\right]\\
            &=(\tau\otimes \Tr)\left[(p\otimes 1)(J_{\widetilde{\vphi}}\otimes 1)(1\otimes vv^*) (J_{\widetilde{\vphi}}\otimes 1)(p\otimes 1)\right]\\
            &=(\vphi\otimes \Tr)\left[(J_\vphi\otimes 1) vv^* (J_\vphi\otimes 1)\right] = \dim_{(M,\vphi)}(\H,\pi),
    \end{align*}
where we have used (\ref{eqn:lifted_std_intertwiner}) and (\ref{eqn:recover_weight_in_core_corner}) in the last equality.$\hfill\blacksquare$
\end{ex}

\begin{rem}
The computation in the previous example suggests a potential definition for Murray--von Neumann dimension with respect to \emph{integrable} weights on properly infinite von Neumann algebras, which are never strictly semifinite but share many of the same properties (see \cite[Chapter II]{CT77}). A faithful normal semifinite weight $\vphi$ on $M$ is said to be integrable if the faithful normal operator-valued weight from $M$ to $M^\vphi$ given by
    \[
        \E_\vphi(x):=\int_\R \sigma_t^{\vphi}(x)\ dt \qquad x\in M_+,
    \]
is semifinite. When $M$ is a type $\mathrm{III}_1$ factor (so that the continuous core $c(M)$ is itself a factor), a particularly nice example is the dual weight $\vphi:=\widetilde{\tau}$ to a faithful normal semifinite tracial weight $\tau$ on $c(M)$, which is also known as a \emph{dominant} weight. For this weight one has $M^\vphi\cong c(M)$ and $(M\cup \{\Delta_\vphi^{it}\colon t\in \R\})''\cong c(M)$, and so the Murray--von Neumann dimension theory for modules over $c(M)$ could be used to define such a theory for $(M,\vphi)$-modules in analogy with Theorem~\ref{thm:dimension_amplification_compression_formula}. $\hfill\blacksquare$
\end{rem}

The last property that we highlight is the ability to detect cyclic vectors and separating vectors using Murray--von Neumann dimension. This is analogous to the behavior for coupling constants of finite von Neumann algebras (see \cite[Proposition V.3.13]{Tak02}).

\begin{thm}
Let $M$ be a von Neumann algebra with an extremal faithful normal state $\vphi$, and let $(\H,\pi)$ be a left $(M,\vphi)$-module.
    \begin{enumerate}[label=(\alph*)]
    \item $\dim_{(M,\vphi)}(\H,\pi) \leq 1$ if and only if $\pi(e_\vphi)\H$ contains a cyclic vector for $M$.

    \item $\dim_{(M,\vphi)}(\H,\pi) \geq 1$ if and only if $\pi(e_\vphi)\H$ contains a separating vector for $M$. 
    \end{enumerate}
\end{thm}
\begin{proof}
Let $v\colon \H\to L^2(M,\vphi)\otimes \K$ be as standard intertwiner for $(\H,\pi)$.

    \begin{enumerate}[label=(\alph*):]
    \item Suppose $\dim_{(M,\vphi)}(\H,\pi)\leq 1$. Using the factoriality of $M^\vphi\bar\otimes B(\K)$ we may assume $vv^* \leq 1 \otimes p$ for some minimal projection $p\in B(\K)$. If $\xi\in p\K$ is spanning, then $v^*(1\otimes \xi)$ is the desired cyclic vector. Conversely, suppose $\eta\in \pi(e_\vphi) \H$ is a cyclic vector for $M$. Then from the proof of Proposition~\ref{prop:characterization_abstract_modules} we can choose $v$ with range in $L^2(M,\vphi)$ so that
        \[
            \dim_{(M,\vphi)}(\H,\pi) = \vphi(J_\vphi vv^* J_\vphi) \leq \vphi(1)= 1.
        \]

    \item Suppose $\dim_{(M,\vphi)}(\H,\pi) \geq 1$. Using factoriality of $M^\vphi\bar\otimes B(\K)$ we may assume $vv^* \geq 1\otimes p$ for some minimal projection $p\in B(\K)$. If $\xi\in p\K$ is spanning, then $v^*(1\otimes \xi)$ is the desired separating vector. Conversely, suppose $\eta\in \pi(e_{\vphi}) \H$ is a separating vector for $M$. Then $\psi=\<\,\cdot\,\eta,\eta\>$ is a faithful normal linear functional satisfying $\psi\circ\sigma_t^\vphi = \psi$ for all $t\in \R$. Thus $\psi = \vphi(h\cdot h)$ for some $h\in M^\vphi_+$ with trivial kernel. From the proof of Proposition~\ref{prop:characterization_abstract_modules} we can choose $v$ so that $v\pi(x)\eta = (x\otimes 1)(h\otimes \xi)$ for some $\xi\in \K$. Thus $vv^*$ is at least as big as $[(Mh)\otimes \xi]$. But since $h$ has trivial kernel $[(Mh)\otimes \xi] = 1\otimes [\xi]$, and therefore $\dim_{(M,\vphi)}(\H,\pi)\geq \vphi\otimes \Tr(1\otimes [\xi])=1$.\qedhere 
    \end{enumerate}
\end{proof}

\subsection{Connection to semifinite Murray--von Neumann dimension}
Let $(N,\tau)$ be a von Neumann algebra equipped with a faithful normal semifinite tracial weight. In \cite[Appendix B]{Pet13}, Petersen defined the Murray--von Neumann dimension for a left $(N,\tau)$-module $(\H,\pi)$ as
    \[
        \dim'_{(N,\tau)} (\H,\pi) := \sup\{\tau(p)\dim_{(pNp, \tau_p)}(\pi(p)\H, \pi|_{pMp} ) : p \in N \text{ with } \tau(p)<\infty\},
    \]
where $\tau_p(\cdot) = \tau(p)^{-1} \tau(p \cdot p)$  (see also \cite[Definition A.14]{KPV15}). In this subsection, we relate this definition to our definition of Murray--von Neumann dimension given in Definition~\ref{defi:MvN_dim}. 

First, we show our definition recovers Petersen's for $\vphi=\tau$ a faithful normal semifinite tracial weight. Recall from the discussion preceding Definition~\ref{defi:MvN_dim} that this already holds when $\tau$ is a state.

\begin{prop}\label{prop:semifinite_dimension_achived_by_projection}
Let $(N,\tau)$ be a von Neumann algebra with a normal semifinite faithful tracial weight and $(\mathcal{H}, \pi)$ be a left $(N,\tau)$-module. Then 
    \[
        \dim'_{(N,\tau)} (\H, \pi) = \dim_{(N,\tau)}(\H, \pi).
    \]
\end{prop}
\begin{proof}
Let $v \colon\H \to L^2(N,\tau)\otimes \K$ be a standard intertwiner for $(\H, \pi)$ and denote
    \[
        q:=(J_\tau\otimes 1)vv^*(J_\tau\otimes 1) \in N\bar\otimes B(\K).
    \]
By Definition~\ref{defi:MvN_dim} we have $\dim_{(N,\tau)}(\H, \pi)=(\tau\otimes \Tr)(q)$. For any $p \in \dom(\tau)$ with central support $z$ in $N$, Remark~\ref{rem:computation_of_dimension_under_central_corner} implies
    \[
        \dim_{(pNp,\tau|_{pNp})}(\pi(p)\H, \pi)=\dim_{(Nz,\tau|_{Nz})}(\pi(z)\H,\pi)= (\tau \otimes \Tr)(q (z \otimes 1))\leq (\tau\otimes \Tr)(q).
    \]
It follows that $\dim'_{(N,\tau)} (\H, \pi) \leq \dim_{(N,\tau)}(\H,\pi)$. If $\dim'_{(N,\tau)}(\H, \pi)=+\infty$ then we have the claimed equality, so now suppose that $\dim'_{(N,\tau)}(\H,\pi) < \infty$. Then by definition of $\dim'_{(N,\tau)}$ there exists a countable family of projections $\{p_n\}_{n\in \N} \subset \dom(\tau)$ with
    \[
        \dim'_{(N,\tau)}(\H, \pi) - \frac{1}{n} \leq \dim'_{(p_n N p_n, \tau|_{p_n N p_n})}(\pi(p_n) \H, \pi) \leq \dim'_{(N,\tau)}(\H,\pi). 
    \]
Using Remark~\ref{rem:computation_of_dimension_under_central_corner} again, if $z_n$ is the central support of $p_n$ in $N$ then the middle expression above equals $(\tau\otimes \Tr)(q(z_n\otimes 1))$ for all $n\in \N$. Setting $z := \bigvee_{n \in \N}z_n \otimes1 $, it follows that
    \[ 
       \dim'_{(N,\tau)}(\H, \pi) =  (\tau\otimes \Tr)(qz).
    \]
Assume to the contrary that $\dim'_{(N,\tau)}(\H, \pi) < (\tau\otimes \Tr)(q)=\dim_{(N,\tau)}(\H,\pi)$, in which case one has
    \[
        0<(\tau \otimes \Tr)(q(1-z)) = \sum_{a \in A} (\tau \otimes \Tr)\left[(1 \otimes e_{a,a})q(1-z)(1 \otimes e_{a,a})\right],
    \]
where $\{e_{a,b}\colon a,b\in A\}$ is a family of matrix units for $B(\K)$. So at least one term in the sum is non-zero, which implies 
for some $a\in A$ then there exists a non-zero projection $p \in \dom(\tau)$ such that $p\otimes e_{a,a} \leq (1\otimes e_{a,a})q(1-z)(1\otimes e_{a,a})$. Note that multiplying this inequality by $z_n\otimes 1$ gives $(pz_n)\otimes e_{a,a} \leq 0$, and hence $p$ is centrally orthogonal to $p_n$ for all $n\in \N$. So if $z_0$ is the central support of $p$ in $N$ then $z_0+z_n$ is the central support of $p+p_n\in \dom(\tau)$ for all $n\in \N$. Additionally, we have
    \[
        (\tau\otimes \Tr)(p\otimes e_{a,a}) = (\tau\otimes \Tr)(q(1-z) (p\otimes e_{a,a})) \leq (\tau\otimes \Tr)(q(1-z)(z_0\otimes 1)) = (\tau\otimes \Tr)(q(z_0\otimes 1))
    \]
Hence for all $n\in \N$, using Remark~\ref{rem:computation_of_dimension_under_central_corner} again as well as the above observations we have
    \begin{align*}
        \dim'_{(N,\tau)}(\H,\pi) &\geq \dim'_{((p+p_n)N(p+p_n),\tau|_{(p+p_n)N(p+p_n)})}(\pi(p+p_n)\H,\pi)\\
            &=  (\tau\otimes \Tr)(q(z_0+z_n)\otimes 1)\\
            &= (\tau\otimes \Tr)(q(z_0\otimes 1)) + (\tau\otimes \Tr)(q (z_n\otimes 1))\\
            &\geq (\tau\otimes \Tr)(p\otimes e_{a,a}) + \dim'_{(N,\tau)}(\H,\pi) - \frac1n.
    \end{align*}
Since $(\tau\otimes \Tr)(p\otimes e_{a,a})>0$, for sufficiently large $n$ the last expression above will be strictly larger than $\dim'_{(N,\tau)}(\H,\pi)$, which is a contradiction. Thus $\dim'_{(N,\tau)}(\H,\pi) = \dim_{(N,\tau)}(\H,\pi)$.
\end{proof}

Recall from Proposition~\ref{prop:basic_construction} that the basic construction $\<M,e_\vphi\>$ for $M^\vphi \subset M$ is a semifinite von Neumann algebra, and by definition any $(M,\vphi)$-module $(H,\pi)$ is an $\<M,e_\vphi\>$-module. The following result shows that one gets the same Murray--von Neumann dimension whether one views $(\H,\pi)$ as an $(M,\vphi)$-module or an $\<M,e_\vphi\>$-module provided $\<M,e_\vphi\>$ is equipped with the correct tracial weight.

\begin{thm}\label{thm:dimension_amplification_compression_formula}
Let $M$ be a von Neumann algebra equipped with a faithful normal strictly semifinite weight $\vphi$. Let $\tau$ denote the unique tracial weight on $\<M,e_\vphi\>$ satisfying $\tau(e_\vphi x e_\vphi) = \vphi(x)$ for all $x\in M_+$ (see Proposition~\ref{prop:basic_construction}). Then for any left $(M,\vphi)$-module $(\H,\pi)$ we have
    \[
        \dim_{(M,\vphi)}(\H,\pi)= \dim_{(\<M,e_\vphi\>,\tau)}(\H,\pi) = \dim_{(M^\vphi,\vphi|_{M^\vphi})}(\pi(e_\vphi)\H, \pi(\,\cdot\, e_\vphi)).
    \]
\end{thm}
\begin{proof}
Recall from Proposition~\ref{prop:basic_construction} that $e_\vphi$ has full central support in $\<M,e_\vphi\>$, and so by Proposition~\ref{prop:dimension_of_compressions} we have
    \[
        \dim_{(\<M,e_\vphi\>,\tau)}(\H,\pi) = \dim_{(e_\vphi\<M,e_\vphi\>e_\vphi,\tau|_{e_\vphi\<M,e_\vphi\>e_\vphi})}(\pi(e_\vphi)\H, \pi). 
    \]
Identifying $e_\vphi \langle M, e_\vphi\rangle e_\vphi = M^\vphi e_\vphi \cong M^\vphi$, under which the restriction of $\tau$ becomes $\vphi$, the above equals  $\dim_{(M^\vphi,\vphi|_{M^\vphi})}(\pi(e_\vphi)\H, \pi(\,\cdot\, e_\vphi))$. On the other hand, if $v\colon \H\to L^2(M,\vphi)\otimes \K$ is standard intertwiner for $(\H, \pi)$ as an $(M,\vphi)$-module, then $(e_\vphi\otimes 1)v = v\pi(e_\vphi)$ is a standard intertwiner for $(\pi(e_\vphi)\H, \pi(\,\cdot\, e_\vphi))$ as an $(M^\vphi,\vphi|_{M^\vphi})$-module since
    \[
        v\pi(e_\vphi)v^*(L^2(M,\vphi) \otimes \K) = vv^*(e_\vphi \otimes 1) (L^2(M,\vphi) \otimes \K)= vv^* (L^2(M^\vphi,\vphi) \otimes \K),
    \]
and $vv^*\in (J_\vphi M^\vphi J_\vphi)\bar\otimes B(\K) = (M^\vphi\otimes \C)'\cap B(L^2(M^\vphi,\vphi)\otimes \K)$.
Therefore, we have
    \[ 
        \dim_{(M^\vphi,\vphi|_{M^\vphi})}(\pi(e_\vphi)\H, \pi(\,\cdot\, e_\vphi)) =(\varphi \otimes \Tr)\left[ ( J_\varphi \otimes 1)vv^* (J_\varphi \otimes 1) \right]= \dim _{(M,\varphi)}(\H,\pi),
    \]
as claimed.
\end{proof}

\begin{rem}
In light of the previous theorem, we can extend the definition of $(\H,\vphi)$-modules to include arbitrary $\<M,e_\vphi\>$-modules without a Hilbert space structure (i.e. \emph{abstract} $(M,\vphi)$-modules), by using a combination of  \cite[Definition 0.4]{Luc98} and \cite[Definitions B.13 and B.17]{Pet13}.$\hfill\blacksquare$
\end{rem}

\subsection{Extremal almost periodic advantages}

In this subsection, we highlight a few advantages for our Murray--von Neumann dimension in the case of extremal almost periodic weights. Recall that when $M$ admits an extremal almost periodic weight it is necessarily a factor by \cite[Theorem 10]{Con72}, and $\Sd(\vphi)$ is a subgroup of $\R_+$ by Remark~\ref{rem:Sd_is_subgroup}. The following proposition gives necessary and sufficient conditions for a representation of $M$ on $\H$ to extend to $\<M,e_\vphi\>$. This should be compared with Proposition~\ref{prop:characterization_abstract_modules}.\ref{part:group_rep}, where one assumes there is a $\sigma^\vphi$-covariant unitary representation of $\R$ on $\H$ that has sufficiently many fixed points so as to have dense orbit under $M$. By exchanging $\R$ with a compact group in the following proposition, this latter condition on fixed points holds automatically.

\begin{prop}
Let $M$ be a von Neumann algebra equipped with an extremal almost periodic weight $\vphi$ and let $\widehat{\Sd(\vphi)}\overset{\alpha}{\curvearrowright}M$ be the point modular extension of its modular automorphism group. For any non-trivial representation $\pi_0\colon M\to B(\H)$, the following are equivalent:
    \begin{enumerate}[label=(\roman*)]
        \item there exists a representation $\pi\colon \<M,e_\vphi\>\to B(\H)$ with $\pi|_M = \pi_0$;

        \item there exists a strongly continuous unitary representation $\widehat{\Sd(\vphi)}\overset{V}{\curvearrowright} \H$ satisfying
            \[
                V_s \pi_0(x) V_s^* = \pi_0( \alpha_s(x))
            \]
        for all $s\in \widehat{\Sd(\vphi)}$ and $x\in M$.
    \end{enumerate}
\end{prop}
\begin{proof}
$(i)\Rightarrow(ii):$ For $s\in \widehat{\Sd(\vphi)}$ define
    \[
        U_s:=\sum_{\lambda \in \Sd(\vphi)} (s\mid \lambda) 1_{\{\lambda\}}(\Delta_\vphi)
    \]
so that $U_s x U_s^* = \alpha_s(x)$ for all $x\in M$ (see Section~\ref{sec:discrete_cores}). Thus letting $V_s:=\pi(U_s)$ for each $s\in \widehat{\Sd(\vphi)}$ gives the desired covariance relation with $\pi_0=\pi|_M$. We also note that $V$ is a strongly continuous representation of $\widehat{\Sd(\vphi)}$ since $U$ is strongly continuous and $\pi$ is normal.\\

\noindent$(ii)\Rightarrow(i):$ Let $\hat{\iota}\colon \R\to \widehat{\Sd(\vphi)}$ be the transpose of the inclusion map $\iota\colon \Sd(\vphi)\to \R_+$ (see Section~\ref{sec:discrete_cores}). Observe that $W_t:=V_{\hat{\iota}(t)}$ defines a strongly continuous unitary representation $\R\overset{W}{\curvearrowright} \H$ satisfying
    \[
        W_t \pi_0(x) W_t^* = \pi_0(\alpha_{\iota(t)}(x)) = \pi_0(\sigma_t^\vphi(x)).
    \]
So by Proposition~\ref{prop:characterization_abstract_modules} it suffices to show $\pi_0(M)\H^W$ is dense in $\H$. Note that $\H^W=\H^V$, so we will argue that $\pi_0(M) \H^V$ is dense.

Since $\widehat{\Sd(\vphi)}$ is compact, we obtain an orthogonal decomposition
    \[
        \H = \bigoplus_{\lambda\in \Sd(\vphi)} \H_\lambda,
    \]
where $\H_\lambda = \{ \xi\in \H\colon V_s\xi = (s|\lambda) \xi\ \forall s\in \widehat{\Sd(\vphi)}\}$ (see \cite[Theorem 4.45]{Fol16} for example). Observe that for $\xi\in \H_\lambda$ and $x\in M^{(\vphi,\mu)}$ one has
    \[
        V_s \pi_0(x) \xi = \pi_0(\alpha_s(x)) V_s\xi =  (s|\mu) \pi_0(x) (s|\lambda) \xi = (s|\mu\lambda) \pi_0(x)\xi,
    \]
so that $x\H_\lambda \subset \H_{\mu\lambda}$. 

Now, for each $\lambda\in \Sd(\vphi)$ let $\mathcal{V}_\lambda \subset M^{(\vphi,\lambda)}$ be as in Lemma~\ref{lem:spectral_projections_of_Delta}, and note that $\sum_{v\in\mathcal{V}_\lambda} vv^*=1$ since $M^\vphi$ is a factor. Using $v^*\in M^{(\vphi, \frac1\lambda)}$ and the above observation we have
    \[
        \H_\lambda = \left( \sum_{v\in \mathcal{V}_\lambda} \pi_0(vv^*) \right) \H_\lambda \subset \sum_{v\in \mathcal{V}_\lambda} \pi_0(v) \H_{1} \subset  \pi_0(M) \H_1.
    \]
Noting that $\H_1= \H^V$, we see that this implies $\pi_0(M) \H^V$ is dense.
\end{proof}

Our next goal is to prove Theorem~\ref{thmalpha:theorem_A}, which shows for certain factors $M$ that the Murray--von Neumann dimension as a function on $(M,\vphi)$-modules varies only up to a constant according to the choice of an extremal almost periodic weight. This is an extension of the case for semifinite factors and their tracial weights, and can be thought of as the analogue of a Haar measure on a locally compact group being unique up to a scaling factor. We first require the following lemma, which may also be of independent interest.

\begin{lem}\label{lem:factor_levels}
Let $M$ be a factor with separable predual equipped with an almost periodic weight $\vphi$. Then the following are equivalent:
    \begin{enumerate}[label=(\roman*)]
        \item $M^\vphi$ is a factor;
        \item $\Sd(\vphi)$ is a group and $M\rtimes_{\alpha} \widehat{\Sd(\vphi)}$ is a factor, where $\widehat{\Sd(\vphi)}\overset{\alpha}{\curvearrowright}M$ is the point modular extension of the modular automorphism group;
        \item $\<M,e_\vphi\>$ is a factor.
    \end{enumerate}
In this case one has $M\rtimes_{\alpha} \widehat{\Sd(\vphi)}\cong \<M,e_\vphi\>\cong M^\vphi\bar\otimes B(\ell^2(\Sd(\vphi)))$.
\end{lem}
\begin{proof} 
$(i)\Rightarrow(ii):$ Lemma~\ref{lem:extremal_from_Gamma} implies $\Sd(\vphi)$ is a group and $\Gamma(\alpha)=\Sd(\vphi)$, and the latter in turn implies the crossed product
    \[
        M\rtimes_{\alpha} \widehat{\Sd(\vphi)}
    \]
is a factor by \cite[Corollary XI.2.8]{Tak03}.\\

\noindent $(ii)\Rightarrow (iii):$ Appealing to \cite[Theorem 4.2]{Pas77} we see that $M\rtimes_{\alpha} \widehat{\Sd(\vphi)}$ surjects onto
    \[
        (M\cup\{\Delta_\vphi^{it}\colon t\in \R\})''.
    \]
But Proposition~\ref{prop:basic_construction} tells us the above is $\<M,e_\vphi\>$, and since the crossed product is a factor the surjection is necessarily an isomorphism. In particular, we have that $\<M,e_\vphi\>$ is a factor.\\

\noindent $(iii)\Rightarrow (i):$ This follows from $\<M,e_\vphi\> = (J_\vphi M^\vphi J_\vphi)'$.\\

\noindent For the last statement, we have already witnessed the isomorphism between the crossed product and basic construction. Since $M^\vphi\cong e_\vphi \<M,e_\vphi\> e_\vphi$ and since $\Sd(\vphi)$ is either trivial or countably infinite, we see from Proposition~\ref{prop:basic_construction} that the basic construction is an $|\Sd(\vphi)|$-fold amplification of $M^\vphi$.
\end{proof}

We now prove our first main theorem.

\begin{thm}[{Theorem~\ref{thmalpha:theorem_A}}]\label{thm:dim_scaling_constant}
Let $M$ be a factor with separable predual, and let $\vphi$ and $\psi$ be extremal almost periodic weights. Suppose that either:
    \begin{enumerate}[label=(\alph*)]
    \item $M$ is type $\mathrm{III}_\lambda$ with $0<\lambda <1$; or
    \item $M$ is full.
    \end{enumerate}
Then there exists a $*$-isomorphism $\theta\colon \<M,e_\psi\> \to \<M,e_\vphi\>$ and a constant $c_{\psi,\vphi}>0$ so that for all left $(M,\vphi)$-modules $(\H,\pi)$ one has
    \[
        \dim_{(M,\psi)}(\H, \pi\circ \theta) = c_{\psi,\vphi} \dim_{(M,\vphi)}(\H, \pi).
    \]
\end{thm}
\begin{proof}
We first claim $\Sd(\vphi)=\Sd(\psi)$. In case (a), this follows from \cite[Theorem 4.2.6]{Con73} which states that $\S(\vphi) = \S(M)= \{0\}\cup \lambda^\Z$ is a consequence of the extremality of $\vphi$. But then $\Sd(\vphi)=\lambda^\Z$ since these are the isolated points in the spectrum, and similarly for $\Sd(\psi)$. In case (b), this follows from \cite[Lemma 4.8]{Con74} which states that $\Sd(\vphi)=\Sd(M)=\Sd(\psi)$ is a consequence of their respective extremality. (Note that $\Sd(M) \neq \R_+$ since it has a separable predual and admits almost periodic weights by assumption.)

Let $\Lambda$ denote the group $\Sd(\vphi)=\Sd(\psi)$ and let $G:=\widehat{\Lambda}$ be its dual group. Note that the actions
    \[
        G\overset{\sigma^{(\vphi,\Lambda)}}{\curvearrowright}M \qquad \text{ and } \qquad G\overset{\sigma^{(\psi,\Lambda)}}{\curvearrowright}M
    \]
are cocycle equivalent: in case (a) where $G= \R/T\Z$ for $T=\frac{2\pi}{\log(1/\lambda)}$ this follows directly from the Connes cocycle derivative theorem \cite[Theorem VIII.3.3]{Tak03}; and in case (b) this follows from \cite[Lemma 4.2]{Con74}. Thus we have
    \[
        M\rtimes_{\sigma^{(\vphi,\Lambda)}} G \cong M\rtimes_{\sigma^{(\psi,\Lambda)}} G,
    \]
(see \cite[Theorem X.1.7.(ii)]{Tak03}). Using Lemma~\ref{lem:factor_levels} we obtain $*$-isomorphism $\theta\colon \<M,e_\psi\>\to \<M,e_\vphi\>$.

Finally, we use Proposition~\ref{prop:basic_construction}.\ref{part:tracial_weight} to produce a unique faithful normal semifinite tracial weight $\tau_\vphi$ on $\<M,e_\vphi\>$ satisfying $\tau_\vphi(e_\vphi x e_\vphi) = \vphi(x)$ for $x\in M_+$. If we define $\tau_\psi$ similarly, then we must have $\tau_\psi\circ \theta^{-1} = c_{\psi,\vphi} \tau_\vphi$ for some constant $c_{\psi,\vphi}>0$ because $\<M,e_\vphi\>$ is a semifinite factor. Thus if $(\H,\pi)$ is a left $(M,\vphi)$-module, then using Theorem~\ref{thm:dimension_amplification_compression_formula} twice we have
    \begin{align*}
        \dim_{(M,\psi)}(\H,\pi\circ \theta) &= \dim_{(\<M,e_\psi\>, \tau_\psi)}(\H, \pi\circ \theta)\\
            &= \dim_{(\<M,e_\vphi\>, \tau_\psi\circ \theta^{-1})}(\H,\pi)\\
            &= c_{\psi,\vphi} \dim_{(\<M,e_\vphi\>, \tau_\vphi)}(\H,\pi)= c_{\psi,\vphi} \dim_{(M,\vphi)}(\H,\pi),
    \end{align*}
as claimed.
\end{proof}

\begin{rem}
The main obstacle to proving Theorem~\ref{thm:dim_scaling_constant} in full generality is establishing cocycle equivalence of $\sigma^{(\vphi,\Lambda)}$ and $\sigma^{(\psi,\Lambda)}$ beyond cases (a) and (b). This is likely impossible to do for arbitrary $M$, but may be tractable for special cases like group von Neumann algebras.$\hfill\blacksquare$
\end{rem}


\section{Application to the Index for Subfactors}\label{sec:appliction_index_subfactors}

In this section we show that our extended Murray--von Neumann dimension can be used to estimate (and in some cases even compute exactly) the index for inclusions of type $\mathrm{III}$ factors with expectation. We begin with a few observations that will be used implicitly in the rest of the section.

Let $N\overset{\E}{\subset} M$ be an inclusion of factors and suppose $\vphi$ is a faithful normal strictly semifinite weight on $N$. Then $\vphi\circ \E$ is a faithful normal strictly semifinite weight on $M$ since $N^\vphi \subset M^{\vphi\circ \E}$. In this case
  \[
        \begin{array}{ccc}
        N & \subset & M\\
        \cup & & \cup\\
        N^{\vphi} & \subset &M^{\vphi\circ \E}
        \end{array}
    \]
forms a commuting square and $\Delta_\vphi$ can be identified with $e_N\Delta_{\vphi\circ \E}$, where $e_N$ denotes the Jones projection onto $L^2(N,\varphi)\leq L^2(M,\vphi\circ \E)$. In particular, one has $\Sd(\vphi)\subset \Sd(\vphi\circ \E)$. If one further assumes that $\vphi\circ \E$ is almost periodic then so is $\vphi$, and if both weights are extremal then $\text{Sd}(\vphi)$ is a subgroup of $\text{Sd}(\vphi\circ \E)$. We also need the following lemma, which is likely already known to experts.

\begin{lem}\label{lem:ext_ap_has_II_1_centralizer}
For a diffuse factor $M$ with an extremal almost periodic weight $\vphi$, $M^\vphi$ is a type $\mathrm{II}_1$ factor if $\vphi(1)<\infty$ and otherwise is a type $\mathrm{II}_\infty$ factor.
\end{lem}
\begin{proof}
Since $\vphi$ is extremal and strictly semifinite, we have that $M^\vphi$ is a semifinite factor. We will show that $M^\vphi$ is diffuse and hence a type $\mathrm{II}$ factor. The subtype is then determined by whether or not $1$ is a finite projection in $M^\vphi$, but this is equivalent to whether or not $\vphi(1)<\infty$ since $\vphi$ is the unique (up to scaling) faithful normal semifinite tracial weight on $M^\vphi$.
If $\Sd(\vphi)=\{1\}$ then $M^\vphi =M$ is a diffuse and we are done. Otherwise, there exists $\lambda\in \Sd(\vphi)\cap (1,+\infty)$ and hence we can find a partial isometry $v\in M^{(\vphi,\lambda)}$. If $M^\vphi$ is not diffuse, then we can find a minimal projection $p\in M^\vphi$ satisfying $p\leq vv^*$. But then
    \[
        \vphi(p) = \vphi(vv^*p) = \lambda \vphi(v^*p v) > \vphi(v^*pv);
    \]
that is, $v^*pv$ is a projection in $M^\vphi$ with trace strictly smaller than $p$, contradicting the minimality of $p$. Thus $M^\vphi$ must be diffuse.
\end{proof}

Our first step toward showing the index can be related to our extended Murray--von Neumann dimension is to establish when $L^2(M,\vphi\circ\E)$ is an $(N,\vphi)$-module. It turns out that there is a family of representations  of  $\<N,e_\vphi\>$ on $L^2(M,\vphi\circ\E)$ indexed by transversals for $\Sd(\vphi)\leq \Sd(\vphi\circ \E)$.

\begin{thm}\label{thm:subfactors_give_modules}
Let $N\overset{\E}{\subset} M$ be an inclusion of factors, and let $\vphi$ be an extremal almost periodic weight on $N$ such that $\vphi\circ \E$ is also an extremal almost periodic weight. For any transversal $C\subset \text{Sd}(\vphi\circ \E)$ of coset representatives for $\text{Sd}(\vphi)$ there exists a unique representation $\pi_C\colon \<N,e_\vphi\>\to \<M,e_{\vphi\circ\E}\>$ satisfying $\pi_C\mid_N =\text{id}$ and
    \[
        \pi_C(e_\vphi) = \sum_{\mu\in C} 1_{\{\mu\}}(\Delta_{\vphi\circ \E}).
    \]
Moreover, when $\vphi$ is a state one has
    \begin{align}\label{eqn:subfactor_dimension_formula}
        \dim_{(N,\vphi)}(L^2(M,\vphi\circ \E), \pi_C) = \left(\sum_{\mu\in C} \mu \right) [M^{\vphi\circ \E}\colon N^\vphi].
    \end{align}
\end{thm}
\begin{proof}
Let us first fix some notation. Denote $\psi:=\vphi\circ \E$ and for each $\lambda\in \text{Sd}(\psi)$ let $e_\lambda:=1_{\{\lambda\}}(\Delta_\psi)$ (so that $e_1=e_\psi$). Let $\widehat{\E}_\vphi\colon \<N,e_\vphi\>\to \widehat{N}_+$ and $\widehat{\E}_\psi\colon \<M,e_\psi\>\to \widehat{M}_+$ be the operator-valued weights dual to $\E_\vphi$ and $\E_\psi$, respectively (see Section~\ref{sec:basic_construction_PP_bases_and_index}), and consider the faithful normal semifinite weights $\widehat{\vphi}:=\vphi\circ \widehat{\E}_\vphi$ and $\widehat{\psi}:=\psi\circ \widehat{\E}_\psi$. 

Denote by $N_1$ the subalgebra of $\<M,e_\psi\>$ generated by $N\cup \{\Delta_\psi^{it}\colon t\in \R\}$. Observe that
    \[
        N_1 = (N\cup \{e_\mu\colon \mu\in C\})''.
    \]
Indeed, for $\lambda\in \text{Sd}(\vphi)$ and $\mathcal{V}_\lambda\subset N^{(\vphi,\lambda)}$ as in Lemma~\ref{lem:spectral_projections_of_Delta} one has
    \[
        e_{\lambda \mu} = \sum_{v\in \mathcal{V}_\lambda} ve_\mu v^*.
    \]
Since $C\text{Sd}(\vphi)=\text{Sd}(\psi)$, it follows that $(N\cup \{e_\mu\colon \mu\in C\})''$ contains all the spectral projections of $\Delta_\psi$ and hence $\Delta_\psi^{it}$ for all $t\in \R$. So the claimed equality holds. 

Next note for $x\in M^{(\psi,\lambda)}$ and $\mu,\nu\in \text{Sd}(\psi)$ that $e_{\mu} x e_{\nu} = \delta_{\mu=\lambda \nu} x e_\mu$. Consequently, for any $x\in N^{(\vphi,\eig)}$ and $\mu,\nu\in C$ we have $e_\mu x e_\nu = \delta_{\mu=\nu} \E_\vphi(x) e_\mu$. It follows that $N_1$ is densely spanned by $N$ and $\{Ne_\mu N\colon \mu\in C\}$. Moreover, we have
    \[
        N_1 = \overline{\text{span}}\{ Ne_\mu N\colon \mu \in C\},
    \]
since the argument above showed that $1 = \sum_{\mu\in C} \sum_{\lambda\in \Sd(\vphi)} \sum_{v\in \mathcal{V}_\lambda} v e_\mu v^*$ belongs to the latter set.

Now, since $\widehat{\E}_\vphi(e_\vphi)=1$, we have
    \[
        A:=\sqrt{\dom}(\vphi)^*e_\vphi \sqrt{\dom}(\vphi) \subset \dom(\widehat{\vphi}).
    \]
Also note that $A''=\<N,e_\vphi\>$ by Proposition~\ref{prop:basic_construction}, and $A$ is globally invariant under $\sigma_t^{\widehat{\vphi}}$ since $\sigma_t^{\widehat{\vphi}}|_{N} = \sigma_t^\vphi$ and $\sigma_t^{\widehat{\vphi}}(e_\vphi) = e_\vphi$ (see the proof of \cite[Proposition 2.2]{ILP98}). Consequently $A$ is dense in $L^2(\<N,e_\vphi\>,\widehat{\vphi})$ by \cite[Lemma 2.1]{ILP98}. Similarly, for each $\lambda\in \text{Sd}(\psi)$ we have
    \[
        \widehat{\E}_\psi (e_\lambda) = \sum_{v\in \mathcal{V}_\lambda} v \widehat{\E}_\psi(e_\psi) v^* = \sum_{v\in \mathcal{V}_\lambda} v v^* = 1 
    \]
by Lemma~\ref{lem:spectral_projections_of_Delta}. Consequently $\sqrt{\dom}(\psi)^* e_\lambda \sqrt{\dom}(\psi) \subset \dom(\widehat{\psi})$ for all $\lambda\in \text{Sd}(\psi)$. In particular,
    \[
        B:=\text{span}\{ \sqrt{\dom}(\vphi)^* e_\mu \sqrt{\dom}(\vphi)\colon \mu\in C\} \subset \dom(\widehat{\psi}).
    \]
Since $B''=N_1$, it follows that $\widehat{\psi}$ is semifinite on $N_1$. Moreover, since it is globally invariant under $\sigma_t^{\widehat{\psi}}$, \cite[Lemma 2.1]{ILP98} implies $B$ is dense in $L^2(N_1,\widehat{\psi})$.

For $a,c\in N$, $b,d\in \sqrt{\dom}(\vphi)$, and $\mu\in C$ we have
    \[
        \< ae_\mu b, ce_\mu d\>_{\widehat{\psi}} = \psi( d^* \widehat{\E}_\psi(e_\mu c^*a e_\mu) b) = \psi( d^* \E_\vphi(c^*a) b) = \vphi( d^* \E_\vphi(c^*a) b) = \< a e_\vphi b, ce_\vphi d\>_{\widehat{\vphi}}.
    \]
The density of $A\subset L^2(\<N,e_\vphi\>,\widehat{\vphi})$ implies $a e_\vphi b\mapsto a e_\mu b\in B$ extends to an isometry $W_\mu \colon L^2(\<N,e_\vphi\>,\widehat{\vphi}) \to L^2(N_1,\widehat{\psi})$ for each $\mu \in C$. Note that $W_\mu, W_\nu$ have orthogonal ranges for distinct $\mu,\nu\in C$ and so
    \[
        \sum_{\mu\in C} W_\mu W_\mu^*
    \]
is a projection. Since $B=\text{span}\{W_\mu A\colon \mu\in C\}$ lies in the image of this projection it must equal $1$. Hence
    \[
        \pi_C(x):= \sum_{\mu\in C} W_\mu x W_\mu^*
    \]
defines a normal unital $*$-homomorphism on $\<N,e_\vphi\>$. Viewing $a e_\varphi b \in L^2(\<N,e_\vphi\>, \widehat{\vphi})$, observe that for $x\in N$ we have
    \[
        W_\mu x a e_\vphi b =  x a e_\mu b = x W_\mu a e_\vphi b,
    \]
and
    \[
        W_\mu e_\vphi a e_\vphi b = W_\mu \E_\vphi(a) e_\vphi b= \E_\vphi(a) e_\mu b = e_\mu a e_\mu b = e_\mu  W_\mu a e_\vphi b.
    \]
Hence $\pi_C(x)=x$ for $x\in N$ and $\pi_C(e_\vphi) = \sum_{\mu\in C} e_\mu$. The uniqueness of $\pi_C$ is then immediate from the density of $Ne_\vphi N$ in $\<N,e_\vphi\>$.

Finally, assume $\vphi$ is a state. Note that $N^\vphi$ and $M^{\vphi\circ \E}$ are then type $\mathrm{II}_1$ factors by Lemma~\ref{lem:ext_ap_has_II_1_centralizer}. To show (\ref{eqn:subfactor_dimension_formula}) we first invoke Theorem~\ref{thm:dimension_amplification_compression_formula} and Proposition~\ref{prop:net_of_submodules} to see
    \[
        \dim_{(N,\vphi)}(L^2(M,\psi), \pi_C) = \dim_{(N^\vphi, \vphi)} \left(\pi_C(e_\vphi) L^2(M,\psi), \pi_C(\,\cdot\, e_\vphi)\right) = \sum_{\mu\in C} \dim_{(N^\vphi, \vphi)}\left( e_\mu L^2(M,\psi)\right).
    \]
Now, for each $\mu\in C$ let $\mathcal{V}_{\mu^{-1}}\subset M^{(\psi, \mu^{-1})}$ be as in Lemma~\ref{lem:spectral_projections_of_Delta}. It follows that
    \[
        e_\mu = J_\psi e_{\mu^{-1}} J_\psi = \sum_{v\in \mathcal{V}_{\mu^{-1}}} (J_\psi v J_\psi) e_\psi (J_\psi v^* J_\psi).
    \]
Thus
    \begin{align*}
        \dim_{(N^\vphi, \vphi)}\left( e_\mu L^2(M,\psi)\right) &= \sum_{v\in V_{\mu^{-1}}} \dim_{(N^\vphi, \vphi)} \left(  (J_\psi v J_\psi) e_\psi (J_\psi v^* J_\psi) L^2(M,\psi)\right)\\
            &= \sum_{v\in V_{\mu^{-1}}} \dim_{(N^\vphi, \vphi)} \left(  e_\psi (J_\psi v^* J_\psi) (J_\psi v J_\psi) e_\psi  L^2(M,\psi)\right) \\
            &= \sum_{v\in V_{\mu^{-1}}} \dim_{(N^\vphi, \vphi)} \left(  J_\psi v^*v J_\psi L^2(M^\psi,\psi) \right) \\
            &=  \sum_{v\in V_{\mu^{-1}}} \psi(v^*v) \dim_{(N^\vphi, \vphi)} L^2(M^\psi,\psi)\\
            &= \sum_{v\in V_{\mu^{-1}}} \mu \psi(vv^*) [M^\psi\colon N^\vphi] = \mu [M^\psi\colon N^\vphi],
    \end{align*}
where we have used $\psi(1)=1$ in the last equality. Summing over $\mu\in C$ then yields (\ref{eqn:subfactor_dimension_formula}).
\end{proof}

\begin{rem}\label{rem:subfactor_dimension_can_usually_be_made_finite}
In the setting of Theorem~\ref{thm:subfactors_give_modules}, when $\vphi$ is a state, observe that so long as $\text{Sd}(\vphi)\neq \{1\}$ one can always choose a transversal $C$ so that
    \[
        \sum_{\mu\in C} \mu < \infty. 
    \]
Indeed, in this case $\text{Sd}(\vphi)$ is a non-trivial subgroup of $\R_+$ and hence all of its cosets in $\text{Sd}(\vphi\circ \E)$ have arbitrarily small representatives. Hence if $\vphi$ is a state, $\text{Sd}(\vphi)\neq \{1\}$, and $[M^{\vphi\circ \E}\colon N^\vphi] <\infty$, then there always exists a choice of transversal so that the quantity in (\ref{eqn:subfactor_dimension_formula}) is finite.$\hfill\blacksquare$
\end{rem}

By adapting the proof of \cite[Proposition 5.21]{Bur17}, which relies on a result communicated by Jesse Peterson, we show that extremal almost periodic inclusions of factors always admit Pimsner--Popa orthogonal bases and moreover they can be chosen to consist of eigenoperators. This observation will be useful in the proof of our second main theorem, particularly in the infinite index case.

\begin{prop}\label{prop:existence_PP_basis}
Let $N\overset{\E}{\subset} M$ be an inclusion of diffuse factors with separable preduals, and let $\vphi$ be an extremal almost periodic state on $N$ such that $\vphi\circ \E$ is also an extremal almost periodic state. For any transversal $C\subset \text{Sd}(\vphi\circ \E)$ of coset representatives for $\text{Sd}(\vphi)$, there exists Pimsner--Popa orthogonal basis for $N\overset{\E}{\subset}M$ consisting of elements from $\bigcup_{\mu\in C} M^{(\vphi \circ \E,\mu)}$.
\end{prop}
\begin{proof}
Denote $\psi:=\vphi\circ \E$. Fix a transversal $C\subset \text{Sd}(\psi)$, and for each $\mu\in C$ let $\mathcal{V}_\mu\subset M^{(\psi, \mu)}$ be as in Lemma~\ref{lem:spectral_projections_of_Delta}. Since $N^\vphi$ and $M^\psi$ are type $\mathrm{II}_1$ factors by Lemma~\ref{lem:ext_ap_has_II_1_centralizer}, any projection in $M^\psi$ is equivalent via unitary conjugation to a projection in $N^\vphi$, and hence we may assume $v^*v\in N^\vphi$ for all $v\in \mathcal{V}_\mu$ by precomposing with a suitable family of unitaries in $M^\psi$.

We first claim that for each $v\in \bigcup_{\mu\in C} \mathcal{V}_\mu$ there exists a family $\{m_{v,i}\in M^\psi \colon i\in I_v\}$ satisfying
    \begin{align}\label{eqn:PP_subbasis_spanning}
        \sum_{i\in I_v} m_{v,i} e_{N^\vphi} m_{v,i}^* = v^*v,
    \end{align}
and
    \begin{align}\label{eqn:PP_subbasis_orthogonal}
        \E(m_{v,i}^* v^*v m_{v,j}) = \delta_{i=j} p_{v,i},
    \end{align}
for some projection $p_{v,i}\in N^\vphi$. Toward constructing this family, let $T\colon \<M^\psi, e_{N^\vphi}\> \to \widehat{M^\psi}_+$ be the unique faithful normal semifinite operator-valued weight satisfying $T(e_{N^\vphi}) = 1$. Note that $\tau:=\psi\circ T$ is then the unique (up to scaling) faithful normal semifinite tracial weight on $\<M^\psi, e_{N^\vphi}\>$. 

If one has $v^*v\in \dom(T)$ (for example, if $[M^\psi\colon N^\vphi]<\infty$), then letting $\{w_{v,i}\in \<M^\psi, e_{N^\vphi}\>\colon i\in I_v\}$ be a family of partial isometries satisfying $w_{v,i}^* w_{v,i}\leq e_{N^\vphi}$ and $\sum_i w_{v,i} w_{v,i}^* = v^*v$ we set $m_{v,i}:= T(w_{v,i})$. Then $m_{v,i} e_{N^\vphi}= w_{v,i}$ (see \cite[Section 1.1.5]{Pop95book}), so that (\ref{eqn:PP_subbasis_spanning}) is immediate and
    \[
        \E(m_{v,i}^* v^*v m_{v,j}) e_{N^\vphi} =  e_{N^\vphi} m_{v,i}^* v^*v m_{v,j} e_{N^\vphi} = w_{v,i}^* v^*v w_{v,j} = w_{v,i}^* w_{v,j} = \delta_{i=j} w_{v,i}^* w_{v,i}.
    \]
Applying $T$ to the above yields (\ref{eqn:PP_subbasis_orthogonal}) since $N^\psi\ni x\mapsto x e_{N^\vphi} \in e_{N^\vphi}\< M^\psi, e_{N^\vphi}\> e_{N^\vphi}$ is an isomorphism. 

Otherwise, $v^*v\not\in \dom(T)$ implies $\tau(v^*v)=\psi(T(v^*v))=\infty$ and we instead follow the strategy outlined in \cite{Bur17}. By \cite[Lemma 5.19]{Bur17} there exists $x\in \dom(T)_+$ with $\ker{x}=\{0\}$, and since $T$ is $M^\psi$-bimodular we have $v^*v x v^*v\in \dom(T)_+$ with $\ker(v^*v x v^*v)= \ker(v^*v)$. (Note that the existence of such an $x$ is where our assumption of separable preduals is used; see the proof of \cite[Lemma 5.19]{Bur17}.) Letting $q_n:= 1_{(\frac1n, \infty)}( v^*v x v^*v)$ for each $n\in \N$, it follows that $q_n\leq n v^*v x v^*v$ so that $q_n\in \dom(T)$, and as $n\to\infty$ we have that $q_n$ increases to:
    \[
        1 - [\ker{(v^*v x v^*v)}] = 1- [\ker{(v^*v)}] = 1- (1- v^*v) = v^*v.
    \]
Setting $p_n:= q_{n+1} - q_n$, we have $p_n \in\dom(T)_+$ for all $n\in \N$ with $\sum_n p_n = v^*v$. Note that the normality of $\tau$ implies
    \[
        \sum_{n\in \N} \tau(p_n) = \tau(v^*v) = \infty,
    \]
and of course $\tau(p_n)= \psi( T(p_n)) <\infty$. Hence we may assume $\tau(p_n)=1$ by adding finite sets of projections and taking subprojections since these operations preserve $\dom(T)$. As $\tau(e_{N^\vphi}) = \psi(1) =1$, it follows that there exist partial isometries $w_{v,n}\in \<M^\psi, e_{N^\psi}\>$ satisfying $w_{v,n}w_{v,n}^* =p_n$ and $w_{v,n}^* w_{v,n} = e_{N^\vphi}$ for each $n\in \N$. Then $w_{v,n}^* = e_{N^\vphi} w_{v,n}^* p_n \in \sqrt{\dom}(T)$ implies by \cite[Lemma 5.18]{Bur17} that $m_{v,n}:=T(w_{v,n} e_{N^\vphi})\in M^\psi$ satisfies $w_{v,n} = w_{v,n} e_{N^\vphi}= m_{v,n} e_{N^\vphi}$. Then (\ref{eqn:PP_subbasis_spanning}) follows from
    \[
        \sum_{n\in \N} m_{v,n} e_{N^\vphi} m_{v,n}^* = \sum_{n\in \N} w_{v,n} w_{v,n}^* = \sum_{n\in \N} p_n = v^*v.
    \]
We also have
    \[
        \E(m_{v,n}^* v^*v m_{v,k}) e_{N^\vphi} = e_{N^\vphi}m_{v,n}^* v^*v m_{v,k}e_{N^\vphi} = w_{v,n}^* v^*v w_{v,k} = w_{v,n}^* w_{v,k} = \delta_{n=k} w_{v,n}^* w_{v,n},
    \]
so that (\ref{eqn:PP_subbasis_orthogonal}) follows by applying $T$ to the above.

With our first claim in hand, we next claim that
    \[
        \left\{ v m_{v,i}\colon v\in \bigcup_{\mu\in C} \mathcal{V}_\mu,\ i\in I_v\right\}
    \]
is the desired Pimsner--Popa orthogonal basis for $N\overset{\E}{\subset}M$. First observe $m_{v,i}\in M^\psi$ implies $v m_{v,i} \in M^{(\psi, \mu)}$ for $v\in M^{(\psi,\mu)}$. Next, to see that the system is orthogonal, note that for distinct $v,w\in \bigcup_{\mu\in C} \mathcal{V}_\mu$ we have $v^*w = 0$ if both belong to $\mathcal{V}_\mu$, and otherwise $v^*w \in M^{(\psi, \lambda)}$ for some $\lambda\notin \text{Sd}(\vphi)$ so that $\E(m_{v,i}^* v^* w m_{w,j}) =0$. For a single $v\in \bigcup_{\mu\in C} \mathcal{V}_\mu$, the orthogonality follows from (\ref{eqn:PP_subbasis_orthogonal}). It remains to show that the set is spanning.

Note that $N\cap M^\psi= N^\varphi$ implies that $e_N e_\psi = e_{N^\varphi}$ when $L^2(M^\psi,\psi)$ is identified as a subspace of $L^2(M,\psi)$. Thus in the context of $\<M,e_N, e_\psi\>$, (\ref{eqn:PP_subbasis_spanning}) becomes
    \begin{align}\label{eqn:PP_spanning_1}
        \sum_{i\in I_v} m_{v,i} e_{N}e_{\psi} m_{v,i}^* = v^*v e_\psi.
    \end{align}
Next, for each $\lambda\in \text{Sd}(\vphi)$ fix some $\mathcal{V}_\lambda\in N^{(\vphi,\lambda)}$ as in Lemma~\ref{lem:spectral_projections_of_Delta}. Recall that this same lemma gives
    \[
        \sum_{\lambda\in \text{Sd}(\vphi)} \sum_{w\in \mathcal{V}_\lambda} w e_\vphi w^*= \sum_{\lambda\in \text{Sd}(\vphi)} 1_{\lambda}(\Delta_\vphi) =1.
    \]
The above equality is occurring as operators on $L^2(N,\varphi)$, and when this is identified as a subspace of $L^2(M,\psi)$ one has $e_\varphi = e_N e_\psi$ and the right-hand side becomes $e_N$. Thus $\sum_\lambda \sum_w w (e_N e_\psi) w^* = e_N$, which further yields
    \begin{align}\label{eqn:PP_spanning_2}
        \sum_{\lambda\in \text{Sd}(\vphi)} \sum_{w\in \mathcal{V}_\lambda} (J_\psi w J_\psi) e_{N} e_\psi (J_\psi w^* J_\psi) = e_N,
    \end{align}
where we have used that $e_N$ and $e_\psi$ both commute with $J_\psi$. Finally, we compute
    \begin{align*}
        \sum_{\mu\in C} \sum_{v\in \mathcal{V}_\mu }    \sum_{i\in I_v} v m_{v,i} e_N m_{v,i}^* v^* &=  
 \sum_{\mu\in C} \sum_{v\in \mathcal{V}_\mu } \sum_{i\in I_v} \sum_{\lambda \in \text{Sd}(\vphi)} \sum_{w\in \mathcal{V}_\lambda} v m_{v,i} (J_\psi w J_\psi)e_N e_\psi (J_\psi w^* J_\psi )m_{v,i}^* v^*\\
    &= \sum_{\mu\in C} \sum_{v\in \mathcal{V}_\mu } \sum_{i\in I_v} \sum_{\lambda \in \text{Sd}(\vphi)} \sum_{w\in \mathcal{V}_\lambda} (J_\psi w J_\psi )v m_{v,i}e_N e_\psi m_{v,i}^* v^* (J_\psi w^* J_\psi )\\
    &=  \sum_{\mu\in C} \sum_{v\in \mathcal{V}_\mu }  \sum_{\lambda \in \text{Sd}(\vphi)} \sum_{w\in \mathcal{V}_\lambda} (J_\psi w J_\psi )v v^* v e_\psi v^*   (J_\psi w^* J_\psi ) \\
    &= \sum_{\mu\in C} \sum_{\lambda \in\text{Sd}(\vphi)} \sum_{w\in \mathcal{V}_\lambda} (J_\psi w J_\psi ) 1_{\{\mu\}}(\Delta_\psi) (J_\psi w^* J_\psi )\\
    &= \sum_{\mu\in C} \sum_{\lambda \in \text{Sd}(\vphi)} 1_{\mu\lambda}(\Delta_\psi)= 1,
    \end{align*}
where we have used (\ref{eqn:PP_spanning_2}) in the first equality, (\ref{eqn:PP_spanning_1}) in the third equality, and (\ref{eqn:spectral_compression_equality}) from Lemma~\ref{lem:spectral_projections_of_Delta} in the fourth and fifth equalities.
\end{proof}

We now prove our second main theorem.

\begin{thm}[{Theorem~\ref{thmalpha:theorem_B}}]\label{thm:dimension_vs_index}
Let $N\overset{\E}{\subset} M$ be an inclusion of diffuse factors with separable preduals, and let $\vphi$ be an extremal almost periodic state on $N$ such that $\vphi\circ \E$ is also an extremal almost periodic state. Let $C\subset \text{Sd}(\vphi\circ \E)$ be a transversal of coset representatives for $\text{Sd}(\vphi)$ and let $\pi_C \colon \<N,e_\vphi\> \to \<M,e_{\vphi\circ \E}\>$ be the unique representation satisfying $\pi_C|_N=\text{id}$ and
    \[
        \pi_C(e_\vphi) = \sum_{\mu\in C} 1_{\{\mu\}}(\Delta_{\vphi\circ \E})
    \]
(see Theorem~\ref{thm:subfactors_give_modules}). Then one has
    \begin{align}\label{ineq:dimension_index_relation}
        \inf(C) \Ind{\E} \leq \dim_{(N,\vphi)}(L^2(M,\vphi\circ \E),\pi_C) \leq \sup(C) \Ind{\E}.
    \end{align}
Moreover, the following are equivalent:
    \begin{enumerate}[label=(\roman*)]
    \item $\Ind{\E}<\infty$;
    \item $[\text{Sd}(\vphi\circ \E)\colon \text{Sd}(\vphi)]<\infty$ and $\dim_{(N,\vphi)}(L^2(M,\vphi\circ \E), \pi_C)<\infty$;
    \item $[\text{Sd}(\vphi\circ \E)\colon \text{Sd}(\vphi)]<\infty$ and $[M^{\vphi\circ \E}\colon N^\vphi]<\infty$.
    \end{enumerate}
\end{thm}
\begin{proof}
As usual, denote $\psi:=\vphi\circ \E$. Also denote $C^{-1}:=\{\mu^{-1}\colon \mu\in C\}$, which also forms a transversal of coset representatives. From Proposition~\ref{prop:existence_PP_basis}, there exists a Pimsner--Popa orthogonal basis $\{m_i\colon i\in I\} \in \bigcup_{\mu\in C^{-1}} M^{(\psi,\mu)}$ for $N\overset{\E}{\subset} M$. Thus we can define an isometry by
    \begin{align*}
        v\colon L^2(M,\psi) &\to L^2(N,\vphi)\otimes \ell^2(I)\\
            \xi &\mapsto \sum_{i\in I} (e_N J_\psi m_i^* J_\psi \xi)\otimes \delta_i,
    \end{align*}
where we have identified $L^2(N,\vphi)\cong e_N L^2(M,\psi)$. (Note that under this identification one has $e_\vphi = e_N e_\psi$.)  For $x\in N$ one has $vx = (x\otimes 1) v$ by definition of $v$. Additionally, if $m_i\in M^{(\psi, \nu)}$ for $\nu\in C^{-1}$ then one has
    \begin{align*}
        e_N J_\psi m_i^* J_\psi \pi_C(e_\vphi) &=\sum_{\mu \in C} e_N J_\psi m_i^* J_\psi 1_{\{\mu\}}(\Delta_\psi)\\
        &= \sum_{\mu\in C} e_N 1_{\{\nu\mu\}}(\Delta_\psi) J_\psi m_i^* J_\psi \\
        &= e_N e_\psi J_\psi m_i^* J_\psi = e_\vphi e_N J_\psi m_i^* J_\psi.
    \end{align*}
It follows that $v\pi_C(e_\vphi) = (e_\vphi\otimes 1)v$, and hence $v$ is a standard intertwiner for $(L^2(M,\psi),\pi_C)$ as a left $(N,\vphi)$-module.

Now, the $(i,j)$-entry of $vv^*$ is
    \[
        (e_N J_\psi m_i^* J_\psi) (J_\psi m_j J_\psi e_N) = J_\psi\E(m_i^* m_j)J_\psi e_N = \delta_{i=j} J_\psi\E(m_i^* m_i)J_\psi e_N,
    \]
which equals $\delta_{i=j} J_\vphi \E(m_i^* m_i) J_\vphi$ under the identification $L^2(N,\vphi)\cong e_N L^2(M,\psi)$. Letting $\nu(i) \in C^{-1}$ be such that $m_i\in M^{(\psi, \nu(i))}$ for each $i\in I$, we compute
    \begin{align*}
        \dim_{(N,\vphi)}(L^2(M,\psi), \pi_C) &= (\vphi\otimes \Tr)\left( (J_\vphi\otimes 1) vv^* (J_\vphi\otimes 1) \right)\\
            &= \sum_{i\in I} \vphi(\E(m_i^*m_i))\\
            &= \sum_{i\in I} \nu(i)^{-1} \psi(m_im_i^*)\\
            &\leq \sup(C) \psi\left( \sum_{i\in I} m_i m_i^*\right) = \sup(C) \Ind{\E}.
    \end{align*}
The lower bound follows similarly.

Toward proving the final statement, for each $\mu\in C$ let $\mathcal{V}_\mu\subset M^{(\psi, \mu)}$ be as in Lemma~\ref{lem:spectral_projections_of_Delta}. As in the proof of Proposition~\ref{prop:existence_PP_basis}, we may assume $v^*v\in N^\vphi$ for all $v\in \mathcal{V}_\mu$. Consequently, $\bigcup_{\mu\in C} \mathcal{V}_\lambda$ forms a Pimsner--Popa orthogonal system of $N\overset{\E}{\subset} M$. Therefore
    \[
        \Ind{\E} \geq \sum_{\mu\in C} \sum_{v\in\mathcal{V}_\mu}  vv^* = \sum_{\mu\in C} 1 = |C|,
    \]
Thus if $\Ind{\E}<\infty$, then $|C|<\infty$ and using (\ref{ineq:dimension_index_relation}) one has $\dim_{(N,\vphi)}(L^2(M,\psi),\pi_C)<\infty$. This gives (i)$\Rightarrow $(ii) and the converse direction follows immediately from (\ref{ineq:dimension_index_relation}). Finally, the equivalence (ii)$\Leftrightarrow$(iii) follows from (\ref{eqn:subfactor_dimension_formula}) in Theorem~\ref{thm:subfactors_give_modules} because the sum over $C$ in this formula is necessarily finite.
\end{proof}

\begin{rem}
An alternative proof of the inequality (\ref{ineq:dimension_index_relation}) in Theorem~\ref{thm:dimension_vs_index} can be given as roughly as follows. Denote $\Lambda:=\Sd(\vphi)$ and $\Gamma:=\Sd(\vphi\circ \E)$ and let $H$ and $G$ denote their respective dual groups. Let $H\overset{\alpha}{\curvearrowright} N$ and $G\overset{\beta}{\curvearrowright} M$ be the point modular extensions of the modular automorphism groups of $\vphi$ and $\vphi\circ \E$, respectively, and recall from Lemma~\ref{lem:factor_levels} that we have
    \[
        N\rtimes_\alpha H\cong N^\vphi\bar\otimes B(\ell^2 \Lambda) \qquad \text{ and } \qquad M\rtimes_\beta G \cong M^{\vphi\circ \E} \bar\otimes B(\ell^2 \Gamma).
    \]
By Takesaki duality, one then has
    \[
        N\bar\otimes B(\ell^2 \Lambda) \cong \left(N^\vphi\bar\otimes B(\ell^2 \Lambda)\right) \rtimes_{\widehat{\alpha}} \Lambda \qquad \text{ and } \qquad M\bar\otimes B(\ell^2 \Gamma) \cong \left(M^{\vphi\circ \E} \bar\otimes B(\ell^2 \Gamma)\right)\rtimes_{\widehat{\beta}} \Gamma,
    \]
where $\widehat{\alpha}$ and $\widehat{\beta}$ are the actions dual to $\alpha$ and $\beta$, respectively. Assume $\Lambda$ (and hence $\Gamma$) is non-trivial so that $B(\ell^2\Lambda)\cong  B(\ell^2\Gamma) \cong B(\ell^2)$. Then the inclusion
    \[
        N\bar\otimes B(\ell^2) \overset{\E\otimes 1}{\subset} M\bar\otimes B(\ell^2)
    \]
can be identified with an inclusion of the form
    \[
        \left(N^\vphi\bar\otimes B(\ell^2)\right) \rtimes_{\widehat{\alpha}} \Lambda \subset \left(M^{\vphi\circ \E} \bar\otimes B(\ell^2)\right)\rtimes_{\widehat{\beta}} \Gamma,
    \]
(note that $\widehat{\beta}$ extends $\widehat{\alpha}$ since $\Lambda\leq \Gamma$). Since $\E\otimes 1$ is the unique $(\vphi\otimes \Tr)$-preserving faithful normal conditional expectation onto $N\bar\otimes B(\ell^2)$, there is a unique faithful normal conditional expectation $\mathcal{F}$ corresponding to the above inclusion and one has
    \[
        \Ind{\E} = \Ind{(\E\otimes 1)} = \Ind{\mathcal{F}}.
    \]
Now, since the above inclusion admits $\left(M^{\vphi\circ\E}\bar\otimes B(\ell^2)\right) \rtimes_{\widehat{\alpha}} \Lambda$ as an intermediate algebra, it follows that
    \[
        \Ind{\E}=\Ind{\mathcal{F}} = [\Gamma\colon \Lambda][M^{\vphi\circ \E}\colon N^\vphi].
    \]
Consequently, (\ref{ineq:dimension_index_relation}) follows from (\ref{eqn:subfactor_dimension_formula}) and the inequality
    \[
        \inf(C)[\Gamma\colon \Lambda] \leq \sum_{\mu\in C} \mu \leq \sup(C)[\Gamma\colon \Lambda],
    \]
for a transversal $C\subset \Gamma$ of coset representatives for $\Lambda$. If $\Lambda$ is trivial, then slight modifications to the above argument can also be made to cover the two cases when $\Gamma$ is trivial or not.$\hfill\blacksquare$
\end{rem}

\begin{rem}\label{rem:tight_bounds}
Under the hypotheses of Theorem~\ref{thm:dimension_vs_index}, suppose $\Sd(\vphi)$ is dense in $\R_+$. Then all of its cosets in $\Sd(\vphi\circ \E)$ will be dense, and so given $\epsilon>0$ one can choose a transversal satisfying $C\subset (1-\epsilon, 1+\epsilon)$. Consequently, (\ref{ineq:dimension_index_relation}) implies
    \[
        (1-\epsilon) \Ind{\E} \leq \dim_{(N,\vphi)}(L^2(M,\varphi\circ \E),\pi_c) \leq (1+\epsilon) \Ind{\E}.
    \]
However, unless $C$ is a finite set then $\dim_{(N,\vphi)}(L^2(M,\varphi\circ \E),\pi_c)$ (and hence $\Ind{\E}$) will be infinite by (\ref{eqn:subfactor_dimension_formula}) in Theorem~\ref{thm:subfactors_give_modules}.$\hfill\blacksquare$
\end{rem}

The inequality (\ref{ineq:dimension_index_relation}) in Theorem~\ref{thm:dimension_vs_index} reduces to an equality when $\Sd(\vphi)=\Sd(\vphi\circ\E)$ and one chooses the transversal $C=\{1\}$. In this case, the representation $\pi_C$ from Theorem~\ref{thm:subfactors_give_modules} allows one to identify $\<N,e_\vphi\>$ with $\<N,e_{\vphi\circ\E}\>$ inside $\<M,e_{\vphi\circ \E}\>$, and moreover the index and Murray--von Neumann dimension both agree with $[M^{\vphi\circ \E}\colon N^\vphi]$ by (\ref{eqn:subfactor_dimension_formula}) in Theorem~\ref{thm:subfactors_give_modules}. In the next few corollaries, we consider a few situations where these conditions follow from natural assumptions.

\begin{cor}\label{cor:index_with_same_Sd_inv}
Let $N\overset{\E}{\subset} M$ be an inclusion of diffuse factors with separable preduals. Suppose $\vphi$ is an almost periodic state  on $M$ satisfying $\vphi\circ \E = \vphi$ and $\Sd(N)=\Sd(\vphi) = \Sd(M)$. Then $\vphi$ and $\vphi|_N$ are both extremal and one has
    \[
        \Ind{\E} = \dim_{(N,\vphi|_N)} L^2(M,\vphi)=[M^{\vphi}\colon N^{\vphi|_N}],
    \]
where $\<N,e_{\vphi|_N}\>$ is represented as $\<N,e_\vphi\> \subset B(L^2(M,\vphi))$.
\end{cor}
\begin{proof}
The extremality of $\vphi$ follows immediately from Lemma~\ref{lem:extremal_from_Gamma}. Noting that $\Sd(N)\subset \Sd(\vphi|_N) \subset \Sd(\vphi) = \Sd(N)$, we see that $\vphi|_N$ is also extremal by Lemma~\ref{lem:extremal_from_Gamma} and that $\Sd(\vphi|_N) =\Sd(\vphi)$. Hence the claimed equalities follow from the discussion preceding the statement of the corollary.
\end{proof}

\begin{cor}
Let $N\overset{\E}{\subset} M$ be an inclusion of full factors with separable preduals satisfying $\Sd(N)=\Sd(M)$. If $\vphi$ is an extremal almost periodic state on $M$ satisfying $\vphi\circ \E = \vphi$, then $\vphi|_N$ is extremal and
    \[
        \Ind(\E) = \dim_{(N,\vphi|_N)}L^2(M, \vphi) = [M^\vphi\colon N^{\vphi|_N}],
    \]
where $\<N,e_{\vphi|_N}\>$ is represented as $\<N,e_\vphi\>\subset B(L^2(M,\vphi))$.
\end{cor}
\begin{proof}
This is a special case of Corollary~\ref{cor:index_with_same_Sd_inv} by \cite[Lemma 4.8]{Con74}.   
\end{proof}

\begin{cor}
Let $N\overset{\E}{\subset} M$ be an inclusion of type $\mathrm{III}_\lambda$ factors for $0<\lambda <1$ with separable preduals. For any extremal almost periodic state $\vphi$ on $M$ satisfying $\vphi\circ \E = \vphi$, one has that $\vphi|_N$ is extremal and
    \[
        \Ind(\E) = \dim_{(N,\vphi|_N)}L^2(M, \vphi) = [M^\vphi\colon N^{\vphi|_N}],
    \]
where $\<N,e_{\vphi|_N}\>$ is represented as $\<N,e_\vphi\>\subset B(L^2(M,\vphi))$.
\end{cor}
\begin{proof}
Set $t_0:=\frac{2\pi}{\log{\lambda}}$. By \cite[Theorem 4.2.6]{Con73}, the extremality of $\vphi$ implies $\sigma_{t_0}^\vphi=1$ and $\Sd(\vphi)=\lambda^\Z$. It follows that $\sigma_{t_0}^{\vphi|_N}=1$  so applying \cite[Theorem 4.2.6]{Con73} again gives that $\vphi|_N$ is extremal with $\Sd(\vphi|_N)=\lambda^\Z$. The equality then follows by the discussion preceding Corollary~\ref{cor:index_with_same_Sd_inv}.
\end{proof}

Recall that given a finite index inclusion of $\mathrm{II}_1$ factors $A\leq B$ and left $B$-module $(\H,\pi)$, one has that $(\H,\pi|_A)$ is a left $A$-module and
    \[
        \dim_{A}(\H,\pi|_A) = [B\colon A] \dim_{B}(\H,\pi).
    \]
The following result gives the analogue of this in the case of extremal almost periodic inclusions under the assumption $\Sd(\vphi)=\Sd(\vphi\circ \E)$. Note that the special case of $N\subset N\otimes M_n(\C)$ was already observed in the discussion following Proposition~\ref{prop:amplification_formula}.

\begin{prop}\label{prop:finite_index_amplification_formula}
Let $N\overset{\E}{\subset} M$ be an inclusion of diffuse factors with separable preduals, and let $\vphi$ be an extremal almost periodic state on $N$ such that $\vphi\circ \E$ is also an extremal almost periodic state with $\Sd(\vphi)=\Sd(\vphi\circ \E)$. Identify $\<N,e_\vphi\> \cong \<N,e_{\vphi\circ \E}\> \subset \<M,e_{\vphi\circ \E}\>$. If $(\H,\pi)$ is a left $(M,\vphi\circ \E)$-module, then
    \[
        \dim_{(N,\vphi)}(\H,\pi|_{\<N,e_\vphi\>}) = \Ind(\E) \dim_{(M,\vphi\circ\E)}(\H,\pi).
    \]
\end{prop}
\begin{proof}
As usual, denote $\psi:=\vphi\circ \E$. Let $v\colon \H\to L^2(M,\psi)\otimes \K$ be a standard intertwiner of $(\H,\pi)$ as a left $(M,\psi)$-module. We use Proposition~\ref{prop:existence_PP_basis} to find a Pimsner--Popa orthogonal basis $\{m_i\colon i\in I\}\subset M^{\psi}$ for $N\overset{\E}{\subset} M$. As in the proof of Theorem~\ref{thm:dimension_vs_index},
    \begin{align*}
        w\colon L^2(M,\psi) &\to L^2(N,\vphi)\otimes \ell^2(I)\\
            \xi &\mapsto \sum_{i\in I} (e_N J_\psi m_i^* J_\psi \xi)\otimes \delta_i\
    \end{align*}
defines a standard intertwiner for $L^2(M,\psi)$ as a left $(N,\vphi)$-module. Consequently, $(w\otimes 1)v$ is a standard intertwiner for $(\H,\pi)$ as a left $(N,\vphi)$-module. Fix a family of matrix units $\{e_{a,b}\colon a,b\in A\}$ for $B(\K)$ and let $x_{a,b}\in M^\psi$ be such that $(1\otimes e_{a,a}) vv^* (1\otimes e_{b,b}) = (J_\psi x_{a,b} J_\psi)\otimes e_{a,b}$. That is, $vv^*$ can be identified with the matrix over $A$ whose $(a,b)$-entry is $J_\psi x_{a,b} J_\psi$. Then $(w\otimes 1)vv^*(w^*\otimes 1)$ can be identified with the matrix (over $I\times A$) whose $((i,a),(j,b))$-entry is
    \[
        (e_N J_\psi m_i^* J_\psi)(J_\psi x_{a,b} J_\psi)(J_\psi m_j J_\psi e_N )= J_\psi \E(m_i^* x_{a,b} m_j) J_\psi e_N,
    \]
which equals $J_\vphi \E(m_i^* x_{a,b} m_j) J_\vphi$ under the identification $L^2(N,\vphi)\cong e_N L^2(M,\psi)$. Consequently, one has
    \begin{align*}
        \dim_{(N,\vphi)}(\H,\pi|_{\<N,e_\vphi\>}) &= \vphi\otimes \Tr\otimes \Tr( (J_\vphi\otimes 1\otimes 1)(w\otimes 1)vv^*(w^*\otimes 1) (J_\vphi\otimes 1\otimes 1))\\
            &= \sum_{(i,a)\in I\times A} \varphi( \E( m_i^* x_{a,a} m_i))\\
            &= \sum_{(i,a)\in I\times A} \psi( x_{a,a} m_i m_i^*)\\
            &= \Ind(\E)\sum_{a\in A} \psi(x_{a,a}) \\
            &= \Ind(\E)(\psi\otimes \Tr)( (J_\psi\otimes 1) vv^* (J_\psi \otimes 1)) = \Ind(\E) \dim_{(M,\psi)}(\H,\pi),
    \end{align*}
where in the third equality we have used $m_i \in M^\psi$, and in the fourth equality we have used $\sum_{i\in I} m_i m_i^*=\Ind(\E)$.
\end{proof}

\begin{rem}
Under the same hypotheses as in Proposition~\ref{prop:finite_index_amplification_formula} \emph{except} allowing for $\Sd(\vphi)\neq \Sd(\vphi\circ \E)$, one obtains
    \[
        \inf(C) \Ind{\E} \dim_{(M,\vphi\circ \E)}(\H,\pi) \leq \dim_{(N,\vphi)}(\H,\pi\circ \pi_C) \leq \sup(C) \Ind(\E) \dim_{(M,\vphi\circ \E)}(\H,\pi),
    \]
by the same argument as in the proof of (\ref{ineq:dimension_index_relation}) in Theorem~\ref{thm:dimension_vs_index}. In fact, (\ref{ineq:dimension_index_relation}) is then a special case of the above inequality for $\H=L^2(M,\vphi\circ \E)$. $\hfill\blacksquare$
\end{rem}

We conclude this section with an example showing that finite Murray--von Neumann dimension \emph{does not} imply finite index, even though the converse is true by Theorem~\ref{thm:dimension_vs_index}. Note that, by Theorem~\ref{thm:dimension_vs_index}, such an example requires $\text{Sd}(\vphi)\leq \text{Sd}(\vphi\circ \E)$ to be an infinite index inclusion of groups.

\begin{ex}
Let $0<\lambda,\gamma<1$ be such that $\log(\lambda)$ and $\log(\gamma)$ are linearly independent over $\Q$. Then $\Gamma:=\<\lambda, \gamma\> \leq \R_+$ is dense and $\Lambda:=\<\lambda\>$ is an infinite index subgroup of $\Gamma$. Let $(A,\tau)$ be a type $\mathrm{II}_\infty$ factor with faithful normal semifinite tracial weight that admits a $\tau$-scaling action $\Gamma\overset{\alpha}{\curvearrowright} A$: $\tau\circ \alpha_\mu = \mu \tau$ for all $\mu\in \Gamma$. (For example, by Takesaki duality one can take $A\cong (R_\lambda \bar\otimes R_\gamma)\rtimes_\beta \widehat{\Gamma}$, where $R_\lambda$ and $R_\gamma$ are the hyperfinite type $\mathrm{III}_\lambda$ and $\mathrm{III}_\gamma$ factors and the action $\beta$ by the dual group $\widehat{\Gamma}$ is point modular extension of the modular automorphism group of the tensor product of Powers' states $\phi_\lambda$ and $\phi_\gamma$.) If $\Phi$ is the faithful normal semifinite weight on $A\rtimes_\alpha \Gamma$ dual to $\tau$, then $\Phi|_A = \tau$ by the discreteness of $\Gamma$ and one has $(A\rtimes_\alpha \Gamma)^\Phi = A$. In particular, $\Phi$ is extremal and strictly semifinite. Furthermore, $\Phi$ is almost periodic by Lemma~\ref{lem:ap_iff_eigenoperators_generate} because the action being $\tau$-scaling implies $\C[\Gamma] \subset (A\rtimes_\alpha \Gamma)^{(\Phi,\eig)}$. Similarly, $\Phi$ is extremal almost periodic on $A\rtimes_{\alpha|_\Lambda} \Lambda$, which we view as a subalgebra of $A\rtimes_\alpha \Gamma$.

Now, fix $p\in A$ with $\Phi(p)=\tau(p)=1$ and define
    \begin{align*}
        M &:= p(A\rtimes_\alpha \Gamma)p\\
        N &:= p(A\rtimes_{\alpha|_\Lambda} \Lambda)p.
    \end{align*}
Then the state $\vphi:= \Phi(p\cdot p)$ is extremal almost periodic on both $M$ and $N$ by Remark~\ref{rem:almost_periodic_corners}. Since $N^{(\vphi|_N,\eig)}$ is dense in $N$, it follows that $N$ is globally invariant under the modular automorphism group of $\vphi$ and hence  there exists a faithful normal $\vphi$-preserving conditional expectation $\E\colon M\to N$ by \cite[Theorem IX.4.2]{Tak03}.  Additionally, one has
    \[
        M^\vphi = N^\vphi = pAp,
    \]
so that $[M^\vphi\colon N^\vphi]=1$. We also have $\text{Sd}(\vphi)=\Gamma$ and $\text{Sd}(\vphi|_N)=\Lambda$, and so by Remark~\ref{rem:subfactor_dimension_can_usually_be_made_finite} there exists a choice of transversal $C\subset \Gamma$ of coset representatives for $\Lambda$ so that
    \[
        \dim_{(N,\vphi)}(L^2(M,\vphi),\pi_C)<\infty,
    \]
where $\pi_C$ is as in Theorem~\ref{thm:subfactors_give_modules}. On the other hand, $[\Gamma \colon \Lambda]=\infty$ implies $\Ind{\E}=\infty$ by Theorem~\ref{thm:dimension_vs_index}.$\hfill\blacksquare$
\end{ex}

\section{Constructing Extremal Almost Periodic Inclusions}\label{sec:constructing_extremal_almost_periodic_inclusions}

We conclude this article by showing that the almost periodic hypothesis appearing in the results of the previous section is mild in the sense that many inclusions can be reduced to such inclusions (see Proposition~\ref{prop:reducing_to_ap_inclusions}). We also present a construction that builds a family of extremal almost periodic inclusions starting with the initial data of an irreducible inclusion (see Example~\ref{ex:building_extremal_ap_inclusions}).

\begin{prop}\label{prop:existence_of_almost_periodic_part}
Let $M$ be a von Neumann algebra equipped with a faithful normal strictly semifinite weight $\vphi$. Then the restriction of $\vphi$ to $(M^{(\vphi,\eig)})''$ is almost periodic and there exists a unique $\vphi$-preserving faithful normal conditional expectation from $M$ onto this subalgebra.
\end{prop}
\begin{proof}
Denote $N:=(M^{(\vphi,\eig)})''$. Then $\vphi|_N$ is faithful and normal, and even strictly semifinite since $M^\vphi = N^{\vphi|_N}$. Consequently, $\vphi$ is almost periodic by Lemma~\ref{lem:ap_iff_eigenoperators_generate}. Finally, \cite[Theorem IX.4.2]{Tak03} yields the desired conditional expectation since $M^{(\vphi,\eig)}$ and hence $N$ are globally invariant under the modular automorphism group of $\vphi$.
\end{proof}

\begin{defi}
Given a von Neumann algebra $M$ equipped with a faithful normal strictly semifinite weight $\vphi$, we call
    \[
        (M^{(\vphi,\text{ap})},\vphi_{\text{ap}}) :=  \left(  (M^{(\vphi,\eig)})'', \vphi|_{(M^{(\vphi,\eig)})''} \right)
    \]
the \textbf{almost periodic part} of $(M,\vphi)$. We denote by $\E_{\text{ap}}\colon M\to M^{(\vphi,\text{ap})}$ the unique $\vphi$-preserving faithful normal conditional expectation onto the almost periodic part.$\hfill\blacksquare$
\end{defi}

Lemma~\ref{lem:ap_iff_eigenoperators_generate} implies $(M^{(\vphi,\text{ap})}, \vphi_{\text{ap}}) = (M,\vphi)$ if and only if $(M,\vphi)$ is almost periodic. Additionally, one always has $(M^{(\vphi,\text{ap})})^{\vphi_{\text{ap}}} = M^\vphi$, so that $\vphi_{\text{ap}}$ is extremal if and only if $\vphi$ is extremal. We suspect that in this case and when $M$ is also a factor, one has that $\Ind{\E_{\text{ap}}}$ is either $1$ or $+\infty$, but we are unable to prove it at this time. The following example illustrates this dichotomy.

\begin{ex}
Let $\R\overset{U}{\curvearrowright}\H_\R$ be an orthogonal action on a real Hilbert space, and let $(\Gamma(U)'',\vphi_U)$ be the associated free Araki--Woods factor equipped with its free quasi-free state which is known to be extremal (see \cite{Shl97}). By Zorn's Lemma, there exists a maximal invariant subspace $\K_\R\leq \H_\R$ so that the restricted action $\R\overset{V}{\curvearrowright}\K_\R$ is almost periodic. Then the almost periodic part of $(\Gamma(U)'',\vphi_U)$ is given by
    \[
        (\Gamma(V)'', \varphi_V), 
    \]
which is freely complemented by a diffuse factor $N$ when $\K_\R \neq \H_\R$ (see \cite{BC16}). For the conditional expectation $\E_{\text{ap}}\colon \Gamma(U)'' \to \Gamma(V)''$, if $U=V$ we of course have $\Ind{\E_{\text{ap}}}=1$. But otherwise the diffuse free complement $N$ forces $\Ind{\E_{\text{ap}}}=+\infty$; indeed, note that $\E_{\text{ap}}(x) = \varphi_U(x)$ for all $x\in N$ so that the diffuseness of $N$ gives an infinite index through its probabilistic definition (see \cite[Definition 1.1.1]{Pop95book}).
$\hfill\blacksquare$
\end{ex}

Recall from Remark~\ref{rem:strictly_semifinite_corners} that faithful normal strictly semifinite weights restrict to weights with the same properties on corners under projections from the centralizer.

\begin{lem}\label{lem:almost_periodic_parts_of_corners}
Let $M$ be a von Neumann algebra equipped with a faithful normal strictly semifinite weight $\vphi$. For $p\in M^\vphi$ one has
    \[
        (pMp)^{(\vphi|_{pMp}, \text{ap})} = pM^{(\vphi,\text{ap})} p
    \]
and
    \[
        (\vphi|_{pMp})_{\text{ap}} = (\vphi_{\text{ap}})|_{pM^{(\vphi,\text{ap})} p}.
    \]
\end{lem}
\begin{proof}
First note $pM^{(\vphi,\eig)}p \subset (pMp)^{(\vphi|_{pMp}, \eig)}$. The reverse inclusion also holds as a consequence of the KMS condition (and the fact that $p\in M^\vphi$). The equality of these $*$-algebras extends to the desired equality of von Neumann algebras, which in turn gives the equality of weights after deciphering the notation.
\end{proof}

Recall for an inclusion $N\overset{\E}{\subset} M$ that if $\vphi$ is a faithful normal strictly semifinite weight on $N$, then $\vphi\circ \E$ is a faithful normal strictly semifinite weight on $M$ since $N^\vphi\circ M^{\vphi\circ \E}$.

\begin{prop}\label{prop:reducing_to_ap_inclusions}
Let $N\overset{\E}{\subset} M$ be an inclusion of von Neumann algebras. If $\vphi$ is a faithful normal strictly semifinite weight on $N$, then 
    \[
        N^{(\vphi, \text{ap})} \subset M^{(\vphi\circ\E, \text{ap})},
    \]
with faithful normal conditional expectation given by the restriction of $\E$ and $(\vphi\circ \E)_{\text{ap}} = \vphi_{\text{ap}}\circ \E$.
\end{prop}
\begin{proof}
For all $t\in \R$ we have $\sigma_t^{\vphi\circ \E}|_N = \sigma_t^\vphi$, which yields the inclusion of the almost periodic parts. The fact that $\E$ commutes with $\sigma^{\vphi\circ \E}$ implies $\E|_{M^{(\vphi\circ \E,\text{ap})}}$ is valued in $N^{(\vphi,\text{ap})}$ as well as the equality of weights.
\end{proof}

For an inclusion  $N\overset{\E}{\subset} M$ of von Neumann algebras with separable preduals, suppose $\vphi$ is an extremal faithful normal strictly semifinite weight on $N$ such $\vphi\circ \E$ is also extremal. Then the previous proposition implies the inclusion $N^{(\vphi,\text{ap})} \subset M^{(\vphi\circ \E,\text{ap})}$ is still with expectation, and moreover $\vphi_{\text{ap}}$ and $\vphi_{\text{ap}}\circ \E$ are extremal almost periodic weights. In particular, $N^{(\vphi,\text{ap})}$ and $ M^{(\vphi\circ \E,\text{ap})}$ are factors. If we compress this inclusion by any non-zero projection $q\in N^\vphi\cap \dom(\vphi|_{N^\vphi})$ then by Lemma~\ref{lem:almost_periodic_parts_of_corners} and Remark~\ref{rem:almost_periodic_corners} we will obtain an inclusion $N_q:=q N^{(\vphi,\text{ap})} q \subset qM^{(\vphi\circ \E,\text{ap})}q=:M_q$ of factors with expectation $\E_q:=\E|_{M_q}$ and an extremal almost periodic state $\vphi_q:=\frac{1}{\vphi(q)} \vphi_{\text{ap}}|_{N_q}$ whose composition with $\E_q$ is an extremal almost periodic state. The index of $\E_q$ can be estimated using Theorem~\ref{thm:dimension_vs_index} if $N_q$ is diffuse. Otherwise, $N_q$ is a type $\mathrm{I}$ factor and so $\vphi_q$ is a tracial state by Remark~\ref{rem:extremal_ap_on_semifinite_is_tracial}. If $M_q$ is not a type $\mathrm{I}$ factor then necessarily $\Ind{\E_q}=\infty$, and otherwise $\vphi_q\circ \E_q$ is also tracial by Remark~\ref{rem:extremal_ap_on_semifinite_is_tracial} and hence $\Ind{\E_q}= \dim_{(N_q, \vphi_q)}L^2(M_q, \vphi_q\circ \E_q)$.

In our last example, we demonstrate how to construct extremal almost periodic inclusions as in Theorems~\ref{thm:subfactors_give_modules} and \ref{thm:dimension_vs_index} from irreducible inclusions.

\begin{ex}\label{ex:building_extremal_ap_inclusions}
Let $N_0 \overset{\E}{\subset} M$ be an irreducible inclusion of factors with separable preduals and $N_0$ semifinite. Let $\tau$ be a faithful normal semifinite tracial weight on $N_0$ and denote $\vphi:=\tau\circ \E$. Then $N_0\subset M^\vphi$ and so it follows that $\vphi$ is strictly semifinite. By replacing $(M,\vphi)$ with its almost periodic part, we assume $\vphi$ is almost periodic. In fact, it will be extremal almost periodic:
    \[
        (M^\vphi)'\cap M^\vphi \leq (N_0)'\cap M = \C.
    \]
Fix a $*$-subalgebra $A \subset M^{(\vphi,\eig)}$ and define $N:=(N_0\cup A)''$. Then $N$ is globally invariant under the modular automorphism group of $\vphi$ and $\vphi|_{N}$ is semifinite (since $N_0\leq N$), so there exists a faithful normal $\vphi$-preserving conditional expectation $\E\colon M\to N$ by \cite[Theorem IX.4.2]{Tak03}. Additionally, $\vphi|_N$ is extremal almost periodic since $N_0\cup A$ has $\sigma$-weakly dense span in $N$ and 
    \[
        (N^\vphi)'\cap N^\vphi \leq (N_0)'\cap M = \C.
    \]
Note that $\text{Sd}(\vphi|_N)$ is the subgroup of $\text{Sd}(\vphi)$ generated by the $\mu$ such that $A\cap M^{(\vphi,\mu)}$ is non-empty. So assuming $\text{Sd}(\vphi)$ has infinite index subgroups (equivalently, it is dense in $\R_+$), one can use this construction to build infinite index inclusions, even when $[M^{\vphi}\colon N^\vphi]<\infty$. On the other hand, if $[M^\vphi\colon N_0]<\infty$ then $[M^\vphi\colon N^\vphi]<\infty$ for any choice of $A$. In this case, Theorem~\ref{thm:dimension_vs_index} tells us $\Ind{\E}<\infty$ holds if and only if $A$ is chosen to give $[\text{Sd}(\vphi)\colon \text{Sd}(\vphi|_N)]<\infty$.
$\hfill\blacksquare$
\end{ex}

\begin{rem}
Suppose $N_0=M^\vphi$ in the previous example and as usual let $G\overset{\alpha}{\curvearrowright} M$ be the point modular extension of the modular automorphism group of $\vphi$. Then \cite[Theorem 3.15]{ILP98} implies every $N$ in the above example is of the form $M^H=\{x\in M\colon \sigma_s^{(\vphi,\Sd(\vphi))}(x)=x \ s\in H \}$ for some closed subgroup $H\leq G$.$\hfill\blacksquare$
\end{rem}


\bibliographystyle{amsalpha}
\bibliography{references}

\providecommand{\bysame}{\leavevmode\hbox to3em{\hrulefill}\thinspace}
\providecommand{\MR}{\relax\ifhmode\unskip\space\fi MR }
\providecommand{\MRhref}[2]{%
  \href{http://www.ams.org/mathscinet-getitem?mr=#1}{#2}
}
\providecommand{\href}[2]{#2}
\begin{thebibliography}{KPV15}

\bibitem[Ati76]{Ati76}
M.~F. Atiyah, \emph{Elliptic operators, discrete groups and von {N}eumann
  algebras}, Colloque ``{A}nalyse et {T}opologie'' en l'{H}onneur de {H}enri
  {C}artan ({O}rsay, 1974), Ast\'{e}risque, vol. No. 32-33, Soc. Math. France,
  Paris, 1976, pp.~43--72. \MR{420729}

\bibitem[BH16]{BC16}
R\'{e}mi Boutonnet and Cyril Houdayer, \emph{Structure of modular invariant
  subalgebras in free {A}raki-{W}oods factors}, Anal. PDE \textbf{9} (2016),
  no.~8, 1989--1998. \MR{3599523}

\bibitem[Bur17]{Bur17}
Michael Burns, \emph{Subfactors, planar algebras and rotations}, Proceedings of
  the 2014 {M}aui and 2015 {Q}inhuangdao conferences in honour of {V}aughan
  {F}. {R}. {J}ones' 60th birthday, Proc. Centre Math. Appl. Austral. Nat.
  Univ., vol.~46, Austral. Nat. Univ., Canberra, 2017, pp.~25--114.
  \MR{3635668}

\bibitem[CG86]{CG86}
Jeff Cheeger and Mikhael Gromov, \emph{{$L_2$}-cohomology and group
  cohomology}, Topology \textbf{25} (1986), no.~2, 189--215. \MR{837621}

\bibitem[Com71]{Com71}
Fran\c{c}ois Combes, \emph{Poids et esp\'{e}rances conditionnelles dans les
  alg\`ebres de von {N}eumann}, Bull. Soc. Math. France \textbf{99} (1971),
  73--112. \MR{288589}

\bibitem[Con72]{Con72}
Alain Connes, \emph{\'{E}tats presque p\'eriodiques sur une alg\`ebre de von
  {N}eumann}, C. R. Acad. Sci. Paris S\'er. A-B \textbf{274} (1972),
  A1402--A1405. \MR{0295092}

\bibitem[Con73]{Con73}
\bysame, \emph{Une classification des facteurs de type {${\mathrm{III}}$}},
  Ann. Sci. \'Ecole Norm. Sup. (4) \textbf{6} (1973), 133--252. \MR{0341115 (49
  \#5865)}

\bibitem[Con74]{Con74}
\bysame, \emph{Almost periodic states and factors of type
  {${\mathrm{III}}_{1}$}}, J. Functional Analysis \textbf{16} (1974), 415--445.
  \MR{0358374 (50 \#10840)}

\bibitem[CS05]{CS05}
Alain Connes and Dimitri Shlyakhtenko, \emph{{$L^2$}-homology for von {N}eumann
  algebras}, J. Reine Angew. Math. \textbf{586} (2005), 125--168. \MR{2180603
  (2007b:46104)}

\bibitem[CT77]{CT77}
Alain Connes and Masamichi Takesaki, \emph{The flow of weights on factors of
  type {${\rm III}$}}, Tohoku Math. J. (2) \textbf{29} (1977), no.~4, 473--575.
  \MR{480760}

\bibitem[Dyk95]{Dyk95}
Kenneth Dykema, \emph{Crossed product decompositions of a purely infinite von
  {N}eumann algebra with faithful, almost periodic weight}, Indiana Univ. Math.
  J. \textbf{44} (1995), no.~2, 433--450. \MR{1355406}

\bibitem[Dyk97]{Dyk97}
Kenneth~J. Dykema, \emph{Free products of finite-dimensional and other von
  {N}eumann algebras with respect to non-tracial states}, Free probability
  theory ({W}aterloo, {ON}, 1995), Fields Inst. Commun., vol.~12, Amer. Math.
  Soc., Providence, RI, 1997, pp.~41--88. \MR{1426835}

\bibitem[Fol16]{Fol16}
Gerald~B. Folland, \emph{A course in abstract harmonic analysis}, second ed.,
  Textbooks in Mathematics, CRC Press, Boca Raton, FL, 2016. \MR{3444405}

\bibitem[Haa77]{Haa77}
Uffe Haagerup, \emph{An example of a weight with type {III} centralizer}, Proc.
  Amer. Math. Soc. \textbf{62} (1977), no.~2, 278--280. \MR{430801}

\bibitem[Haa79]{Haa79c}
\bysame, \emph{{$L^{p}$}-spaces associated with an arbitrary von {N}eumann
  algebra}, Alg\`ebres d'op\'{e}rateurs et leurs applications en physique
  math\'{e}matique ({P}roc. {C}olloq., {M}arseille, 1977), Colloq. Internat.
  CNRS, vol. 274, CNRS, Paris, 1979, pp.~175--184. \MR{560633}

\bibitem[Hay14]{BenThesis}
Benjamin Hayes, \emph{Extended von neumann dimension for representations of
  groups and equivalence relations}, Doctoral dissertation, University of
  California, Los Angeles, 2014.

\bibitem[ILP98]{ILP98}
Masaki Izumi, Roberto Longo, and Sorin Popa, \emph{A {G}alois correspondence
  for compact groups of automorphisms of von {N}eumann algebras with a
  generalization to {K}ac algebras}, J. Funct. Anal. \textbf{155} (1998),
  no.~1, 25--63. \MR{1622812}

\bibitem[Jon83]{Jon83}
Vaughan F.~R. Jones, \emph{Index for subfactors}, Invent. Math. \textbf{72}
  (1983), no.~1, 1--25. \MR{696688 (84d:46097)}

\bibitem[Kos86]{Kos86}
Hideki Kosaki, \emph{Extension of {J}ones' theory on index to arbitrary
  factors}, J. Funct. Anal. \textbf{66} (1986), no.~1, 123--140. \MR{829381}

\bibitem[Kos98]{Kos98}
\bysame, \emph{Type {III} factors and index theory}, Lecture Notes Series,
  vol.~43, Seoul National University, Research Institute of Mathematics, Global
  Analysis Research Center, Seoul, 1998. \MR{1662525}

\bibitem[KPV15]{KPV15}
David Kyed, Henrik~Densing Petersen, and Stefaan Vaes, \emph{{$L^2$}-{B}etti
  numbers of locally compact groups and their cross section equivalence
  relations}, Trans. Amer. Math. Soc. \textbf{367} (2015), no.~7, 4917--4956.
  \MR{3335405}

\bibitem[Lon89]{Lon89}
Roberto Longo, \emph{Index of subfactors and statistics of quantum fields.
  {I}}, Comm. Math. Phys. \textbf{126} (1989), no.~2, 217--247. \MR{1027496}

\bibitem[L{\"u}c98]{Luc98}
Wolfgang L{\"u}ck, \emph{Dimension theory of arbitrary modules over finite von
  {N}eumann algebras and {$L^2$}-{B}etti numbers. {I}. {F}oundations}, J. Reine
  Angew. Math. \textbf{495} (1998), 135--162. \MR{1603853}

\bibitem[MvN36]{MvN36}
F.~J. Murray and J.~{v}on Neumann, \emph{On rings of operators}, Ann. of Math.
  (2) \textbf{37} (1936), no.~1, 116--229. \MR{1503275}

\bibitem[Pas77]{Pas77}
W.~L. Paschke, \emph{Integrable group actions on von {N}eumann algebras}, Math.
  Scand. \textbf{40} (1977), no.~2, 234--248. \MR{461160}

\bibitem[Pet13]{Pet13}
Henrik~Densing Petersen, \emph{{$L^2$}-{B}etti numbers of locally compact
  groups}, C. R. Math. Acad. Sci. Paris \textbf{351} (2013), no.~9-10,
  339--342. \MR{3072156}

\bibitem[Pop95]{Pop95book}
Sorin Popa, \emph{Classification of subfactors and their endomorphisms}, CBMS
  Regional Conference Series in Mathematics, vol.~86, Conference Board of the
  Mathematical Sciences, Washington, DC; by the American Mathematical Society,
  Providence, RI, 1995. \MR{1339767}

\bibitem[PP86]{PP86}
Mihai Pimsner and Sorin Popa, \emph{Entropy and index for subfactors}, Ann.
  Sci. \'{E}cole Norm. Sup. (4) \textbf{19} (1986), no.~1, 57--106. \MR{860811}

\bibitem[Shl97]{Shl97}
Dimitri Shlyakhtenko, \emph{Free quasi-free states}, Pacific J. Math.
  \textbf{177} (1997), no.~2, 329--368. \MR{1444786 (98b:46086)}

\bibitem[Sin77]{Sin77}
I.~M. Singer, \emph{Some remarks on operator theory and index theory},
  {$K$}-theory and operator algebras ({P}roc. {C}onf., {U}niv. {G}eorgia,
  {A}thens, {G}a., 1975), Lecture Notes in Math., vol. Vol. 575, Springer,
  Berlin-New York, 1977, pp.~128--138. \MR{467848}

\bibitem[Tak73]{Tak73}
Masamichi Takesaki, \emph{Duality for crossed products and the structure of von
  {N}eumann algebras of type {III}}, Acta Math. \textbf{131} (1973), 249--310.
  \MR{438149}

\bibitem[Tak02]{Tak02}
M.~Takesaki, \emph{Theory of operator algebras. {I}}, Encyclopaedia of
  Mathematical Sciences, vol. 124, Springer-Verlag, Berlin, 2002, Reprint of
  the first (1979) edition, Operator Algebras and Non-commutative Geometry, 5.
  \MR{1873025 (2002m:46083)}

\bibitem[Tak03]{Tak03}
Masamichi Takesaki, \emph{Theory of operator algebras. {II}}, Encyclopaedia of
  Mathematical Sciences, vol. 125, Springer-Verlag, Berlin, 2003, Operator
  Algebras and Non-commutative Geometry, 6. \MR{1943006 (2004g:46079)}

\end{thebibliography}

\end{document}